\documentclass{svjour3}
\usepackage{amsmath,amsfonts}
\usepackage{amssymb,amsopn,amscd}
\usepackage{bbm}
\usepackage{graphicx}
\usepackage{epsfig}
\usepackage{color}
\usepackage{subfigure}
\usepackage{bm}
\usepackage[normalem]{ulem}

\newcommand{\bard}[1]{{#1}}

\newcommand{\review}[1]{{{{#1}}}}

\graphicspath{{Figures}}

\newcommand{\der}[2]{ \frac{\text{d} #1}{\text{d} #2} }  
\newcommand{\R}{\mathbbm{R}}
\newcommand{\Vect}{\mathrm{Vect}}

\begin{document}
	\title{Finite-size and correlation-induced effects in Mean-field Dynamics}
	\author{Jonathan Touboul\footnotemark[1] \footnotemark[2] \and G. Bard Ermentrout\footnotemark[2]} 
	
	\titlerunning{Finite-size and Correlation Effects}        
	
	\institute{J. Touboul \at
	              \footnotemark[1]NeuroMathComp Laboratory \\
								INRIA/ ENS Paris\\
								23 Avenue d'Italie\\
								75013 Paris
	              Tel.: +33-1 39 63 57 10\\
	              Fax:  +33-4 92 38 78 45\\
	              \email{jonathan.touboul@sophia.inria.fr}           
	           \and
	           \footnotemark[2] G. Bard Ermentrout and Jonathan Touboul \at
	              Department of Mathematics,\\
	 							University of Pittsburgh, \\
								Pittsburgh PA, \\
								USA.
	}

	\date{Received: date / Accepted: date}

	\maketitle
		%
	
\begin{abstract}
The brain's activity is characterized by the interaction of a very large number of neurons that are strongly affected by noise. However, signals  often arise at macroscopic scales integrating the effect of many neurons into a reliable pattern of activity. In order to study such large neuronal assemblies, one is often led to derive mean-field limits summarizing the effect of the interaction of a large number of neurons into an effective signal. Classical mean-field approaches consider the evolution of a deterministic variable, the mean activity, thus neglecting the stochastic nature of neural behavior. In this article, we build upon \bard{two recent approaches}  that include correlations and higher order moments in mean-field equations, and study how these stochastic effects influence the solutions of the mean-field equations, both in the limit of an infinite number of neurons and for large yet finite networks. \review{We introduce a new model, the infinite model, which arises from both equations by a rescaling of the variables and,  which is invertible for finite-size networks, and hence, provides equivalent equations to those previously derived models.  The study of this model allows us to understand qualitative behavior of such large-scale networks.} We show that, though the solutions of the deterministic mean-field equation constitute uncorrelated solutions of the new mean-field equations, the stability properties of limit cycles are modified by the presence of correlations, and additional non-trivial behaviors including periodic orbits appear when there were none in the mean field.  The origin of all these behaviors is then explored in finite-size networks where interesting mesoscopic scale effects appear. This study leads us to show that the infinite-size system appears as a singular limit of the network equations, and for any finite network, the system will differ from the infinite system. 
\end{abstract}


%

\section*{Introduction}
Cortical activity manifests highly complex behaviors which is often strongly characterized by the presence of noise.  Reliable responses to specific stimuli often arise at the level of population assemblies (cortical areas or cortical columns) featuring a very large number of neuronal cells presenting a highly nonlinear behavior and that are interconnected in a very intricate fashion. Understanding the global behavior of large-scale neural assemblies has been a great endeavor in the past decades. Most models describing the emergent behavior arising from the interaction of neurons in large-scale networks have relied on continuum limits ever since the seminal work of Wilson and Cowan and Amari \cite{amari:72,amari:77,wilson-cowan:72,wilson-cowan:73}. Such models tend to represent the activity of the network through a macroscopic variable, the population-averaged firing rate, that is generally assumed to be deterministic. Many analytical and numerical results and properties have been derived from these equations and related to cortical phenomena, for instance for the problem of  spatio-temporal pattern formation in spatially extended models (see e.g. \cite{coombes-owen:05,ermentrout:98,ermentrout-cowan:79,laing-troy-etal:02}). 

This approach tends to implicitly make the assumption that correlations in the activity do not modify the behavior of the system in large populations. However, the stochasticity of the system, and in particular the presence of correlations in the activity may considerably affect the dynamics in such nonlinear systems, and these can be even more important as the network size increases. These models therefore make the assumption that averaging effects counterbalance the prominent noisy aspect of \emph{in vivo} firing, a feature that can dramatically affect finite-sized network activity, and, thus, that the emergent behavior of a cortical column is deterministic.

However, increasingly many researchers now believe that the different intrinsic or extrinsic noise sources are part of the neuronal signal, and rather than a pure disturbing effect related to the intrinsically noisy biological substrate of the neural system, they suggest that noise conveys information that can be an important principle of brain function \cite{rolls-deco:10}. At the level of a single cell, various studies have shown that the firing statistics are highly stochastic with probability distributions close to Poisson distributions \cite{softky-koch:93}, and several computational studies confirmed the stochastic nature of single-cells firings \cite{brunel-latham:03,plesser:99,touboul-faugeras:07b,touboul-faugeras:08}.  How variability at the single neuron level affects dynamics of cortical networks is less established so far. Theoretically, the interaction of a large number of neurons that fire stochastic spike trains can naturally produce correlations in the firing activity of the population considered. For instance power-laws in the scaling of avalanche-size distributions have been studied both via models and experiments~\cite{beggs-plenz:04,benayoun-cowan:10,levina-etal:09,touboul-destexhe:10}. In  these regimes the randomness plays a central role. 

A different approach has been to study regimes where the activity is uncorrelated. A number of computational studies on the integrate-and-fire neuron showed that under certain conditions neurons in large assemblies end up firing asynchronously, producing null correlations \cite{abbott-van-vreeswijk:93,amit-brunel:97,brunel-hakim:99}. In these regimes, the correlations in the firing activity decrease towards zero in the limit where the number of neurons tends to infinity. The emergent global activity of the population in this limit is deterministic, and evolves according to a mean-field firing rate equation. However, these states only exist in the limit where the number of neurons is infinite, and such asynchronous states do not necessarily exist in cortical areas or in computational models, and even when they exist, they are not necessarily stable. This raises the question of how the finiteness of the number of neurons impacts the behavior of asynchronous states, and how to study non-asynchronous behaviors. The study of finite-size effects in the case of asynchronous states is generally not reduced to the study of mean firing rates and can include higher order moments of firing activity \cite{mattia-del-giudice:02,elboustani-destexhe:09,brunel:00}. In order to go beyond asynchronous states and take into account the stochastic nature of the firing and how this activity scales as the network size increases, different approaches have been developed, such as the population density method and related approaches \cite{cai-tao-etal:04}. Most of these approaches involve expansions in terms of the moments of the resulting random variable, and the moment hierarchy needs to be truncated which is a quite hard task that can raise a number of technical issues (see e.g.\cite{ly-tranchina:07}). A recent approach overcomes this truncation difficulty by considering the mean-field behavior as a stochastic process and therefore treating the rigorous limit equation of the interaction of many stochastic neurons as a fixed point equation in the space of stochastic processes, making use of the powerful tools of stochastic limit theorems and large deviation techniques (see \cite{faugeras-touboul-etal:09}). 

{In order to study the effect of the stochastic nature of the firing in large network, many authors strived to introduce randomness in a tractable form. The models we will be interested in are based on the definition of a Markov chain governing the firing dynamics of the neurons in the network, where the transition probability satisfies a differential equation called the \emph{master equation}. Seminal works of the application of such modeling for neuroscience date back to the early 90s with the work of Ohira and Cowan \cite{ohira-cowan:93}, and have been progressively developed by different authors and today constitute a model of choice~\cite{buice-cowan:07,buice-cowan:10,bressloff:09,elboustani-destexhe:09}. The present manuscript is based on the analysis of} two recently developed stochastic models of the customary Wilson and Cowan equations.The global resulting activity is then readily derived from the activity of all neurons in the network. The two models of interest here were instantiated specifically in order to be compared to standard Wilson and Cowan rate equations, and provide tractable stochastic extensions of the purely deterministic standard neural-field models \cite{buice-cowan:10,bressloff:09}. The two approaches differ {in the choice of the transition probability and on the variables considered, choices mainly based on the fact that} they address two different regimes of network activity: Buice, Cowan and Chow address the highly correlated Poisson-like activity and Bressloff, the asynchronous states. {The two approaches also differ in the way the equations are treated: Buice, Cowan and collaborators in \cite{buice-cowan:07,buice-cowan:10} derive moment equations by carrying out a loop expansion of a path integral representation of their master equation whose truncation is justified under the assumption that the network operates in a Poisson-like regime. Bressloff rigorously justifies his moment truncation by exhibiting an expansion in powers of a small parameter $1/N$ where $N$ corresponds to the size of the network.} These different points of view lead the authors to define different variables and they obtain closely related yet different equations on the variables they define, derive equations for moments from the master equation which they then simplify and truncate to finite order. 

The present paper is organized as follows: in the first section we describe the models we use. Both simplified models introduced involve coupling between mean firing rate and correlations between firing times, with finite-size correlations. In the second section, we are interested in the non-zero correlation solutions of the mean-field equations, and, in particular, focus on the qualitative similarities and differences between the solutions of these stochastic equations compared to the customary Wilson and Cowan behaviors. In the third section, we turn to study the finite-size effects, namely the qualitative differences between the solutions of mean-field limits and Wilson and Cowan system on one hand, and the behaviors of large but finite networks on the other hand. These finite-size solutions are compared back to the initial Markovian model and the differential equations they produced.

\section{Markovian firing models}
In this section, we introduce and describe the master equation formalism used by Buice, Cowan and collaborators \cite{buice-cowan:07,buice-cowan:10} (hereafter BCC)  and Bressloff \cite{bressloff:09} and extend these approaches to take into account different populations. Both models have been previously derived, so we quickly present their derivation and the resulting equations in this section, as our main purpose is to \review{derive from these equations our model of interest, the \emph{infinite-size model}, and to mathematically analyze it}.{ Details on the derivation of these equations and on their theoretical justification can be found in~\cite{buice-cowan:07,buice-cowan:10,bressloff:09}.}

Throughout manuscript, we will be interested in a network structured into $M$ homogeneous neural populations. Each population $i \in \{1,\ldots,M\}$ is composed of $N_i$ neurons, and the total number of neurons is denoted by $N$ (i.e. $N=\sum_{i=1}^{M} N_i$). We are interested in the behavior of the network when the number of neurons $N$ tends to infinity. In this limit, we assume that the proportion of neurons belonging to each population are non-trivial, i.e. $\lim_{N\to \infty} N_i/N = \lambda_i \in (0,1)$. Each neuron of population $i$ interacts with all the neurons of the network, with an efficiency denoted $W_{ij}$. These weights are assumed to scale as $1/N_j$, so that a given activity at the level of population $i$ will produce a bounded activity at the level of population $j$. We therefore define the effective interconnection weights $w_{ij}$ such that $W_{ij} = w_{ij}/N_j$. Each population is assumed to receive an input firing rate $I_i(t)$ considered deterministic and constant in the rest of the paper.

\subsection{Deterministic Wilson and Cowan equations}
The Wilson and Cowan (WC) model is based on the assumption that the synaptic input current to each population is a function of the firing-rate of the pre-synaptic population, and that the contributions of the different populations are linearly summed to produce the post-synaptic population firing rate. It assumes moreover that the population-averaged firing rates $\nu_i(t)$ are deterministic function of time that are governed by the equations:
\begin{equation}\label{eq:WilsonCowan} 
	\der{\nu_i(t)}{t} = -\alpha_i \nu_i(t) + f_i\left( \sum_{j=1}^M w_{ij} \nu_j + I_i(t) \right)
\end{equation} 
where $f_i$ is a function transforming an incoming current into an output firing rate and typically has a sigmoidal shape. $\alpha_i$ are the relaxation rates corresponding to the natural inactivation of each population when they receive no input. 

This purely deterministic approach, often used for spatially extended neural networks, has proved efficient to model a large number of cortical phenomena, but fails to reproduce the stochastic behaviors that can appear at the level of population activity.  We now describe the stochastic framework used by~\cite{buice-cowan:07,buice-cowan:10,bressloff:09} to extend this framework in order to take into account the stochastic nature of the firing. 

\subsection{Markovian Framework}

In the Markovian approach yielding BCC and Bressloff models, each neuron can be either quiescent or active (active meaning in the process of firing an action potential), and the state of the network is then described by the variable $n \in \mathbb{N}^M$ where $n_i(t)$ is the total number of active neurons in population $i$. This chain is modeled as a one-step discrete-time Markov process, i.e. the only possible transitions of this chain are assumed to be $n_i \rightarrow n_i\pm 1 $ for $i\in \{1,\ldots,\,M\}$. Each active neuron of population $i$ returns to a quiescent state with constant probabilistic rate $\alpha_i$ and each quiescent neuron become active with a state-dependent rate $ F_i(n)$ depending on the inputs the neuron receives. The probability of the network to be in a state $n$ at time $t$ therefore satisfies a differential equation called the \emph{master equation}:
\begin{multline}\label{eq:Master}
	\der{P(n,t)}{t}=\sum_{i=1}^M \Big[\alpha_i\,(n_i+1) P(n_{i+},t) - \alpha_i n_i P(n,t) \\ 
	+ F_i(n_{i-})P(n_{i-},t) - F_i(n)\,P(n,t) \Big]
\end{multline}
From this equation \review{the authors derive equations on } the moments of the Markov chain. \review{We propose in appendix \ref{append:RodigTuck} an alternative derivation of the moment equations based on a Langevin approximation of the rescaled Markov chain and on the use of Rodriguez and Tuckwell expansions~\cite{rodriguez-tuckwell:96}. }

In the present paper we are particularly interested in \review{a variable that would correspond in the Markov setting to} the proportion of active neurons in each population: $p_i(t)=n_i(t)/N_i$ for $i=1,\ldots,N$, which is \emph{a priori} a well-behaved bounded stochastic process that takes values in the interval $[0,1]$. \review{Having defined this variable, we will be in a position to rescale BCC and Bressloff variables, and obtain from their derivation the equation governing this variable. In doing so, as further illustrated below, we will not need to derive new moment equations from the master equation, but it will appear as a change of variable in the set of ordinary differential equations these authors previously obtained. In this view, the rescaled mean firing rate is defined as the instantaneous averaged proportion of active neurons in each population $\nu_i(t)=\langle n_i(t) \rangle /N_i$, and rescaled correlations as the average value of the correlations of the proportion of active neurons in each population $C_{ij}=\langle n_i(t)n_j(t)/N_iN_j \rangle$, the two first moments of the process $p(t)$.}

The equations for the moments involve averaged values of $\langle F_i(n)\rangle$ which, since the activation is not polynomial, involve all the moments of the variable $n$, and therefore need to be truncated in order to obtain a closed set of equations. The choice of the function $F_i$ is so far unspecified and will be defined in order to recover the WC equations in the mean-field uncorrelated limit. 

With the aim of specifically investigating (uncorrelated) asynchronous states, Bressloff truncates the moment series expansion up to second order, and considers the evolution of two coupled variables: the mean firing-rate $\nu_i$ and $n_{ij} = N_j C_{ij}$ a correlation variable \review{differing from our variable in that it is scaled by $N_j$.} 
After Taylor expansion and a Van-Kampen system-size expansion, a specific choice of the activation function $F_i$ is made so that one ends up with the following set of ordinary differential equations:
{
\begin{equation}\label{eq:OriginalBressloff}
	\begin{cases}
		\der{\nu_i}{t}&=-\alpha_i \, \nu_i + f_i(s_i) + \frac{1}{2 N}\, f_i''(s_i) \sum_{k,l}w_{ik}\,w_{il}\,n_{kl}\\
		\der{n_{ij}}{t}&=[\alpha_i \, \nu_i + f_i(s_i)]\delta_{ij} - (\alpha_i+\alpha_j)\,n_{ij} \\
		& \qquad+ \sum_{k} [ f_i'(s_i) w_{ik}\,n_{kj} + f_j'(s_j) w_{jk}\,n_{ki} ]
	\end{cases}
\end{equation}
where $\nu_i$ corresponds to the firing rate of population $i$: $\nu_i(t)=\langle n_i(t)\rangle /{N}$ and $n_{ij}$ are defined as the correlations $\langle n_i(t) n_j(t) \rangle/N$. In the original formula, the number of neurons of each population is assumed to be equal to ${N}$. This quantity is bounded by application of Kurtz theorem~\cite{kurtz:76}, and therefore the term $\frac{1}{2 N}\, f_i''(s_i) \sum_{k,l}w_{ik}\,w_{il}\,n_{kl}$ tends to zero as $N$ tends to infinity, implying that the equation of the mean becomes independent of the correlations $n_{ij}$. Since we are interested in the statistics of the proportions of active neurons in each population, we are in a position to derive from the original Bressloff equations \eqref{eq:OriginalBressloff} the system of equations governing our variables $\nu_i$ and $C_{ij}$ formally, and obtain in a multi-population case with distinct population sizes: }
 \begin{equation}\label{eq:BressPopsRescaled}
 	\begin{cases}
 		\der{\nu_i}{t}&=-\alpha_i \nu_i + f_i(s_i) + \frac{1}{2} \,f_i''(s_i) \sum_{k,l}w_{ik}w_{il} C_{kl}\\
 		\der{C_{ij}}{t}&=\frac{1}{N_i}[\alpha_i \nu_i + f_i(s_i)]\delta_{ij} - (\alpha_i+\alpha_j) C_{ij} \\
 		& \qquad+ \sum_{k} [ f_i'(s_i) w_{ik}\, C_{kj} + f_j'(s_j) w_{jk}\, C_{ki} ]
 	\end{cases}
 \end{equation}
{In both equations \eqref{eq:OriginalBressloff} and \eqref{eq:BressPopsRescaled}, we denoted} by $s_i(t)$ the total instantaneous current received by population $i$ at time $t$:
\begin{equation}\label{eq:Si}
	s_i(t)=\sum_{j=1}^M w_{ij} \nu_j(t) + I_i(t).
\end{equation}
{We note that the system \eqref{eq:BressPopsRescaled}, though derived from Bressloff's master equation and formalism, is equivalent to the original Bressloff model \eqref{eq:OriginalBressloff} as long as the total number of neurons $N$ is finite, since it was derived through an invertible change of variable which is invertible as long as $N$ is finite, but not invertible in the limit $N \to \infty$. }

Buice, Cowan and collaborators~\cite{buice-cowan:07,buice-cowan:10} (BCC case), interested in studying the Poisson-like firing modes of the network, transform the moment equations to derive equations of \emph{normal ordered cumulants} measuring the deviations of the moments of the variable $n$ from pure Poisson statistics. The first normal ordered cumulant is equal to $\langle n_i \rangle$, the second ordered cumulant to $c_{ij}=C_{ij} - \nu_i/N_i \delta_{ij}$. {The initial approach of Buice, Cowan and collaborators is not a finite-size expansion \emph{per se} in the general case, hence the parameter according to which the expansion is performed is no longer $1/N$ as in the case of Bressloff's Van Kampen expansion, but consists of a multiple time scale expansion where the small parameter is the decay time of normal order cumulants (which applies far from bifurcation points). However, in the fully connected case we are interested in in the present manuscript with weights scaling as $1/N$, the expansion provided Buice, Cowan and collaborators method can be reduced to a finite-size expansion, and can be compared to Bressloff's case.} In terms of these variables after some approximations and moment truncation to the second-ordered cumulant, they end up  with the following set of ordinary differential equations (this step involves an instantiation of the activation functions $F_i$ different from the Bressloff case):
\begin{equation}\label{eq:BCPops}
	\begin{cases}
		\der{\nu_i} {t} &= -\alpha_i \nu_i + f_i(s_i) + \frac 1 2 f_i''(s_i) \sum_{j,k} w_{ij} w_{ik} c_{jk}\\
		\der{c_{ij}}{t} &= -(\alpha_{i}+\alpha_{j}) c_{ij} + f_i'(s_{i})\sum_{k} w_{ik} c_{kj} + f_j'(s_{j})\sum_{k} w_{jk} c_{ki} + \\
		& \qquad + \frac{1}{N_{j}} f_i'(s_{i})\,w_{ij}\,\nu_{j} + \frac{1}{N_i} f_j'(s_j)\,w_{ji}\,\nu_{i}
	\end{cases}
\end{equation}

\begin{remark}
	{We note that the rescaling does not affect the form of the equations obtained in \cite{buice-cowan:10}. These equations are valid under our assumptions that the synaptic weights $W_{ij}$ scale as $1/N$ in a fully connected network. We emphasize the fact that both expansions can be rigorously derived far away from bifurcation points. In the manuscript, we will however be interested in the bifurcations of the dynamical systems given by \eqref{eq:BressPopsRescaled} and \eqref{eq:BCPops}. Indeed, these bifurcations produce changes in the number and in the stability of attractors (fixed point and limit cycle) which will be visible in non-trivial parameter intervals, and therefore will affect the behaviors of the system far away from bifurcations and will give us indications of the number and stability of attractors in those parameter regions.}

	{These models formally appear to very similar. However, it is important to note that the these two approaches are different: the master equation corresponds to different transition rates (different choices of the functions $F_i$ as stated), and the variables considered differ. The two equations \eqref{eq:BressPopsRescaled} and \eqref{eq:BCPops} constitute our starting point for a mathematical exploration of the solution, and in the rest of this paper, we mathematically analyze these equations. The reader nevertheless needs to keep in mind the differences between the two approaches for interpreting the results.}
\end{remark}
 
 \subsection{The infinite-size model} 
{The equations \eqref{eq:BressPopsRescaled} and \eqref{eq:BCPops} are \emph{a priori} different. First of all, they do not deal with the same quantities: equation \eqref{eq:BressPopsRescaled} couples the mean firing-rate and the correlations while equation \eqref{eq:BCPops} the mean firing-rate and the first order cumulant. However, in the limit $N\to\infty$, the first-order cumulants, $c_{ij}=C_{ij} - \nu_i/N_i \delta_{ij}$, are simply equal to the correlation $C_{ij}$. Moreover, in the rescaled }models we introduced, we observe that when the number of neurons tends to infinity, both Bressloff \eqref{eq:BressPopsRescaled} and BCC \eqref{eq:BCPops} models converge to the same equations. These equations will be referred to as the \emph{infinite size model} and are given by the equations:
 \begin{equation}\label{eq:InfiniteModel}
 	\begin{cases}
 		\der{\nu_i}{t}&=-\alpha_i \nu_i + f_i(s_i) + \frac{1}{2}f_i''(s_i) \sum_{k,l}w_{ik}w_{il} \Delta_{kl}\\
 		\der{\Delta_{ij}}{t}&= - (\alpha_i+\alpha_j)\Delta_{ij} \\
 		& \qquad+ \sum_{k} [ f_i'(s_i) w_{ik} \Delta_{kj} + f_j'(s_j) w_{jk} \Delta_{ki} ]
 	\end{cases}
 \end{equation}
 {We recall that although the interpretation of the variable $\Delta$ differs in the general case, it is the not the case of the infinite model}: in BCC model $\Delta_{ij}=c_{ij}$ are the second ordered cumulant, while in Bressloff model $\Delta_{ij}=C_{ij}$ is the correlation in the firing activity, and in the infinite-size limit, the normal ordered cumulant is equal to the correlation. However, when $N$ becomes finite, the reader needs to bear in mind that in the finite-size unfolding of the infinite-size equations in BCC case, $\Delta$ represents the deviation of spike statistics from a Poisson Process and in Bressloff model the covariance of the firing activity. When letting $N$ be finite, the two models unfold the behavior of the system in a different fashion. We also note that if the variable $\Delta(t)$ is equal to zero, then the mean-firing rates $\nu_i(t)$ satisfy the WC equations. In that view, Bressloff and BCC models are generalizations of WC system that take into account the correlations in the firing. 
 
 In this paper, we will first be interested in studying the mean-field limit solutions of the infinite size equations \eqref{eq:InfiniteModel}, and second in unfolding these mean-field behaviors when the number of neurons is large but finite. Our study will particularly focus on two main aspects: (i) how the stochastic nature of the firings affect the observed behaviors at the macroscopic level through the interplay of firing activity and the correlations, and (ii) how the finiteness of the number of neurons in the network disturbs these behaviors. 

\review{We also note that in view of Kurtz' theorem \cite{kurtz:76}, our variable in the infinite-size system vanishes. We formally consider these equations independently of this restriction. In other words, the infinite-size system is obviously not equivalent to the WC system, but BCC and Bressloff systems should be equivalent, in this limit, to WC system. The study of the infinite-size system will allow us to find what we call \emph{correlation-induced behaviors}, which are qualitative distinctions between the behaviors of the solutions of the infinite-size system and of Wilson and Cowan system. If these behaviors are unfolded non-trivially into solutions of the finite-size equations, these will evidence additional behaviors of the system presented by the finite-size Bressloff and BCC equations that do not correspond to solutions of Wilson and Cowan system. We will address these questions in in section \ref{sec:FiniteSize}. }
 
Besides our theoretical analysis on the equations, we will also go back to the initial Markovian model and compare (with Monte Carlo simulations) the stochastic behaviors it presents to the solutions of WC, BCC and Bressloff models.
 
 \section{Influence of the correlations in the \review{infinite-size system}}\label{section:Infinite}
 We have shown that in the limit where number of neurons is infinite, both Bressloff and BCC models, {in our particular rescaling,} converge to the infinite-size model given by equations \eqref{eq:InfiniteModel}  where the mean-firing rate is coupled to the correlations in the firing activity. When the correlations are equal to zero, the equation for the mean-firing rate is uncoupled from the one for the correlations $\Delta$ and is identical to the WC system. The question we address in this section is how the interaction between the mean activity and the correlations modifies the behavior of the global activity compared to WC system given by equations \eqref{eq:WilsonCowan}. We address the following two questions: (i) are the solutions of Wilson and Cowan equations also solutions of the full systems including correlatons, and if the answer is yes, are the stability properties of these solutions conserved? and (ii) Do the correlations yield other behaviors? 
 
 \subsection{Wilson and Cowan's Solution Solve the Infinite-Size System}
 In equation \eqref{eq:InfiniteModel}, it is easy to see that zero correlation: $\Delta(t)=0$   is always a fixed point of the correlation equations no matter what mean firing rates $\nu(t)$ are. For zero correlations, the equations for the mean firing rate $\nu$ are reduced to the classical WC equations. Therefore, WC solutions define solutions for the infinite-size system with zero correlations. Let us now study the stability of the solutions defined by these uncorrelated WC behaviors in the infinite size system.
 
 The infinite-size system involves a standard $M$-dimensional differential equation on the mean firing rates $\nu(t)$ coupled to a matrix differential equation. In order to use customary approaches for dynamical systems, we transform the correlation matrix into a $M^2$ dimensional vector and express the equations on $\Delta$ in this new format using Kronecker products from linear algebra. To this end, we start by defining the function $\Vect$ transforming a $M \times M$ matrix into a $M^2$-dimensional column vector, as defined in  \cite{neudecker:69}:
 \[\Vect: \begin{cases}
 \R^{M\times M} &\mapsto \R^{M^2}\\
 X & \mapsto [X_{11}, \ldots, X_{M1}, X_{12}(t), \ldots, X_{M2}(t), \ldots X_{1M}(t), \ldots, X_{MM}(t)]^T
 \end{cases}\]
 Let us now denote by $\otimes$ the Kronecker product defined for $A\in \R^{m\times n}$ and $B \in \R^{r\times s}$ as the $(m\,r) \times (n\,s)$ matrix:
 \[A\otimes B := 
 \left(
 \begin{array}{cccc}
 	a_{11} B & a_{12} B & \cdots & a_{1n}B\\
 	a_{21} B & a_{22} B & \cdots & a_{2n}B\\
 	\vdots & \vdots & \ddots & \vdots\\
 	a_{m1} B & a_{m2} B & \cdots & a_{mn}B\\
 \end{array}
 \right)\]
 Some definitions and identities in the field of Kronecker products are reviewed in appendix \ref{append:Kronecker}, and a more complete discussion can be found in \cite{brewer:78} and references therein. We recall here for the sake of completeness some well-known relationship that will be useful. Let $A,B,D,G,X \in \R^{M\times M}$, let $I_M$ be the $M\times M$ identity matrix and $A\cdot B$ or $A\,B$ denote the standard matrix product. We have:
 \begin{equation}\label{eq:Kronecker}
 	\begin{cases}
 		\Vect (AXB) = (B^T \otimes A ) \Vect(X)\\
 		A\oplus A =A\otimes I_M + I_M \otimes A\\
 		(A \otimes B)\cdot (D \otimes G) = (A\cdot D) \times (B\cdot G)
 	\end{cases}.
 \end{equation}
 The relationship $\oplus$ is called Kronecker sum. In this framework, we have the following identity:
 
 \begin{lemma}\label{lem:CKron}
 	Let $V_{\Delta}(t)=\Vect(\Delta(t))$ be the column vectorization of the matrix $\Delta(t)$, and $A(x)$ the matrix defined by $A(x) = -{\bm \alpha} + {\bm F}(x)$ where ${\bm \alpha}$ is the diagonal matrix with $(i,i)$ element ${\bm \alpha}_{ii} = \alpha_i$ and ${\bm F}$ the matrix of general element $({\bm F})_{ij} = (f_i'(\sum_{k=1}^M w_{ik}\,x_k) w_{ij})$.
 	 
 	The variable $V_{\Delta}(t)$ satisfies the differential equation in $\R^{M^2}$:
 	\begin{equation}\label{eq:VDelta}
 		\der{V_{\Delta}(t)}{t} = (A(\nu(t))\oplus A(\nu(t))) V_{\Delta}(t).
 	\end{equation}
 \end{lemma}
 
 \begin{proof}
 	The differential equation governing the evolution of the coefficients of the correlation matrix $\Delta(t)$ of the infinite size system \eqref{eq:InfiniteModel} can be easily reordered into:
 	\[\der{\Delta_{ij}}{t} = \sum_{k=1}^M \left( A_{ik}(\nu(t)) \Delta_{kj}(t) + A_{jk}(\nu(t))\Delta_{ik} \right) \]
 	which, through straightforward linear algebra manipulations, can be written as:
 	\[\der{\Delta}{t} = A(\nu(t))\cdot \Delta(t) + \Delta(t) \cdot A(\nu(t))^T.\]
 	The linear operator $ X \mapsto A(\nu(t))\cdot X + X \cdot A(\nu(t))^T$ for $X \in \R^{M\times M}$ can be written in terms of Kronecker product of matrix on the vectorized version $V_{\Delta}(t)$ of the matrix $\Delta(t)$ using the relationship given in \eqref{eq:Kronecker} and we have: 
 	\begin{align*}
 		\Vect(A(\nu(t))\cdot \Delta(t) + \Delta(t) \cdot A(\nu(t))^T) &= \Vect\Big(A(\nu(t)) \cdot \Delta(t) \cdot I_M\Big) + \Vect\Big(I_M\cdot \Delta(t) \cdot A(\nu(t))^T\Big)\\
 		&= \Big(I_M^T \otimes A\big(\nu(t)\big) + A\big(\nu(t)\big) \otimes I_M\Big) \cdot V_{\Delta}(t) \\
 		&= \Big(A\big(\nu(t)\big)\,\oplus\, A\big(\nu(t)\big)\Big)\cdot V_{\Delta}(t).
 	\end{align*}
 \end{proof}
 
 Now that we have defined a convenient framework to study the infinite size equations, we are in position to study the stability of fixed points and limit cycles of the WC system as solutions with zero correlations of the infinite size system. We start by addressing the problem in the one population model ($M=1$), which allows straightforward calculations and will be helpful in our specific study of the one population infinite-size equation.
 
 \begin{proposition}\label{prop:onepop}
 	In the one population case, any fixed point of Wilson and Cowan equation is a zero-correlation fixed point of the infinite-size system, and with the same stability properties.
 \end{proposition}
 
 \begin{proof}
 In the one population case, equations \eqref{eq:BCPops} constitute a planar dynamical system that reads:
 \begin{equation}\label{eq:BConePop}
 	\begin{cases}
 		\der{\nu}{t} &= -\alpha \nu + f(s) + \frac 1 2 f''(s) w^2 \Delta\\
 		\der{\Delta}{t} &= -2\alpha \Delta + 2\,f'(s) \,w\,\Delta
 	\end{cases}
 \end{equation}
 where $s=w\,\nu+I$. 
 
 As stated, the solution $\Delta=0$ is a fixed point of the correlation equation, and in that case, the activity satisfies the one-population WC equation:
 \begin{equation*}
 	\der{\nu}{t} = -\alpha \nu + f(w\,\nu+I)
 \end{equation*}
 Let $\nu^*$ be a fixed point of this system. The Jacobian matrix of the infinite-size system at the fixed point $(\nu^*,\Delta \equiv 0)$ reads:
 \[\left (
 \begin{array}{cc}
 	-\alpha + w\,f'(w\nu^*+I) & \frac 1 2 f''(w\nu^*+I) w^2\\
 	0 & 2\big(-\alpha + w\,f'(w\nu^*+I) \big)
 \end{array}
 \right)\]
 As we can obviously see, the eigenvalues of this equilibrium are $\lambda_1:=(-\alpha + w\,f'(w\nu^*+I))$ and $\lambda_2=2\,\lambda_1$, and therefore the stability of this fixed point is the same as the stability of the related fixed point in WC's system.
 \end{proof}
 
 We therefore conclude that any stable solution of the one dimensional WC equation provides a zero correlation ($\Delta=0$) stable fixed point of the infinite-size system, and we observe that no stable solution of WC system is destabilized by the presence of correlations, and no unstable solution of the Wilson and Cowan equation will be stabilized. Therefore, zero-correlation fixed points of the infinite size system exists if and only if Wilson and Cowan equation has a fixed point, and the related fixed point of the infinite-size system has the same stability as WC's fixed point. 
 
 We now turn to demonstrating the related properties in arbitrary dimensions $M$, for fixed points and cycles. 
 
 \begin{theorem}\label{thm:WilsonBuiceCowan}
 	Solutions of the WC system provide solutions of the infinite-size system with zero correlations $\Delta=0$. The stability of fixed points and limit cycles of the infinite size system depend on the stability of the solution for the WC system as follows:\\

{\it i)} A fixed point of the infinite-size equation with null correlations is stable if and only if the value of related mean-firing rate is a stable fixed point of WC system;\\

{\it ii)} A cycle of the infinite size system with zero correlations $(\nu(t), \Delta\equiv 0)$ is exponentially unstable if and only if $\nu(t)$ is an unstable cycle of WC system. Stable cycles $\nu(t)$ of WC's system define cycles of the infinite-size system with zero correlations with neutral linear stability. The infinite size system does not present any stable cycle with null correlations. 
\end{theorem}
 
 \begin{proof}
 In order to prove the theorem, it is convenient to introduce the vector fields $F_{\nu}: \R^M\mapsto R^M$, $G_{\nu}(\nu):\R^M \mapsto R^{M\times M^2}$ and $F_\Delta: \R^M \mapsto \R^{M^2\times M^2}$ that govern the dynamics of the infinite size system:
 \[\begin{cases}
 	\der{\nu}{t} &= F_{\nu}(\nu) + G_{\nu}(\nu) V_{\Delta}(t)\\
 	\der{V_{\Delta}}{t} &= F_{\Delta}(\nu) V_{\Delta}(t)
 \end{cases}\]
 We have $F_{\nu}(\nu) = -\mathbf{\alpha}+[f_i(s_i), i=1\ldots M]^T$, $F_{\Delta}(\nu) =A(\nu) \oplus A(\nu)$ where $A(\cdot)$ is defined in lemma \ref{lem:CKron}, and $G_{\nu}(\nu)$ is the $M^2\times M^2$ matrix such that:
 \[G_{\nu}(\nu)V_{\Delta}(t) = [\frac{1}{2}f_i''(s_i) \sum_{k,l}w_{ik}w_{il} \Delta_{kl}, i=1\ldots M]^T\]
 
 Let $(\nu(t),\Delta\equiv 0)$ be a solution of the infinite-size system. The Jacobian matrix of the system evaluated on this solution has the form:
 \begin{align*}
 	J(\nu(t),0) &= \left ( 
 		\begin{array}{cc}
 			d_\nu (F_\nu) + d_{\nu}(G_{\nu})V_{\Delta} \Big \vert_{\nu(t), V_{\Delta}\equiv 0} & G_{\nu}(\nu)\Big \vert_{\nu(t), V_{\Delta}\equiv 0}\\
 			d_\nu (F_\Delta) V_{\Delta} \Big \vert_{\nu(t), V_{\Delta}\equiv 0} & d_{V_{\Delta}}(F_\Delta)\Big \vert_{\nu(t), V_{\Delta}\equiv 0}		
 		\end{array}
 		\right)\\
 		\\
 		&= \left ( 
 			\begin{array}{cc}
 				d_\nu (F_\nu) (\nu(t)) & G_{\nu}(\nu(t))\\
 				{\bm 0} & d_{V_{\Delta}}(F_\Delta)(\nu(t))
 			\end{array}
 			\right)
 \end{align*}
  where $d_X f$ for a multidimensional vector $X$ and a multidimensional function $f$ denotes the differential of $f$ with respect to $X$. Since $F_\nu$  is exactly the vector field associated with WC system, $d_\nu F_\nu$ is an $M\times M$ block corresponding to Jacobian matrix of WC system at the point $\nu$, which is equal to the matrix $A(\nu)$ introduced in lemma \ref{lem:CKron}.
 
 Now that the Jacobian matrix is identified, we prove the assertions of the theorem:\\

{\it i)} Let us consider a fixed point $(\nu, \Delta\equiv 0)$ of the infinite size system. Necessarily, $\nu$ is a fixed point of Wilson and Cowan equation. The stability of this fixed point depends on the spectrum of the Jacobian matrix $J$ of the system at this point. Since the Jacobian matrix of the infinite-size system is block diagonal, its spectrum is therefore composed of the eigenvalues each diagonal block $d_\nu(F_\nu) = A(\nu)$ and $d_{V_{\Delta}}(F_\Delta)(\nu(t))$. The first block, $A(\nu)$, has exactly the eigenvalues of Wilson and Cowan system at its fixed point $\nu$. These eigenvalues are denoted $\{\lambda_i, i=1\ldots M\}$. The second diagonal block of the Jacobian matrix $d_{V_{\Delta}}(F_\Delta)$ is equal to the Kronecker sum $A(\nu)\oplus A(\nu)$ by application of lemma \ref{lem:CKron}. The eigenvalues of a Kronecker sum are known to be all possible pairwise sums of all the eigenvalues of $A(\nu)$, viz. $(\lambda_i+\lambda_j;\;(i,j)\in\{1,\ldots,M\}^2)$ (see an elementary proof of this fact in proposition \ref{prop:KronEigen} of appendix \ref{append:Kronecker}), hence the spectrum of the Jacobian matrix of the infinite size system is composed of the eigenvalues:
 		\[\{\lambda_i, i=1\ldots M\} \cup \{\lambda_i+\lambda_j , (i,j)\in\{1,\ldots,M\}^2\}.\]
 		 These eigenvalues depend in a simple fashion on the eigenvalues of the WC Jacobian matrix at the fixed point $\nu$, and it is easy to show that the fixed point $(\nu, 0)$ in the the infinite system is:
 		\begin{itemize}
 			\item exponentially stable if and only if the $\nu$ is an exponentially stable fixed point of WC system, i.e. if and only if all the eigenvalues $\lambda_i$ have a strictly negative real part.
 			\item exponentially unstable if and only if the $\nu$ is an exponentially unstable fixed point of WC system, i.e. if and only if there exists an eigenvalue $\lambda_i$ with strictly positive real part.
 			\item neutrally stable if and only if $\nu$ is neutrally stable for WC system, i.e. if and only if all eigenvalues have non-positive real part and at least one eigenvalue have a null real part.
 		\end{itemize} 
In summary, the stability of the zero-correlation fixed point $(\nu, \Delta\equiv 0)$ is exactly the same as the stability of $\nu$ as a fixed point of WC system. 

{\it ii)} Let us now address the case of cycles will null correlations. To this end, let us consider that $(\nu(t), \Delta(t)\equiv 0)$ is a periodic orbit of the infinite-size system (i.e. $\nu(t)$ is a periodic orbit of a WC system). Let us denote by $T>0$ the period of this cycle. We show that the null correlations $\Delta \equiv 0$ is an exponentially unstable solution if the cycle $\nu(t)$ is exponentially unstable as a solution of WC system, and neutrally stable if the cycle $\nu(t)$ is exponentially stable or neutrally stable as a solution of Wilson and Cowan system, which proves the theorem.
 		
 		Given that the cycle $\nu(t)$ is known, the correlations $\Delta(t)$ satisfy a linear equation:
 		\[\der{V_{\Delta}(t)}{t}=\Big(A(\nu(t)) \oplus A(\nu(t))\Big)\, V_{\Delta}(t)\]
 		and the time-dependent matrix $\Big(A(\nu(t)) \oplus A(\nu(t))\Big)$ is $T$-periodic. A basic result of Floquet theory implies that the stability of the solution $V_{\Delta} \equiv 0$ depends on the eigenvalues of the resolvant of the system at time $T$, namely the matrix $\Psi(T)$ in $\R^{M^2\times M^2}$ where $\Psi(\cdot)$ is defined by
 		\[
 		\begin{cases}
 			\der{\Psi(t)}{t} &= \Big(A(\nu(t))\oplus A(\nu(t))\Big)\Psi(t)\\
 			\Phi(0) &= I_{M^2}
 		\end{cases}
 		\]
 		More precisely, the null fixed point is exponentially stable if and only if all eigenvalues of $\Psi(T)$ ( the \emph{multipliers}) are of modulus strictly smaller than $1$, neutrally stable if all have modulus smaller or than equal to $1$ with at least one multiplier with modulus equal to $1$, and exponentially unstable if there exists a multiplier with modulus larger than $1$. We therefore need to characterize the eigenvalues of $\Psi(T)$ in order to conclude on the stability of the solution in the infinite-size system.   
 		To this end, let us introduce $\Phi(t)$ the resolvent of WC system, i.e. the unique solution of the fundamental equation:
 		\[
 		\begin{cases}
 			\der{\Phi(t)}{t} &= A(\nu(t))\Phi(t)\\
 			\Phi(0) &= I_M
 		\end{cases}
 		\]
 		We prove in theorem \ref{theo:PsiKron} of appendix \ref{append:Kronecker} that $\Psi(t)=\Phi(t)\otimes\Phi(t)$. Let us denote by $\{\mu_i,\; i=1,\ldots,M\}$ the eigenvalues of $\Phi(T)$. Because $\Psi(T)$ is the Kronecker square of $\Phi(T)$, its eigenvalues are are all pairwise products of the eigenvalues of $\Phi(T)$, i.e. $\{\mu_i\,\mu_j, \; i,j=1\ldots M\}$, as shown in proposition \ref{prop:KronEigen} of appendix \ref{append:Kronecker}.
 		
If the cycle $\nu(t)$ is exponentially unstable as a solution of WC system, then Floquet theory implies that there necessarily exists a multiplier of $\Phi(T)$, $\mu_i$, such that $\vert \mu_i \vert >1$. Therefore the resolvant $\Psi(T)$ has a multiplier equal to $\mu_i^2$ which has modulus equal to $\vert \mu_i \vert^2 >1$ and therefore this cycle is exponentially unstable as a solution of the infinite-size system. 
 		
 		If $\nu(t)$ is exponentially stable or neutrally stable as a solution of WC system, then a basic result of Floquet theory states that the resolvant if the linearized system evaluated on the cycle at time $T$, $\Phi(T)$, has at least a multiplier equal to $1$ and all other multipliers with multipliers of modulus smaller than $1$. We now provide an elementary proof of the existence of the $1$ multiplier: let $\xi(t)=F_{\nu}(\nu(t))$. Then since $\nu(t)$ is $T$-periodic, so is $\xi(t)$, and furthermore we have:
 		\begin{align*}
 			\der{\xi(t)}{t}=\der{F_{\nu}(\nu(t))}{t} = (d_{\nu}F_{\nu})(\nu(t))\,\der{\nu(t)}{t} = (d_{\nu}F_{\nu})(\nu(t))\,F_{\nu}(\nu(t)) = A(\nu(t))\,\xi(t)
 		\end{align*}
 		Moreover, since $\nu(t)$ is not a fixed point, $F(\nu(0))\neq 0$. We have shown $\xi(t)$ is solution of the linearized equation and hence using its $T$-periodicity, we have $\xi(T) = \Phi(T)\xi(0) = \xi(0)$, which proves that $\xi(0)$ is eigenvector of $\Phi(T)$ associated with the eigenvalue $1$. Therefore, the resolvent $\Psi(T)$ necessarily has an eigenvalue equal to one (associated with the eigenvector $\Vect(\xi(0)\xi(0)^T)$), which implies that the null fixed point has a neutral linear stability as a solution of the infinite-size system. 
 We directly conclude from this result that the infinite size system does not features any exponentially stable cycle, which ends the proof. 
 \end{proof}
 
We conclude that the solutions of WC system always provide solutions of the infinite-size system, and that all fixed points with null correlations in the infinite-size system are fixed points of WC system, with the same stability properties. However, cycles of WC system lose exponential stability in the infinite-size system. Note that this does not necessarily implies that these cycles become unstable, and a nonlinear stability analysis is necessary. But it implies that the transient phase of convergence towards the cycle or of repulsion from the cycle will be not be exponential.
 
 
 \subsection{Correlation-induced behaviors}
 In the previous section, we just proved that all the fixed points of the WC system are solutions of uncorrelated activity in the infinite size system. We now investigate the existence of new solutions induced by the presence of non-null correlations in the infinite system, that are therefore not solutions of WC system. Such solutions will be referred to as \emph{correlation-induced behaviors}. The general resolution of this problem is quite hard. For this reason, we provide a rigorous analysis of the one population case, and treat two-population case through numerical analysis and simulations.
 
 \subsubsection{One population case}
 We have seen in the one population case that any fixed point of WC system was solution of the infinite size system with zero correlations. Moreover, there is no possible cycle in the Wilson and Cowan equations in one dimension. We investigate now the existence of cycles or additional fixed points in the infinite size system with one population. We recall that acceptable solutions in the one population case necessarily have non-negative firing rate $\nu$ and $\Delta$. Indeed, in the case of the infinite-size model arising in the limit of Bressloff's rescaled model, the correlation $\Delta$ needs to be positive as a limit of positive quantities (and as the covariance of the Markov process), and in BCC case $\Delta$ is equal {to the limit of the first order cumulant which differs from} the correlation up to the coefficient $-\nu/N$ that vanishes in the infinite-size limit.
 
 \begin{theorem}\label{thm:OnePopNoSolution}
 	In the one population infinite size system, there does not exist any acceptable stable fixed point with strictly positive correlation $\Delta$.\\
 \end{theorem}
 
 \begin{proof}
 	Let us assume that $\Delta \neq 0$. The fixed point equation of the infinite-size system:
 \begin{equation}\label{eq:FP1Pop}
 	\begin{cases}
 		-\alpha+f'(w\,\nu+I) w = 0\\
 		-\alpha \nu + f(w\nu + I) + \frac 1 2 f''(w\nu+I)w^2\,\Delta = 0
 	\end{cases}
 \end{equation}
 	The first equation of the system \eqref{eq:FP1Pop} is independent of $\Delta$ and fixes possible values of $\nu$. Assuming that this equation has a solution $\nu^*$ and denoting $s^*=w\,\nu^*+I$, the Jacobian matrix of the system at $\nu=\nu^*$ reads, using the property that $\nu^*$ solves the first equation of \eqref{eq:FP1Pop}:
 	\begin{equation}\label{eq:Jac1Pop}
 		\left (
 		\begin{array}{cc}
 			\frac 1 2 f'''(s^*) w^3 \Delta & \frac 1 2 f''(s^*) w^2\\
 			2 f''(s^*)\,w^2\,\Delta & 0
 		\end{array}
 		\right)
 	\end{equation}
 Such values of $\nu^*$ necessarily exist for some parameters, and since the function $f$ is assumed to be sigmoidal, it can have multiple solutions. Moreover, because of the particular shape of the sigmoidal function $f$, the second derivative of $f$ can vanish at a point corresponding to the inflection point of the sigmoid. 
 	
 In a very particular case (for precisely tuned values of the parameters), one can hence have $f''(w\,\nu^*+I)=0$. In that case, fixed points only exist if $-\alpha + f'(w\nu^* + I)w=0$, if this condition is satisfied, any value of the correlation $\Delta$ provides a fixed point of the system. This fixed point is never exponentially stable, since in that case the Jacobian matrix has a zero eigenvalue (obvious when using the expression of the Jacobian matrix \eqref{eq:Jac1Pop} and using the fact that $f''(s^*)=0$). 
 	
 In the general case where $f''(s^*)\neq 0$, the system has a fixed point with non-null correlations 
 	\[\left ( \nu^*,\Delta^*:=2 \frac{\alpha \nu^* -f(s^*)}{w^2\,f''(s^*)}\right).\]
 	We therefore need to check whether if this fixed point is stable and acceptable (i.e. $\min(\nu^*,\Delta^*) \geq 0$). The Jacobian matrix at this point has the expression given by \eqref{eq:Jac1Pop}. Its determinant is equal to $-\big(f''(s^*)\big)^2\,w^4\,\Delta^*$ and has an opposite sign to $\Delta^*$. Therefore, for acceptable solutions with  $\Delta^*>0$, the determinant of the Jacobian matrix is strictly negative. We conclude that any fixed point of the infinite-size system with $\Delta>0$ are saddle fixed points. 
 \end{proof}

{We conclude from this theorem that the system does not features any acceptable stable fixed point with non-zero correlations.
We can now easily conclude that there are no acceptable cycles. Consider, first, a limit cycle such that $\Delta$ is strictly positive. Such a limit cycle must surround one or more fixed points whose Poincare index must sum to +1 \cite{arnold}. However, every fixed point in the postive orthant is a saddle point and saddle points always have indices of -1.  Thus, no limit cycle can have a strictly positive $\Delta.$ Any limit cycle must therefore be tangent to the $\Delta=0$ line and the point of tangency cannot contain a fixed point, so, again, the summed index of points in the limit cycle cannot be +1.}

In conclusion, there are no acceptable limit cycles and all acceptable fixed points are saddles. Thus, the only stable behavior are the uncorrelated stationary states of the scalar WC equation. This situation will be quite different in the multi-population case, as we now develop.
 
 \subsubsection{Correlation-induced behaviors in multi-populations networks}
We now turn to study multi-population networks with the aim of identifying solutions of the infinite-size system that qualitatively differ from WC equation. The situation in the multidimensional case will be quite different, since we have proven in theorem \ref{thm:WilsonBuiceCowan} that cycles of WC system lost exponential stability in the infinite-size system, allowing the appearance of distinct transient and/or asymptotic behaviors. 

To identify such correlation-induced behavior we numerically study two particular two population networks. The first model (Model I) is a two-population network including an excitatory and an inhibitory population known to produce oscillations, and the second example (Model II) builds on a famous model of binocular rivalry presenting bistability of fixed points. 

 \paragraph{Perturbation of a WC cycle}\label{sec:NegFeedback}
We proved in theorem \ref{thm:WilsonBuiceCowan} that any cycle of WC model lost exponential stability in the infinite-size system. This loss of exponential stability does not necessarily mean that the cycle loses stability, and this property depends on the higher order terms of the nonlinear equation, but even if the cycle keeps stability, the transient convergence phase towards the cycle will be dramatically slowed. 
 
In this section, we choose to study a two-populations network featuring an excitatory population interconnected with an inhibitory one. Because of the symmetry $\Delta_{ij}=\Delta_{ij}$, the system is of dimension $5$ (the two mean firing-rates and three correlation variables). To fix ideas, we denote by $1$ (resp. $2$) the excitatory (resp. inhibitory) populations. We choose the same activation function for both populations equal to  $f(x)=1/(1+\exp(-x))$ and the same inactivation rate $\alpha=1$, and are interconnected through the following connectivity matrix:
 \[
 \left ( \begin{array}{cc}
 	w_{11}=15 & w_{12}=-12\\
 	w_{21}=16 & w_{22}=-5
 \end{array}
 \right).
 \]
 Each population receives different input currents $I_1=-0.5$ and $I_2=-5$. These parameters define a model called Model I in the sequel. For these functions and parameters, the WC system features a cycle. When taking the input to the excitatory population $I_1$ small, the system features a single stable fixed point, that loses stability as the input $I_1$ increases through a supercritical Hopf bifurcation generating a family of stable limit cycles, that disappear through a homoclinic bifurcation (see Figure~\ref{fig:CyclesMeanField2pops}(a)), branching onto a high-state stable fixed point. Let us now analyze how this picture is modified by the presence of correlations. 
 
 First of all, from theorem \ref{thm:WilsonBuiceCowan}, it is clear that the parameter point where WC system undergoes the supercritical Hopf bifurcation will be a very degenerate bifurcation point for the infinite-size system. Indeed, the Hopf bifurcation is characterized by the presence of a pair of purely imaginay complex conjugate eigenvalues. Then lemma \ref{lem:CKron} and the fact that the Kronecker sum of two matrixes has all pairwise sums of eigenvalues of the matrixes in the sum (see appendix \ref{append:Kronecker}) directly implies that the Jacobian matrix of the full system \eqref{eq:InfiniteModel} has the eigenvalues $\{\pm \lambda, \pm 2 \lambda, 0\}$ and the eigenvalue $0$ is of multiplicity one. Therefore at this point, the Jacobian matrix of the infinite size system has its $5$ eigenvalues having a null real part, and thus at this point the system is very degenerate and in particular the bifurcation is not a generic Hopf bifurcation. It is quite difficult numerically and analytically to study the solutions emerging from this bifurcation point, and many solutions can appear at this point, including cycles and fixed points. Similarly, at the two saddle-node bifurcations of WC system, since the Jacobian matrix of the infinite-size system has all the pairwise sums of the eigenvalues, it will have two null eigenvalues and therefore a degenerate bifurcation. The behavior of the system around these degenerate points will appear more clearly in the finite-size unfolding of these degenerate bifurcation points, and we will observe in particular that the Hopf bifurcation point corresponds to the merging of two generic Hopf bifurcations and the saddle node bifurcations to the merging of two generic saddle-node bifurcations, see section~\ref{sec:NegFeedback} and figure Fig.~\ref{fig:Codim2BC2Pops}. In this section we do not address here the problem of classifying all behaviors of the system around the bifurcation (since as already mentioned, the finite-size unfolding will help answering these questions), and restrict the analysis to the question of whether the WC cycle keeps a nonlinear stability in the new system. To answer this question, we numerically perturb WC's cycle by adding a small initial condition in the correlations. Interestingly, we observe that a new cycle appears, that is attractive in a certain range of parameters (see Figure~\ref{fig:CyclesMeanField2pops}). The WC cycle does not completely loses stability and still appears as the only behavior of the system in a certain range of input values (e.g. $I_1=-0.5$ as plotted in Figure~\ref{fig:CyclesMeanField2pops}(b)). But for smaller values of the parameters, it loses its nonlinear stability and the additional cycle appears (e.g. $I_1=-2$ and Figure~\ref{fig:CyclesMeanField2pops}(c) and (d)). This new cycle is totally different in its shape, has a period close to half the period of the cycle corresponding to WC system, and has non-zero correlations that vary periodically at the same frequency. A continuation of this branch of cycles shows that it is stable in a significant range of parameters (Figure~\ref{fig:CyclesMeanField2pops}(e)). It loses stability when $I_2$ decreases through a period-doubling bifurcation, and as $I_2$ increases through a Neimark-Sacker bifurcation. This branch of limit cycles emerges from the very degenerate point corresponding to the Hopf bifurcation of WC system, as we expected. It disappears at a point corresponding to a subcritical Hopf bifurcation with correlations (on a branch of unstable fixed points that is not plotted in the diagram but that will be further investigated in section \ref{sec:TwoPopsFinite}). 
 
 \begin{figure}[!h]
   \centering
		 \subfigure[Bifurcation Diagram, Wilson and Cowan]{\includegraphics[width=.3\textwidth]{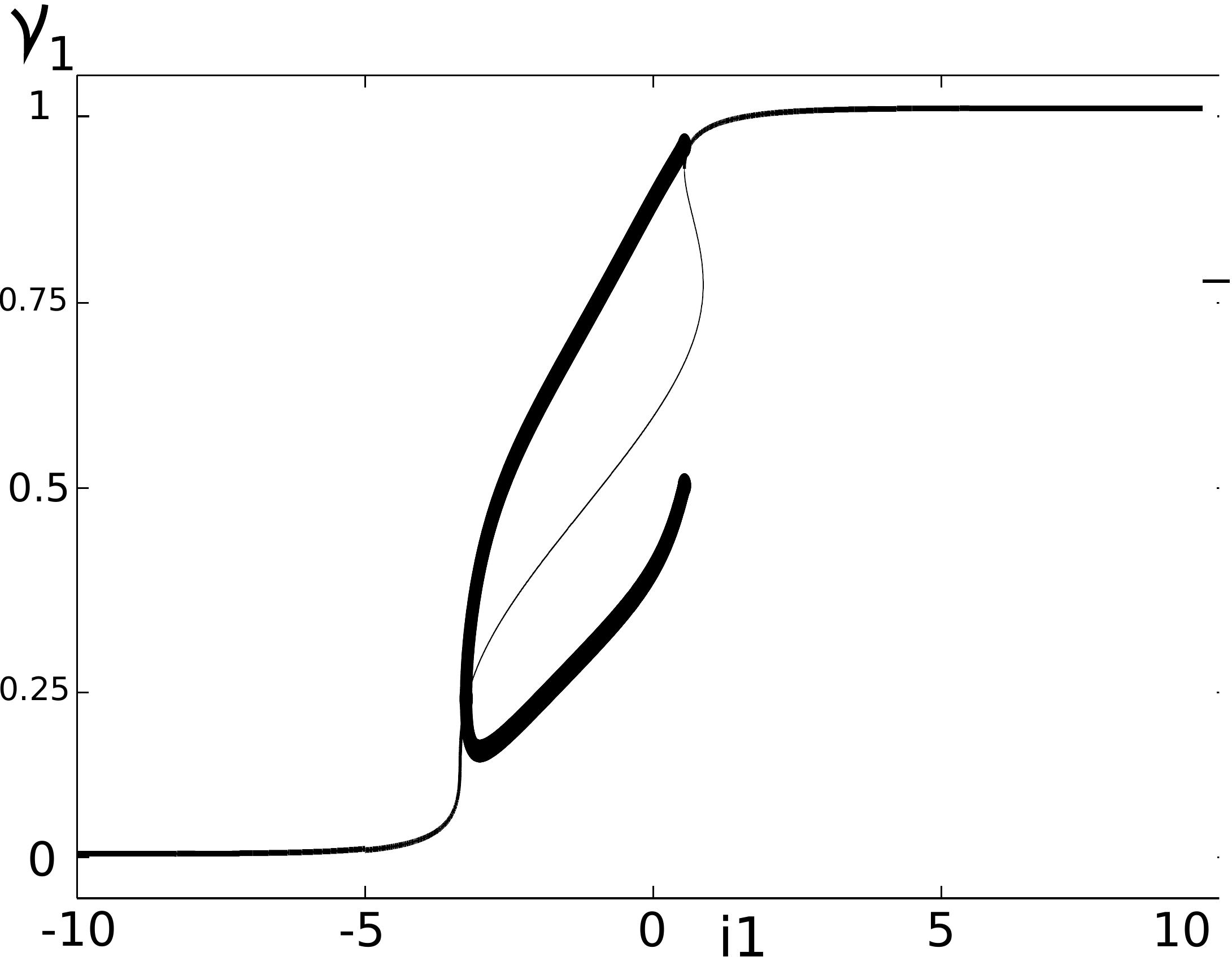}}\quad
     \subfigure[Wilson and Cowan cycle]{\includegraphics[width=.3\textwidth]{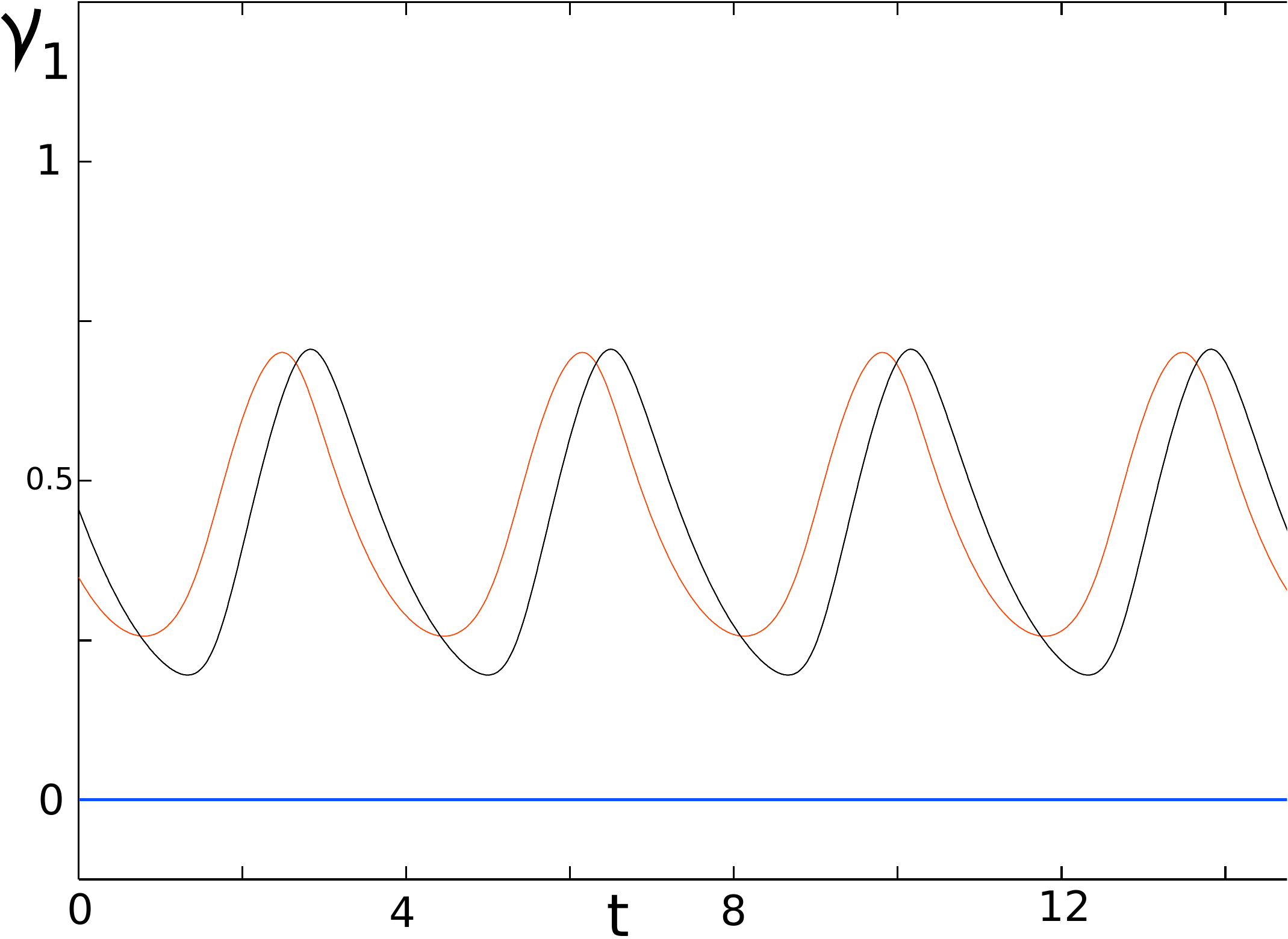}}\quad
     \subfigure[New cycle of the infinite system]{\includegraphics[width=.3\textwidth]{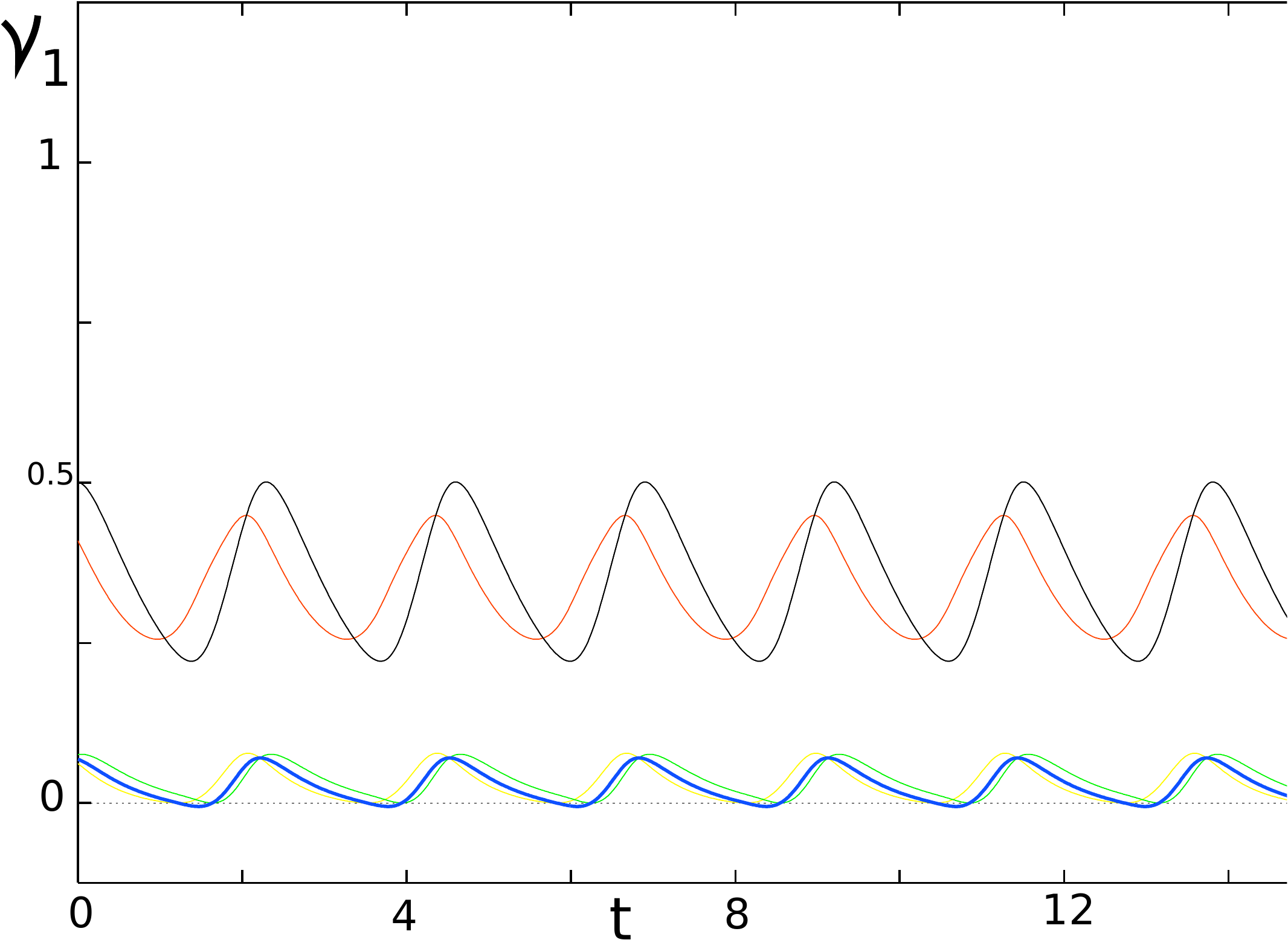}}\\
 		\subfigure[Cycles and trajectories in the plane $(\nu_1, \nu_2)$]{\includegraphics[width=.47\textwidth]{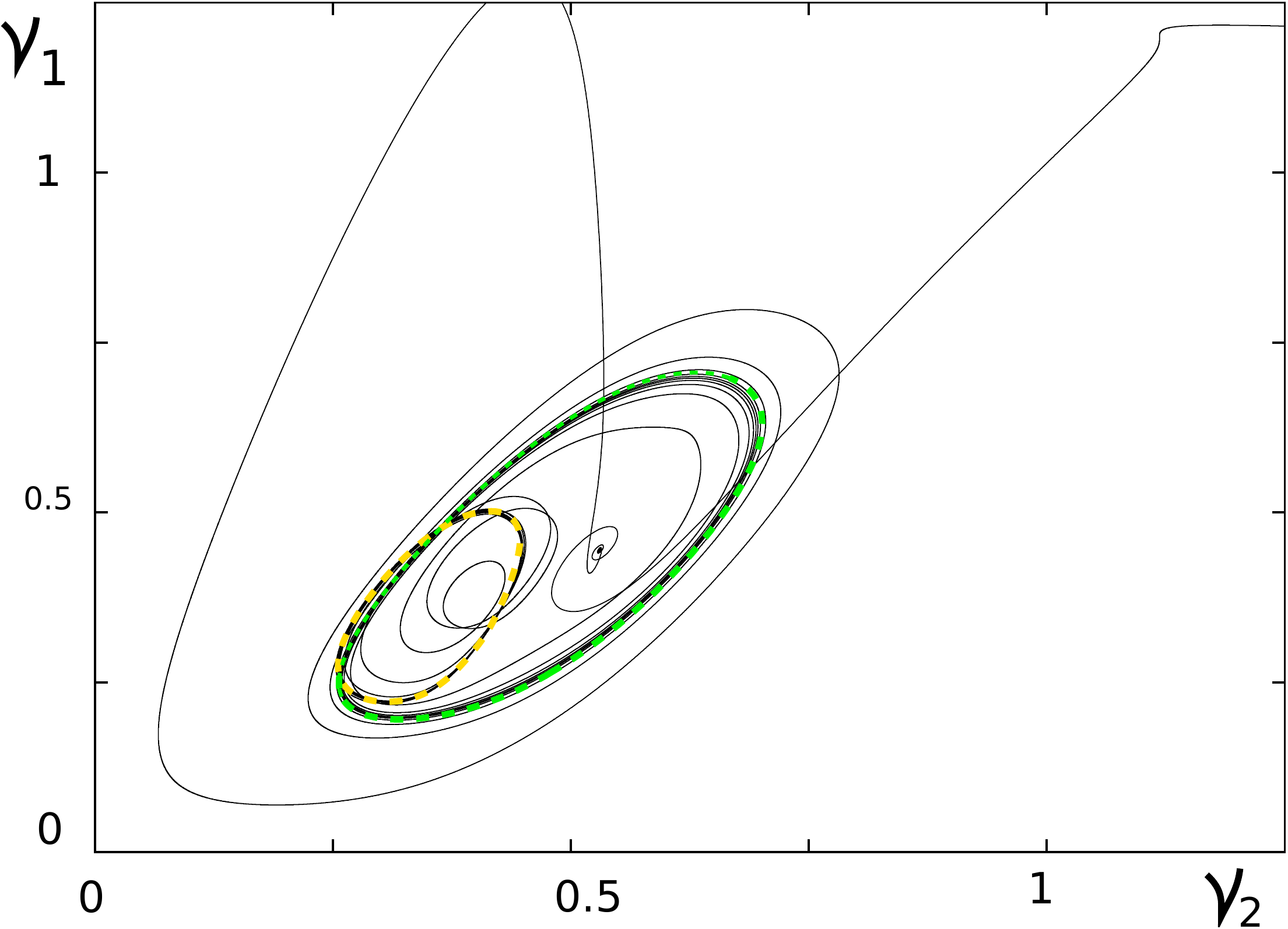}}\quad
 		\subfigure[Partial bifurcation diagram of the infinite-size system]{\includegraphics[width=.47\textwidth]{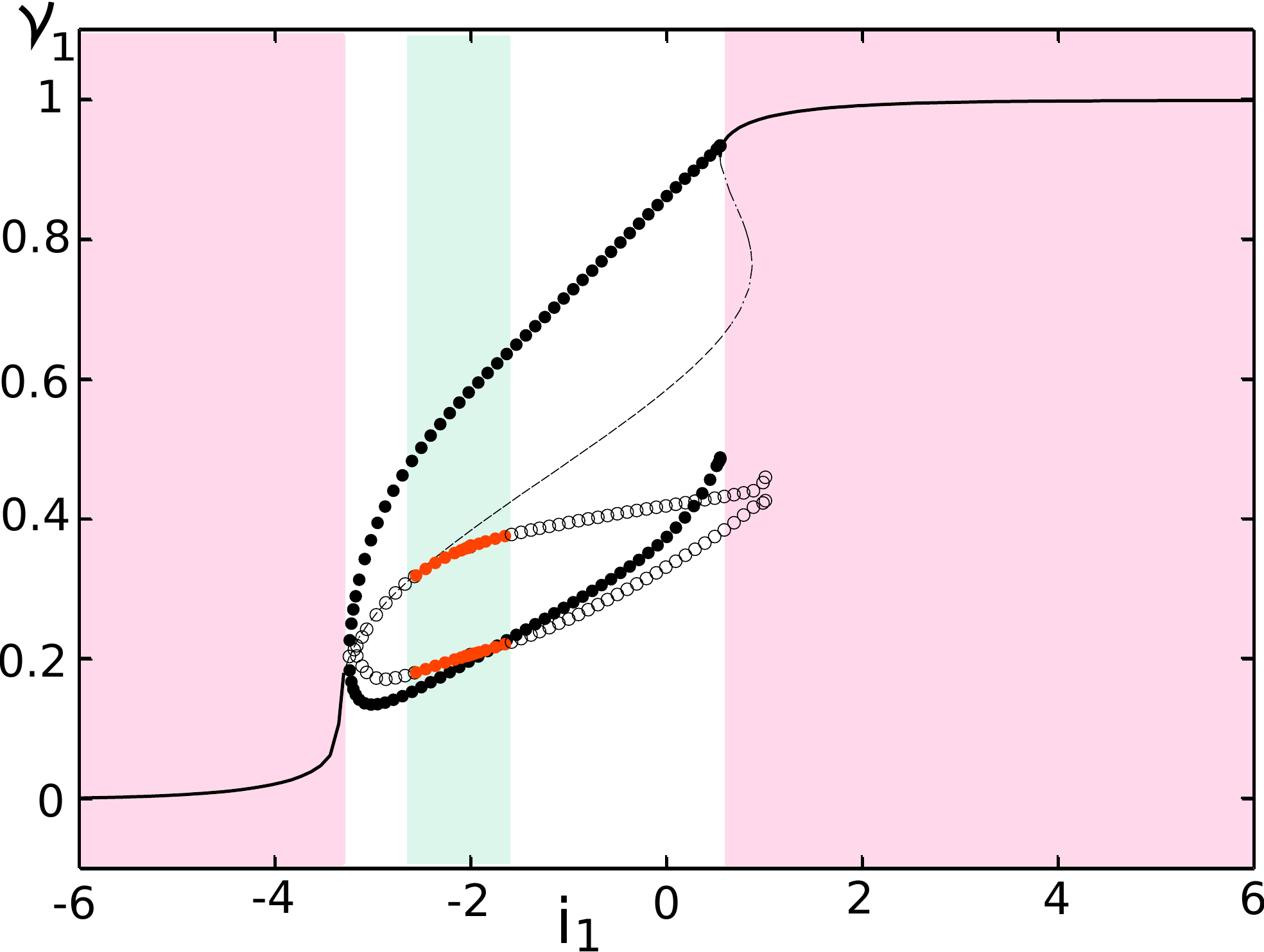}}\quad
   \caption{Model I: Comparison between WC system and the infinite-size system: (a) and (d) are the bifurcation diagram with respect to the parameter $I_1$ of respectively WC and the infinite-size system. The thick lines represent stable fixed point, thin lines unstable fixed points, plain circles the extremal values of a stable cycle and empty circle unstable cycles. (b) represents WC cycle as a solution of the infinite-size system, which shows null correlations. red: $\nu_1$, white: $\nu_2$, yellow: $\Delta_{11}$, green: $\Delta_{22}$, blue: $\Delta_{12}=\Delta_{21}$. For non-zero correlations of the initial condition, the solution converges towards a new cycle with correlactions (c) and (d) which is stable in a determined region of parameters. }
   \label{fig:CyclesMeanField2pops}
 \end{figure}
 
 We conclude with a further observation on transient behavior. We have seen that WC's cycle was neutrally stable, but in some regions of parameters, it has nonlinear stability. In the region of parameters where the WC cycle is nonlinearly stable (white region in figure \ref{fig:CyclesMeanField2pops}(d)), the transient phase of convergence towards this cycle takes a long time. In the green region where it is nonlinearly unstable, the neutral stability of the cycle produce some very strange transient behaviors where the activity of the cortical column is trapped around the neutrally stable WC system for long times, before leaving the neighborhood of this solution and converging towards the only stable solution which is the new cycle.

 \paragraph{Correlation-induced oscillations}\label{sec:second2Pop}
Correlations can have even more dramatic effects than modifying a periodic orbit as shown in Model I, and in this section we describe a case where a periodic orbit arises for parameter values where Wilson and Cowan system only fixed-points. To this end, we  choose a two-population network known to have  bistability. This model consists of two-populations with inhibitory interactions and no self-interaction. The connectivity matrix in that case reads:
 \[\left (
 \begin{array}{cc}
 	w_{11}=0 & w_{12}=-12\\
 	w_{21}=-12 & w_{22}=0
 \end{array}\right).\]
 The inactivation constants $\alpha_i$ are assumed constant equal to $1$, and the inputs to the two populations are distinct and denoted $I_1$ and $I_2$ (hence the system is not fully symmetrical). The firing rate function are also identical and we make the usual choice $f_i(x)=1/(1+\exp(-x))$ for $i\in \{1,2\} $. These parameters define a model called Model II. 
 
 We break the symmetries between the two populations by considering that each population can receive different input, we freeze the input $I_2$ and let $I_1$ vary. In that case, the standard WC system presents a bistable behavior between two stable fixed points as we illustrate in Figure~\ref{fig:Binoc2Pops}. The infinite size system driven by equations \eqref{eq:InfiniteModel} with these parameters has much richer dynamics, as illustrated in Figure~\ref{fig:Binoc2Pops}. We first set  $I_2=0$ and study the bifurcation diagram as $I_1$ varies. We observe that the infinite-size system features a stable limit cycle, which is a qualitatively nontrivial  effect of the correlations in the mean-field limit. All fixed-points behaviors in WC system are preserved, as predicted by theorem \ref{thm:WilsonBuiceCowan} (black curves in Figure~\ref{fig:Binoc2Pops}(b)), but new fixed points with non-zero correlations appear (blue lines in the diagram). Along one of the newly generated branch of fixed point, a supercritical Hopf bifurcation appears, generating a family of stable limit cycles (red circles in Fig.~\ref{fig:Binoc2Pops}(b)) that further undergoes few consecutive period doubling bifurcations. We note that the new branch of fixed points intersects the branch of fixed points of Wilson and Cowan system precisely at the saddle-node bifurcations of Wilson and Cowan system. This fact is compatible with a partial result of the proof of theorem \ref{thm:WilsonBuiceCowan}: at the saddle-node bifurcation, instead of a single eigenvalue equalling zero, two eigenvalues are simultaneously equal to zero, and the saddle-node bifurcations of WC system appear as transcritical bifurcations in the infinite size system. Note that in that case, the WC behavior is nowhere the unique stable behavior of the system, and for any value of the input parameter $I_1$, the infinite system will present attractive behaviors that are different from the WC behaviors. In particular, it has stable fixed points with non-zero correlations and non-zero correlation oscillations. However, we note that though these behaviors exist in the dynamics, the newly discovered fixed points do not constitute acceptable solutions since they are characterized by a negative firing-rate $\nu_2$ and a non-positive definite correlation matrix. However, the cycles constitute acceptable solutions that might have counterparts in the Markov system. 
 \begin{figure}[!h]
 	\centering
 		\subfigure[Wilson and Cowan system]{\includegraphics[width=.45\textwidth]{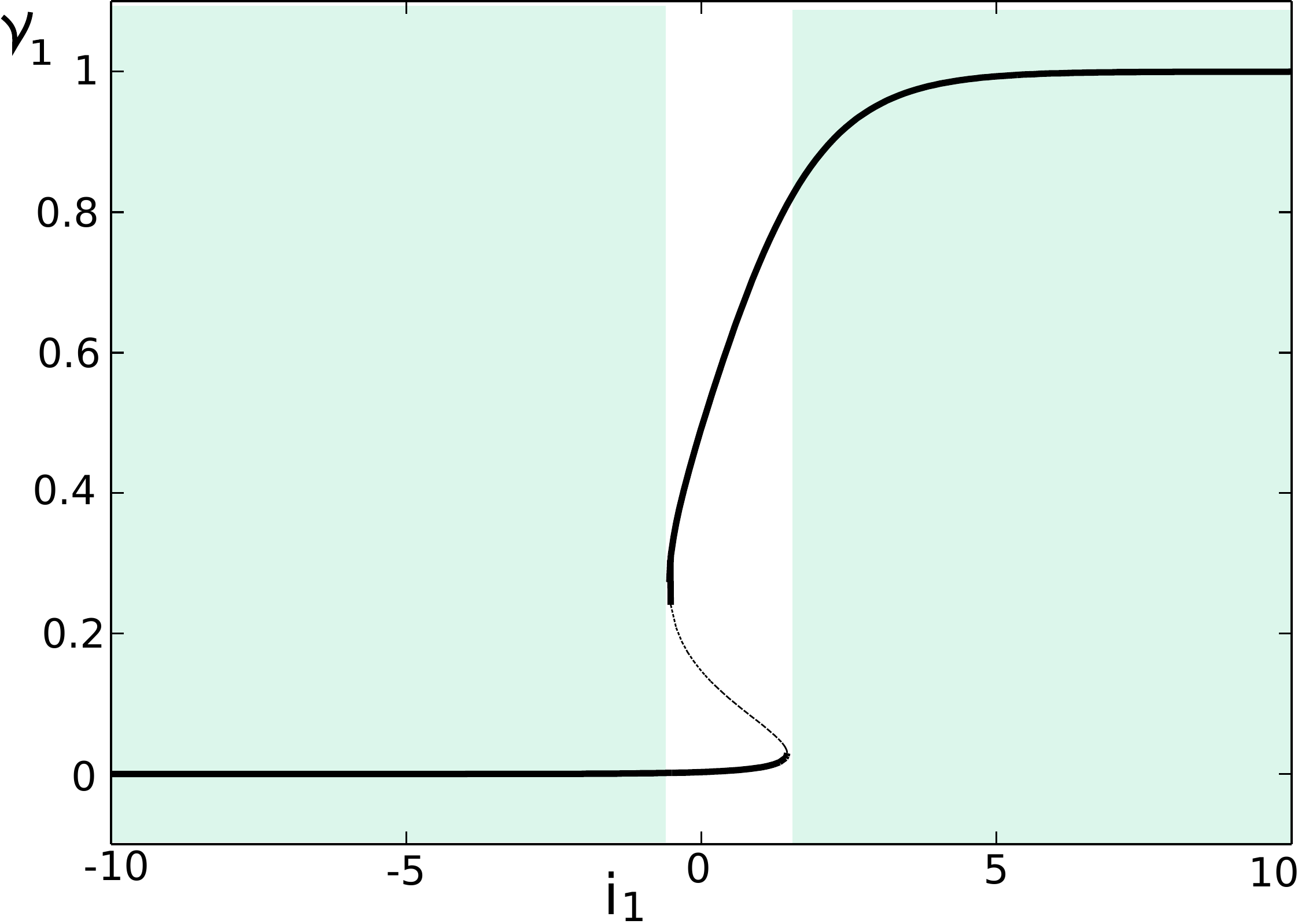}}\quad
 		\subfigure[BCC system]{\includegraphics[width=.45\textwidth]{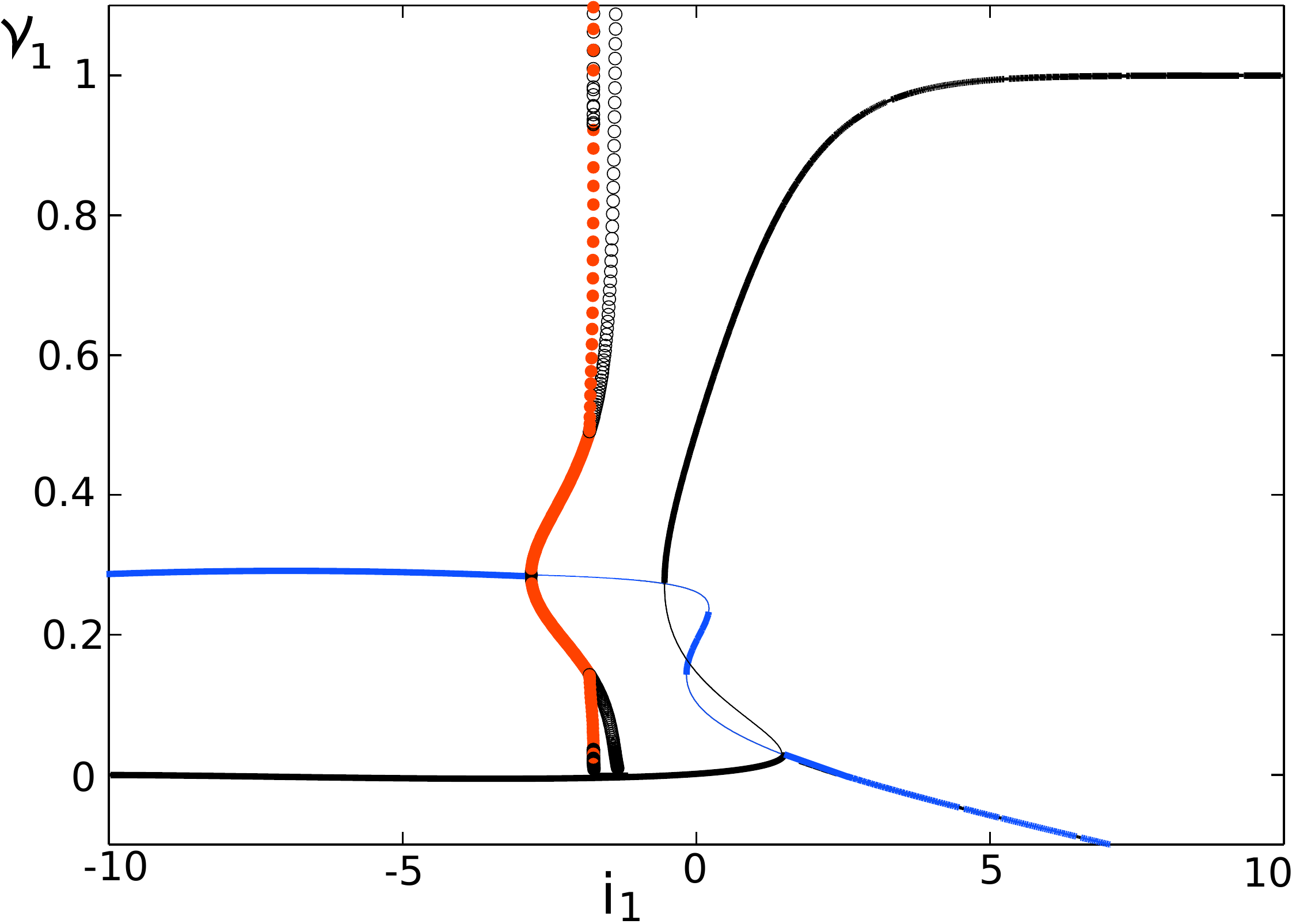}}\\
 		\subfigure[Wilson and Cowan system]{\includegraphics[width=.45\textwidth]{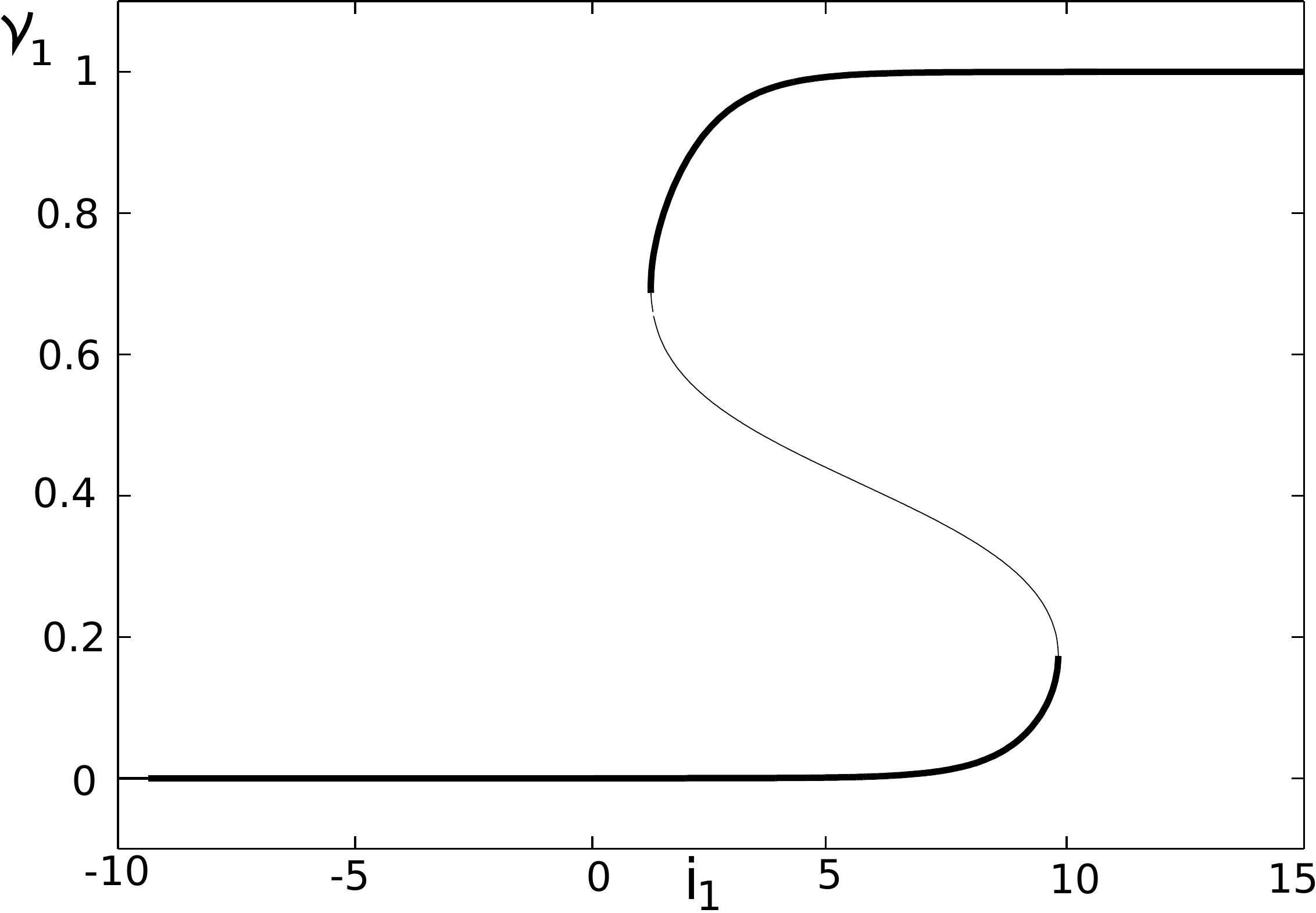}}\quad
 		\subfigure[BCC system]{\includegraphics[width=.45\textwidth]{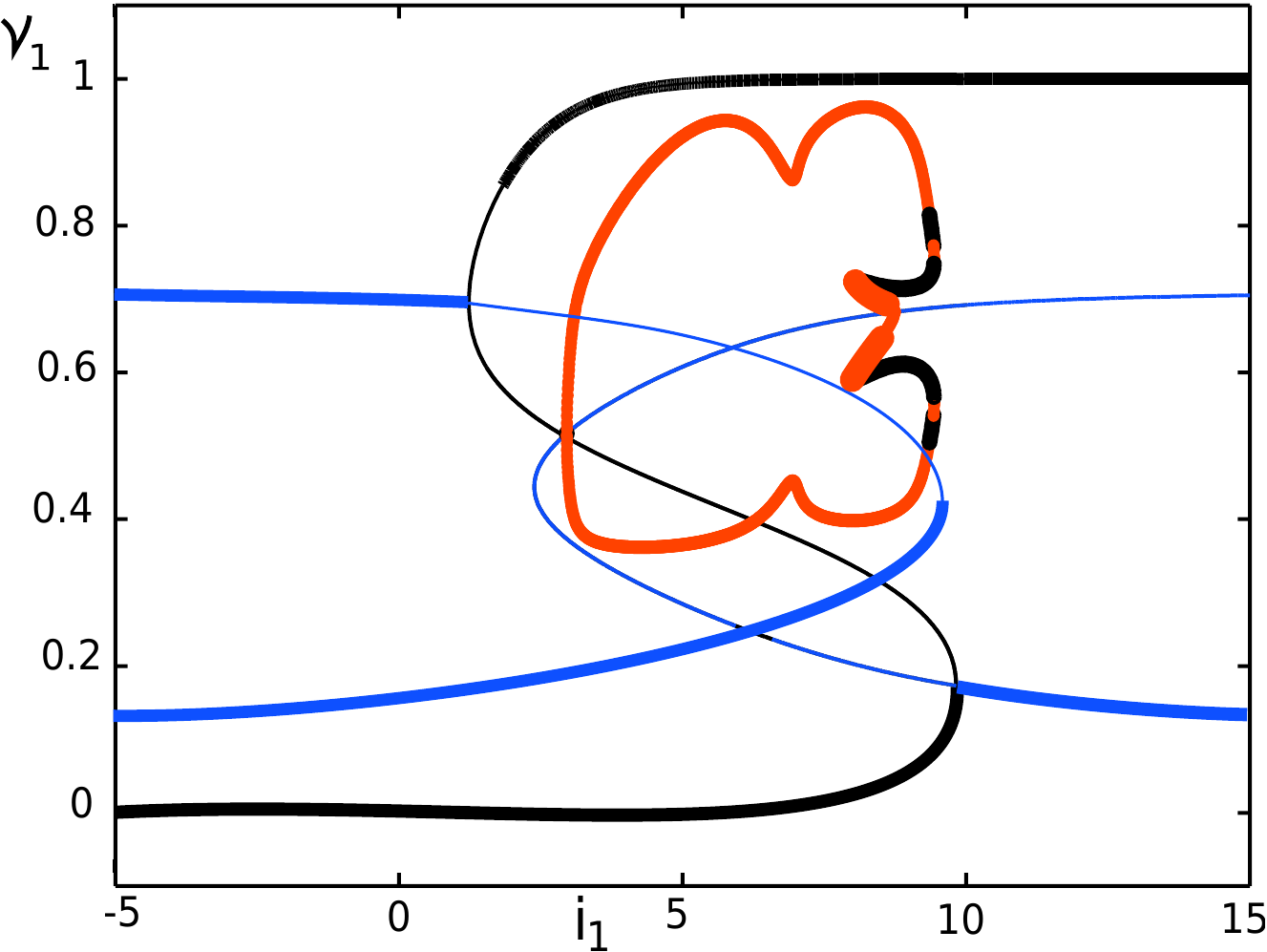}}
 	\caption{Model II, Wilson and Cowan (left column) and infinite-size (right column) systems, with $I_2=0$ (top row) and $I_2=5$ (bottom row).}
 	\label{fig:Binoc2Pops}
 \end{figure}

 We further explore this model by choosing a large value of the input parameter to the second population, $I_2=5$. For this level of input, the bifurcation diagram of WC system presents exactly the same features as in the case $I_2=0$, namely two saddle-node bifurcations and a region of bistability of fixed points. The infinite-size system, similarly to the case $I_2=0$, also has additional branches of fixed points with non-zero correlations that are connected to WC bifurcation diagram at the saddle-node bifurcation points. These saddle-nodes are degenerate and non-generic in the infinite size system from lemma \ref{lem:CKron} and become transcritical bifurcations. These additional non-zero correlation fixed points can be stable or unstable, depending on the parameter, and add additional possible behaviors. From one of these branches, the system undergoes a supercritical Hopf bifurcation and a family of stable limit cycles appears. This branch of limit cycles undergoes two period-doubling bifurcations, torus bifurcations and two fold of limit cycles, yielding additional oscillatory behaviors to the purely fixed-point structure of WC system. 
 
 \medskip
 
 From all these observations on both models, we conclude that the infinite-size system, besides featuring the same solutions as Wilson and Cowan system as a subset of solutions with zero correlations, also has additional behaviors that qualitatively differ from the Wilson and Cowan system. The mean-field limit including correlations is therefore a more complex system than WC system, in which correlations of the firing activity interact in a non-trivial way with the mean firing rate yielding complex behaviors. These qualitative correlation-induced behaviors can therefore be a good indicator of which model better accounts for the behavior of cortical columns and cortical areas \review{if these behaviors hold in the finite-size systems which are equivalent to the moment equations of the initial Markov chain.} 
 
We now turn to investigating finite-size effects arising from the finiteness of the number of neurons in a network.

 \section{Finite-size effects}\label{sec:FiniteSize}
The bifurcation diagram displayed by the infinite-size system studied in the previous section is composed of some very degenerate points, and displays behaviors the WC system does not display.
The question that naturally arises at this point is to understand whether the infinite-size model or WC model really capture the qualitative behavior of large networks in the Markovian framework \review{through the} BCC and Bressloff finite-size equations.

In this section we first study how the infinite-size bifurcation diagram is unfolded in finite-size networks, before simulating the Markov process and comparing its global behaviors to the behaviors of the finite-size systems. We observed that the finite-size BCC and Bressloff equations appeared as a perturbation of the infinite-size system, in which the size of the network appears as a parameter of the equations. In all of this study, instead of considering discrete population sizes, we will consider the inverse of the total population size $n=1/N$ as a continuous parameter. Non-trivial distinctions between finite-size networks and infinite size equation will essentially appear at the bifurcation points. Indeed, since both the rescaled Bressloff and BCC models vector fields are continuous with respect to the parameter $n$, when the infinite-size system ($n=0$) exhibits an hyperbolic fixed point, this fixed point will have a counterpart with the same stability for finite, sufficiently large networks, by application of the implicit functions theorem. This is also the case of hyperbolic fixed points of WC system, since these are hyperbolic fixed points of the infinite-size system with zero correlations. In order to study similarities and qualitative differences between finite-size networks and the infinite size limit, we will therefore focus on the bifurcations of both systems and numerically study the codimension two bifurcations in the one and two populations cases, with respect to the size of the network and to another parameter of interest (for instance an input parameter). 
 
 \subsection{Analysis of single-population finite-size networks}
 We start by investigating a single-population case. In the mean-field limit, we have shown that the infinite size system essentially exhibits the same behaviors as the WC model in terms of fixed points and limit cycles. But is this still the case where the number of neurons is finite?
 
 At the bifurcations of the WC system, singular phenomena can occur. Consider for instance a saddle-node bifurcation point of WC system. At this point, the infinite size system will present a double zero eigenvalue, and the unfolding of this bifurcation can lead to different cases. As we show in this section, it can either unfold into a generic saddle-node bifurcation as we show in the first section, or it can generate an oscillatory behavior through a limit cycle, creating in the finite-size system oscillations which  do not exist in the WC or in the infinite size one population equations.
 
 \subsubsection{Regular unfolding of the infinite-size system}\label{sec:OnePop}
We start by considering BCC model with a strictly positive voltage to rate function $f$, and focus on acceptable solutions that have positive firing rates and correlations. In that case, we numerically show that there is no qualitative difference between the stable solutions of the finite-size BCC network on one hand, and the infinite size system and WC system on the other hand (we showed that both models present the same qualitative behaviors in the one population case). We address this problem numerically, with an inactivation constant $\alpha=1$ (without loss of generality, since modifying this value only amounts rescaling time), $w=10$, and letting both the size of the network $n$ and the input parameter $I$ free. The activity to rate function is chosen to be $f(x)=1/(1+\exp(-x))$. 
 
With these parameters, WC system presents two saddle-node bifurcations when the input parameter $I$ varies, as shown in Figure~\ref{sfig:WC1pop}. The system therefore presents an interval of values of the parameter $I$ where the system presents two stable fixed points and an unstable fixed point, and outside this interval, the system has a unique fixed point which is stable. Similarly to what we showed in the previous section, the infinite-size network presents exactly the same features, and no cycle. 
 
 We now study BCC finite-size equations. We set the size of the network to be $N=50$. BCC finite size equations present the same qualitative behaviors, and the same type of bifurcations: the system undergoes two saddle-node bifurcations as the input parameter value $I$ varies, and a bounded interval of input values $I$ where the system presents a bistable behavior. The only qualitative distinction between finite-size networks and infinite size networks is the existence of a two unstable fixed points in the bistable zone and the fact that the system presents two distinct branches of fixed points, but these differences only affect unstable fixed points and do not produce generic qualitative differences in the behaviors. 
 
 The way this bifurcation diagram is transformed into the bifurcation diagram of the infinite system is very simple. The two distinct branches of fixed point get increasingly close, and in the limit $n=0$, the two branches are superimposed. For finite-size networks, we observe that the value of the correlations of the stable fixed points was non-zero (see Fig.~\ref{fig:1popcodim2BC}). These values were nevertheless quite small, and decrease towards zero as the size of the network as the number of neurons increases. When continuing the two identified saddle-node bifurcations as the network size increases, we observe that the bifurcation curve persists until the limit $n=0$ with no singularity. At the limit however, the non-zero eigenvalue of the Jacobian matrix at the saddle-node bifurcation point tangentially reaches the value $0$ yielding the singularity predicted in proposition \ref{prop:onepop}. 
 
 \begin{figure}[!h]
 	\begin{center}
 		\subfigure[Wilson and Cowan bifurcations]{\includegraphics[width=.4\textwidth]{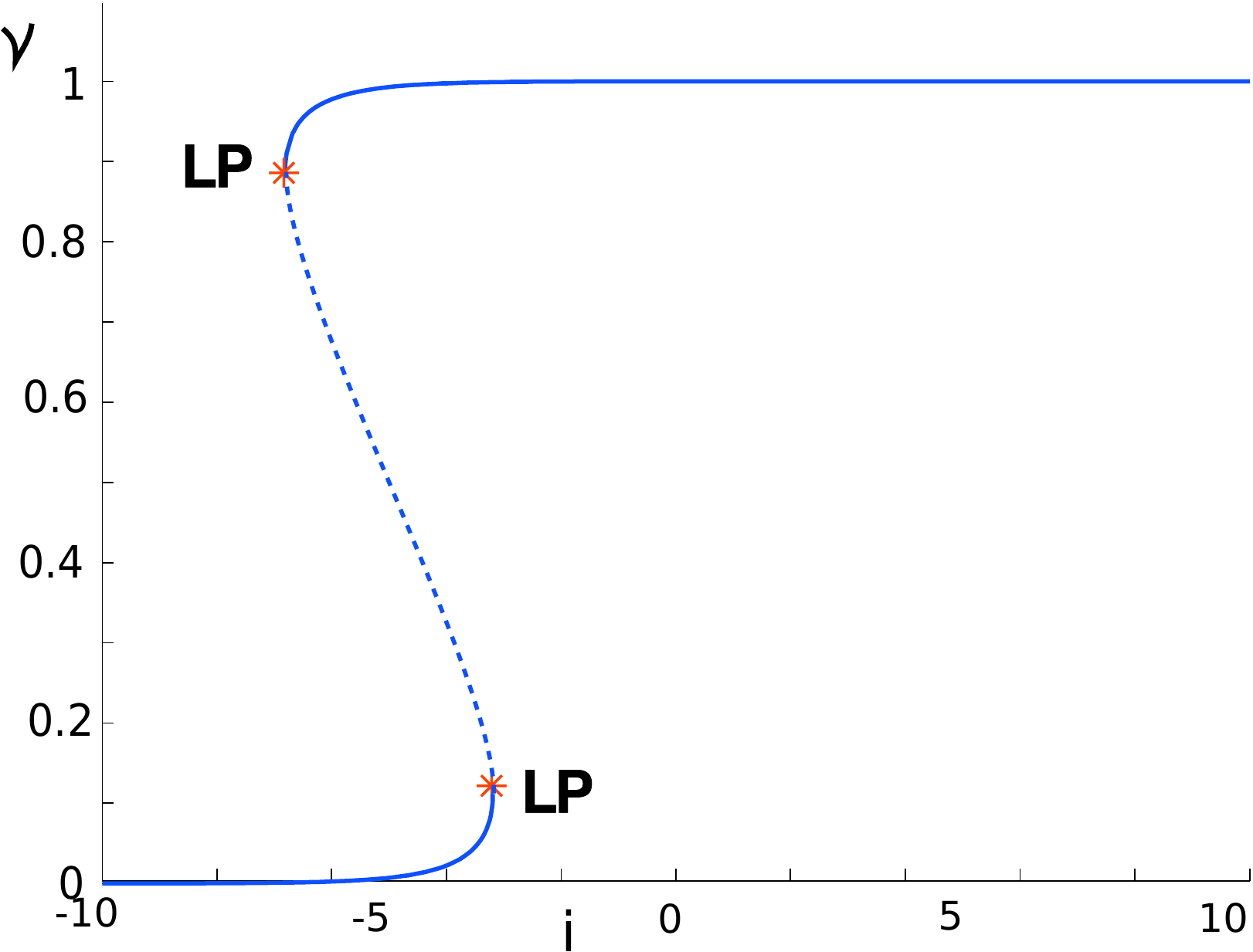} \label{sfig:WC1pop}}\quad
 		\subfigure[BCC bifurcations (firing rate)]{\includegraphics[width=.4\textwidth]{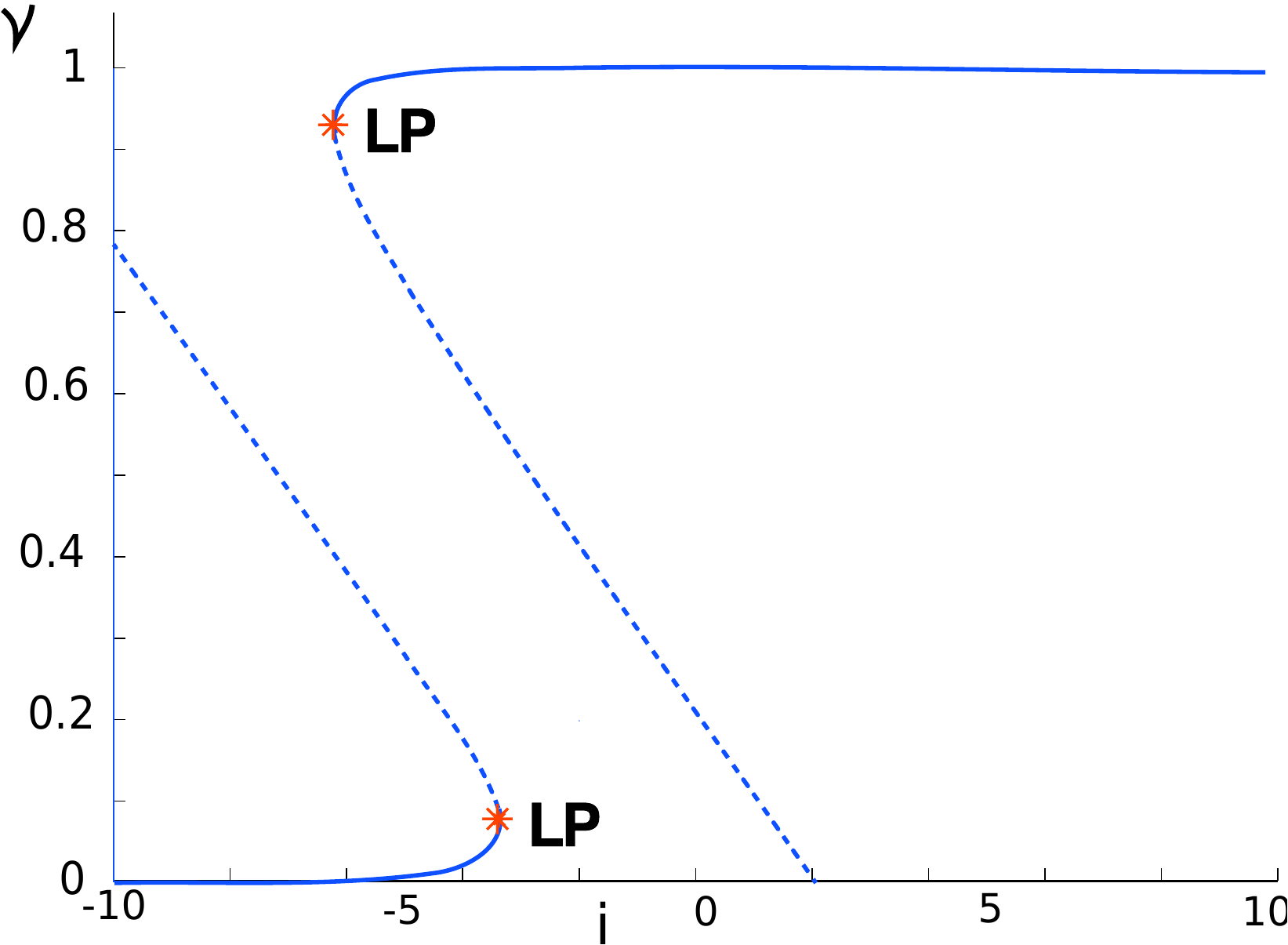}}\\
 		\subfigure[BCC bifurcations (Correlations)]{\includegraphics[width=.4\textwidth]{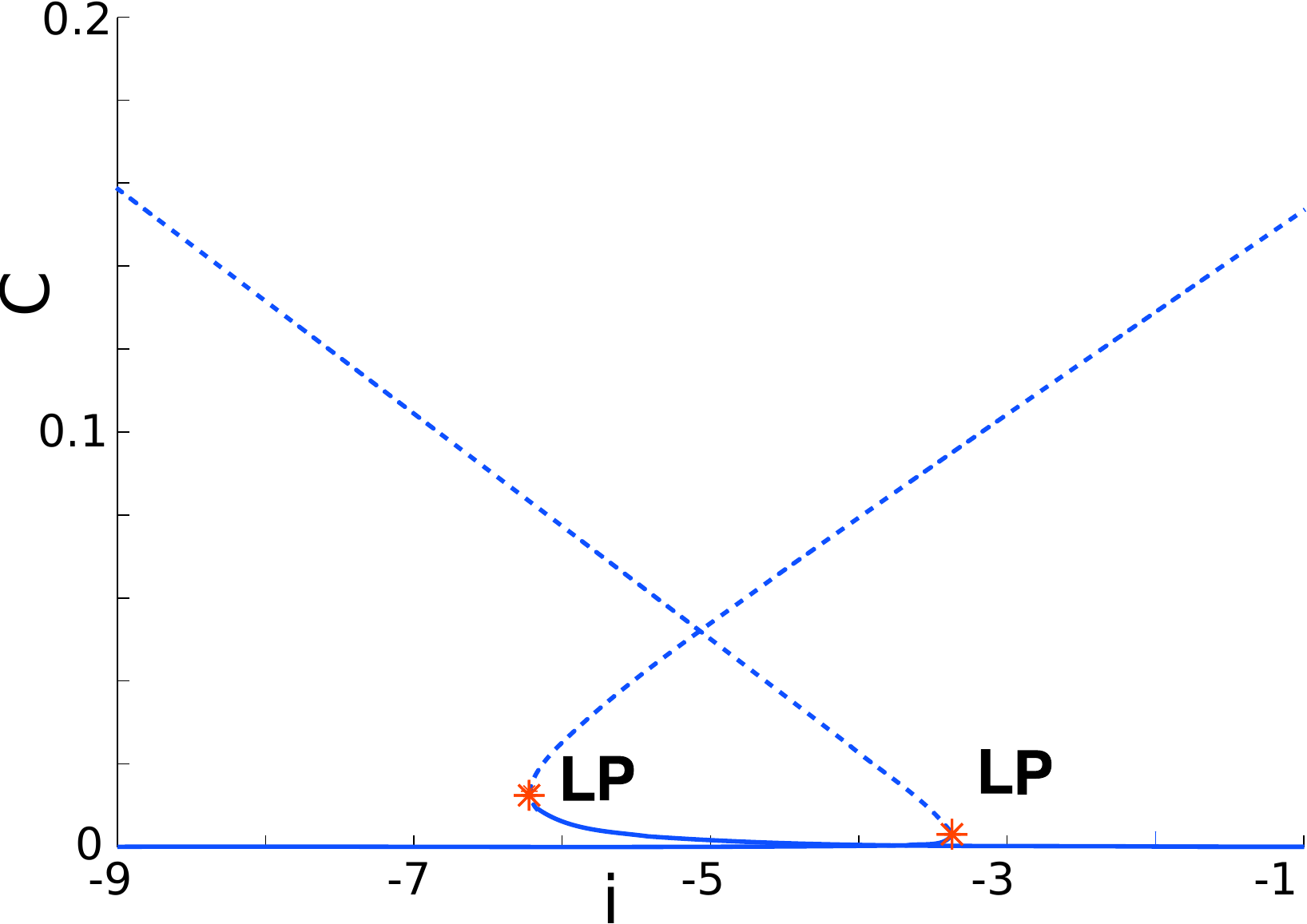}}\quad
 		\subfigure[Codimension two bifurcations]{\includegraphics[width=.4\textwidth]{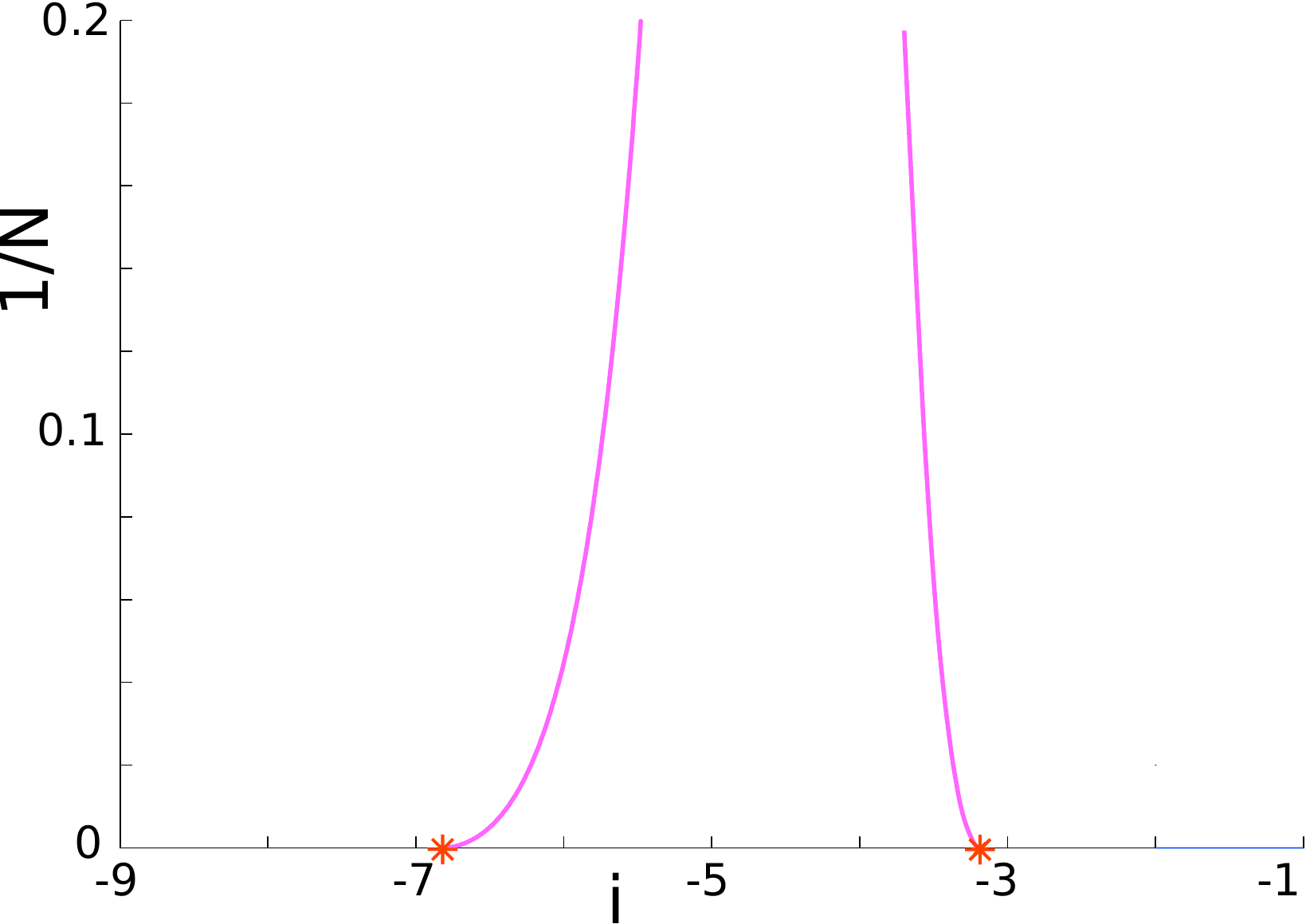}}\\
 	\end{center}
 	\caption{BCC model and the infinite-size system (a) Bifurcation diagram of WC system with respect to the input $I$. Plain curve: stable fixed point, dotted curve: unstable fixed points, LP stands for limit point (or saddle-node bifurcations) (b) and (c) Bifurcation diagram of BCC model with respect to $I$, for $N=50$ neurons. (b) Firing rates, (c) Correlations. (d) Codimension two bifurcation diagram with respect to $I$ and the network size $n=1/N$. Pink line: Saddle-node bifurcations (Limit Points). The star at $n=0$ denote double zero eigenvalue in the Jacobian matrix singular point. The values correspond to the values of WC saddle-node bifurcations.}
 	\label{fig:1popcodim2BC}
 \end{figure}
 
 In this case, we therefore illustrated the fact that BCC system can smoothly unfold the WC system by introducing correlations. These correlations vanish as the size of the population increases, yielding a purely Poisson firing with firing rates given by the solution of WC equations. For the same parameters and functions, it is easy to show, following the same steps, that Bressloff rescaled model also smoothly unfolds the WC system and that the finite-size system presents exactly the same type of bifurcations with two separated branches of fixed points each of which collapse in the limit where the number of neurons is infinite (not shown), and the value of the correlations of the fixed points decreases towards zero when the number of neurons increases. However in this case, one has to bear on mind that the interpretation of the result slightly differs: the limit behavior is now seen as a purely asynchronous state where all the neurons of the population fire independently in the limit $N\to\infty$. These behaviors will closely match the Markovian model's behavior as shown in section \ref{sec:OnePop}. In these cases, WC system conveniently summarizes the behavior of large networks, and provides a simple model to study large networks. However, there exist models for which this conclusion does not hold, and where the infinite-size system's behavior unfolds into more complex behaviors.
 
 \subsubsection{Singular unfolding of the infinite-size system}\label{ssect:OnePopHopf}
 We now turn to study the possible singular unfolding of the infinite size system in one population. To this purpose, we will be specifically interested in studying oscillatory behaviors in the finite-size systems that are not present in the limit where the number of neurons tends to infinity. We will call these qualitative distinctions between finite-size and infinite-size systems singular unfolding. 
 
 To this purpose, we choose a framework where fixed points are easily described, and along the fixed point search for sufficient conditions to get a supercritical Hopf bifurcation, yielding oscillatory behaviors in single population networks that are neither present in the infinite-size system nor in the WC system. To fix ideas, similarly to the previous case treated, we consider BCC one population model. We consider here a non-physiological cases by allowing the firing rate variable to be negative. We choose an activation function $f$ such that $f(I)=0$ for some fixed $I$, which can always be done with an unspecified sigmoid $\tilde{f}$ through the modification $f(x)=\tilde{f}(x) - \tilde{f}(I)$. This quite unrealistic case will allow studying the qualitative differences that can appear between the finite and infinite size systems. However, the fixed point $0$ is quite particular, and we will see that this phenomenon is not generalized in a simple way when considering positive activation functions $f$ and positive firing rates (the affine transform $\tilde{f}(x)-\tilde{f}(I)$ does not have a trivial effect on the nonlinear system). 
 
 The BCC single population model reads:
 \[\begin{cases}
   \der{\nu}{t} &= -\alpha \nu +f(w\,\nu+I) + \frac 1 2 f''(w\,\nu+I) \,w^2 C\\
   \der{C}{t} &= -2\alpha C + 2 f'(w\,\nu+I)w\,(C + \frac 1 N \nu)
 \end{cases}\]
 Therefore, under the assumptions made, it is easy to see that the null solution $\nu=0,C=0$ is always solution of the system. The Jacobian matrix $J$ at this point reads:
 \[
 \left (
 \begin{array}{cc}
   -\alpha + w\,f'(s) & \frac 1 2 f''(s) w^2\\
   2f'(s)\frac{w}{N} & 2\,(-\alpha + w\,f'(s))
 \end{array}
 \right),
 \]
 where we denoted $s=w\,\nu+I=I$ for $\nu=0$. The trace of the Jacobian matrix at this point is equal to $3(-\alpha + w\,f'(s))$ and the determinant is equal to $2(-\alpha + w\,f'(s))^2-f'(s)f''(s)w^3/N $. The trace vanishes at the null fixed point for $(\alpha, w)$ such that 
 \[\alpha=w f'(I),\] 
 and therefore there exist a set of parameters for which the trace vanishes. In order to ensure that this null trace corresponds to the presence of a pair of purely imaginary eigenvalues, we further need to ensure that the determinant is strictly positive. When the trace vanishes, the determinant simply reads $-f'(s)f''(s)w^3/N$. Since $f$ is a sigmoidal function, $f'$ is always positive, and therefore the determinant is positive as soon as $f''(s)<0$, which can be ensured for some specific firing rate functions. In that case, we will denote $\omega_0=\sqrt{-f'(s)f''(s)w^3/N}$. 
 
 Therefore, under the assumptions that there exists $I$ such that $f(I)=0$ and $f''(I)<0$, there exist parameters $\alpha$ and $w$ such that the Jacobian matrix of the system at the null solution has a pair of purely imaginary eigenvalue. Let us now check that this point corresponds to a Hopf bifurcation when varying the parameter $\alpha$. To this purpose, we check transversality and genericity conditions of the Hopf bifurcation at this point, as given for instance in \cite{kuznetsov:98}. These calculations are provided in appendix \ref{append:Calculations}.  The transversality condition is easily checked. It consists in showing that the differential real part of the eigenvalues of the Jacobian matrix at this point, with respect to the parameter $\alpha$, does not vanishes. Let us denote by $\mu(\alpha)$ the real part of the eigenvalues. Since we are in a planar system, we have:
 \begin{align*}
 	\mu(\alpha) &=\frac 1 2 Tr(J)\\
 	& = \frac 3 2 (-\alpha + w\,f'(s))\\
 	\der{\mu(\alpha)}{\alpha} &= -\frac 3 2 < 0
 \end{align*}
 Therefore the transversality condition of the Hopf bifurcation is satisfied at this point. In order to fully state that the system undergoes a Hopf bifurcation, we need to show that the first Lyapunov exponent of the system at this point does not vanish. This Lyapunov exponent, after some tedious calculations using the formula given in \cite{kuznetsov:98} (see appendix \ref{append:Calculations}), is shown to be of the same sign as:
 \[\frac{w^2\,N}{f'(I)^2}\left [ f^{(3)}(I)f'(I) \left(1+\frac {1}{\omega_0}\right) + f''(I)^2\left(\frac{2}{\omega_0} - \frac{14}{3} \right)\right].\]
 The sign of this expression is governed by the values of the differentials of $f$ at the point $I$. If the expression in the bracket is negative, then the Hopf bifurcation is supercritical and the system has stable oscillations. Let us now investigate the dependence of this cycle on the network size. We have shown that $\omega_0=\sqrt{-f'(s)f''(s)w^3/N}$. Therefore, if this cycle appears for some finite network, it will lose stability as the size of the population increases. The period of the generated cycle is proportional to the inverse of the square root of $N$, and therefore slowly tends to zero as the network size increases (note that the cycle will lose stability as the network size increase, since for sufficiently large populations, the Hopf bifurcation will be subcritical). 
 
 Now that we identified simple conditions for the Hopf bifurcation to take place, we exhibit a network that indeed has this bifurcation. We choose for instance as activation function an hyperbolic tangent function \[f(x)=\tanh(x) - \tanh(I)\]
 with $I=0.5$. In that case, we have:
 \begin{equation*}
 	\begin{cases}
 		f(I)&=0\\
 		f'(I)&\approx 0.786\\
 		f''(I) &\approx -0.727<0\\
 		f'''(I)&\approx -0.56		
 	\end{cases}
 \end{equation*}
 Therefore this function satisfies the assumptions done in the previous paragraph. Moreover, since $f'''(I)<0$ and $(2/\omega_0-\frac{14}{3})<0$ ( for $N$ small enough), the first Lyapunov coefficient is negative, and therefore the system undergoes a supercitical Hopf bifurcation and presents oscillations, as illustrated in Figure~\ref{fig:oscillations1pop}. 
 
 \begin{figure}[!h]
   \begin{center}    \subfigure[Oscillations]{\includegraphics[width=.4\textwidth]{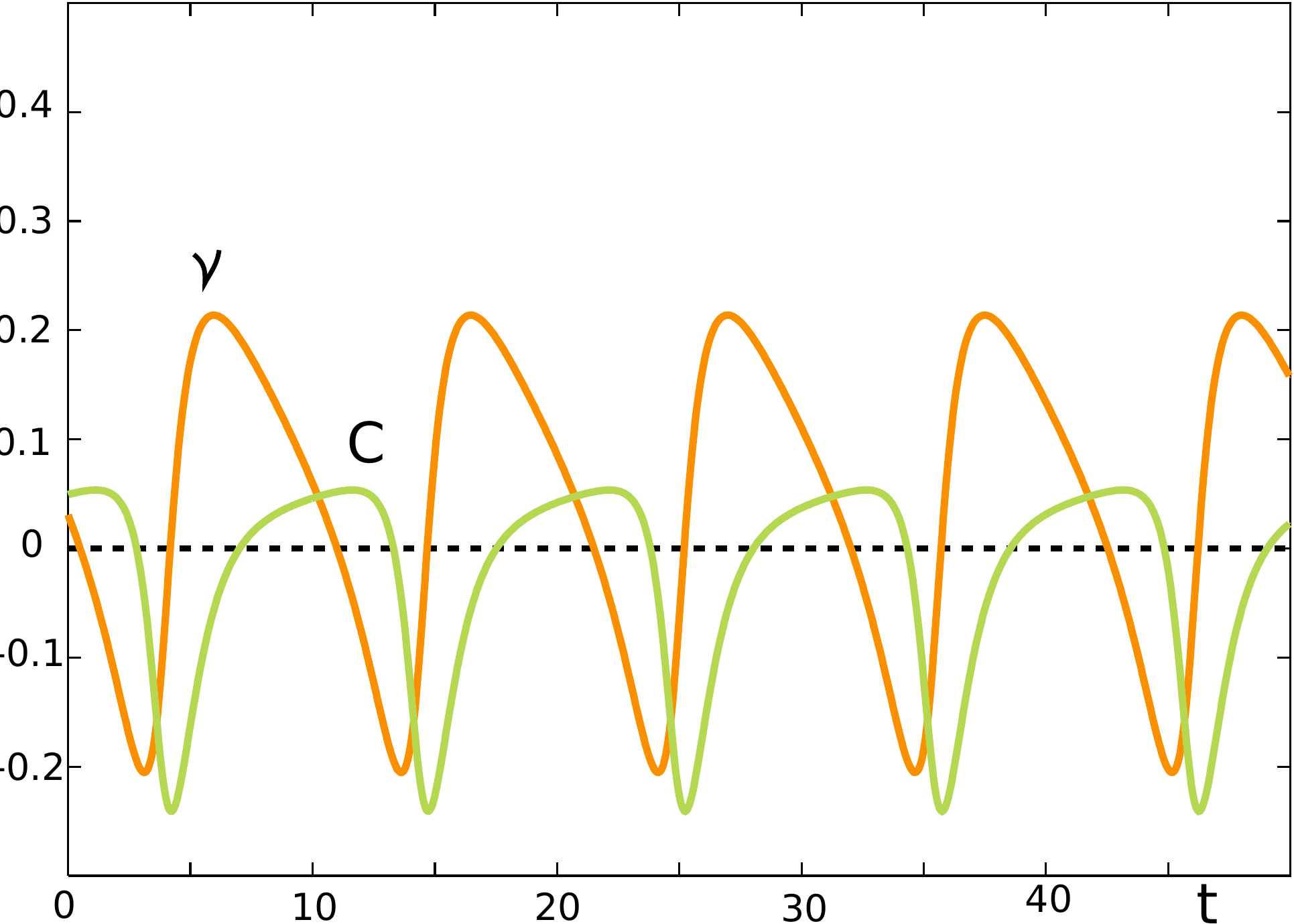}} \quad
     \subfigure[Continuation and Homotopy]{\includegraphics[width=.5\textwidth]{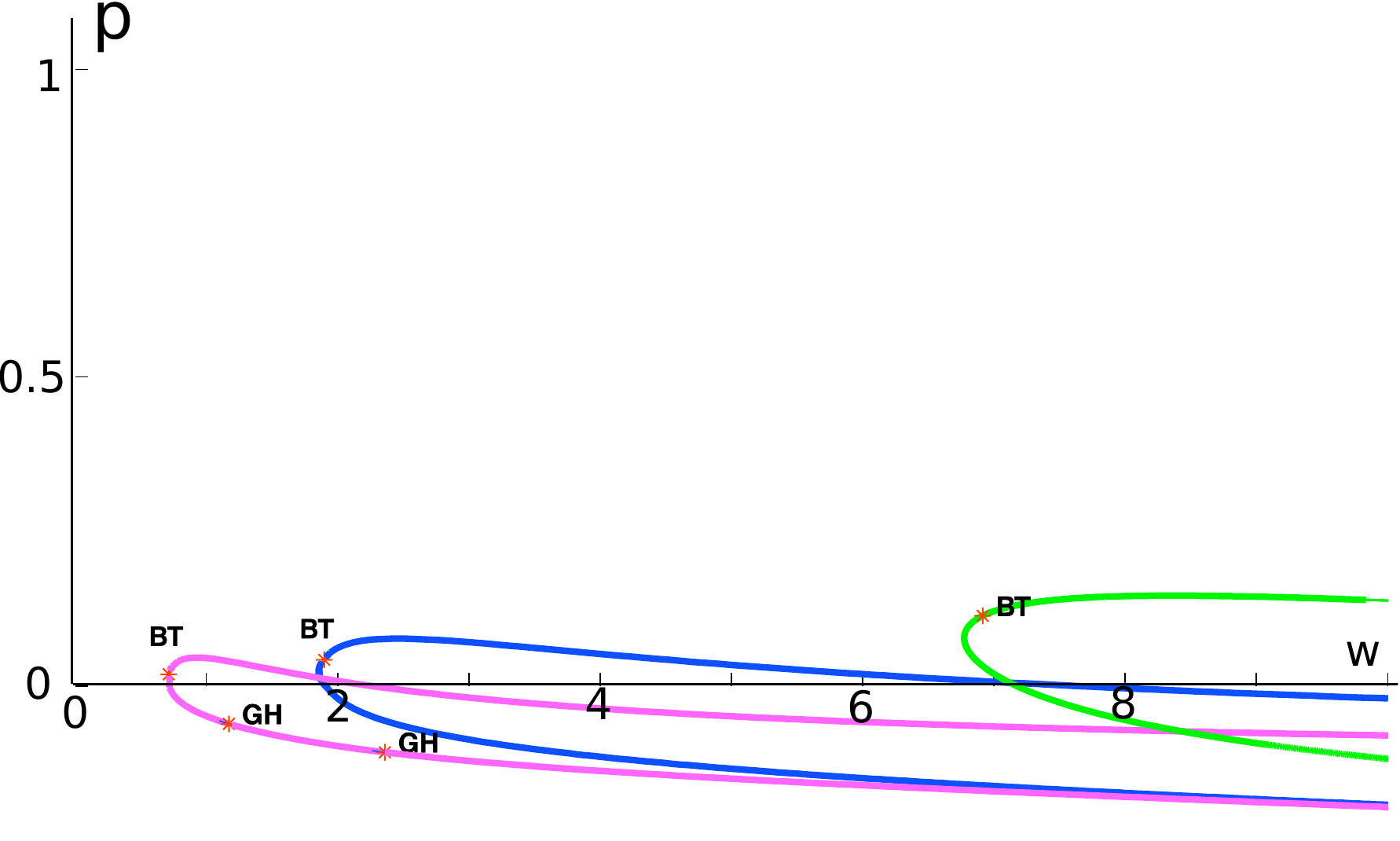}}
   \end{center}
   \caption{Oscillations in a single population network with finite size. (a) is a cycle observe for $f_0$ (p=0), (b) is the codimension two continuation with respect to $w$ and $p$ of the Hopf bifurcation for different values of $\alpha$: pink for $\alpha=.3$, blue: $\alpha=.8$ and green: $\alpha=3$.}
   \label{fig:oscillations1pop}
 \end{figure}
 
 Therefore, we exhibited a case where the finite-size equations have important qualitative differences from the infinite-size equations. However this system involved non-physiological values of the parameters, in particular a non-positive activation function $f$. Therefore, this behavior will never appear in the Markov process that yielded BCC equations. 

\review{The question that then arises is whether such behaviors appear in plausible settings.} In order to answer this question and search for  a system with positive sigmoidal functions that indeed present oscillations in one-population networks, we continued the Hopf bifurcation while continuously transforming the non-positive sigmoidal transform into a positive sigmoid. For instance, we consider $f$ a non-positive sigmoidal function for which the Hopf bifurcation exhibited here appears, and define $f_p = f - p\min(f)$ a transformation of $f$ such that $f_0=f$ and $f_1$ is a positive sigmoid. This homotopic transform smoothly maps the vector field presenting an attractive cycle onto a vector field with a positive activation function. Since we have proved the genericity and transversality conditions on the Hopf bifurcation identified, we are in position to continue it as the parameter $p$ varies together with $\alpha$ in a codimension two bifurcation diagram (by application of the inverse function theorem). We observe in that case that the Hopf bifurcation cannot be continued when varying the homotopy parameter until reaching a positive sigmoid (see Figure~\ref{fig:oscillations1pop}(b) for the example provided above). We have tried numerous other sets of parameters and have never been able to continue periodic orbits to a regime  where there is a non-negative activation function.
 
 Therefore, though the one-population BCC equations can display oscillations for non biologically relevant parameters, we conjecture that there is no acceptable oscillatory solution for positive activation functions.
 
 \subsection{\review{A two-populations finite-size network}}\label{sec:TwoPopsFinite}
 We have seen in the previous section that the one-population finite-size BCC model was accurately approximated by the infinite-size system (and therefore the WC system from the result of a previous section) for biologically realistic parameter, though oscillatory phenomena may occur for non-biologically relevant parameters that allow both the firing-rate and the correlations to be negative. The picture will be quite different in multi-population networks. Indeed, in higher dimensions, the WC system can have  Hopf bifurcations that create very degenerate  bifurcation points in the infinite-size system, which may non-trivially unfold in the finite-size systems. In this section, we consider the two population networks with excitatory self-interaction and negative feedback loop, called Model I, introduced in section \ref{sec:NegFeedback} and address first the case of Bressloff model, before showing that BCC model presents the same features in appendix \ref{appen:2PopsBCFinite}. \review{This study will allow addressing in particular the question of the existence of so-called \emph{quasicycles} and also of aperiodic solutions in the system close to the Hopf bifurcation of the mean-field Wilson and Cowan limit, addressed recently in the same model by Bressloff in~\cite{bressloff:10}. These quasicycles correspond to intrinsic noise-induced oscillations in parameter regimes where the deterministic mean-field limit (in that case Wilson and Cowan system) is below the Hopf bifurcation, and where the system features a stable focus. This phenomenon was also recently found in similar systems arising in biochemical oscillations of cellular systems, see e.g.~\cite{bolland-galla-etal:08,mckane-nagy-etal:07}. Our approach to these phenomena consists in unfolding the Hopf bifurcation of the mean-field system to finite network sizes, in order to understand how the intrinsic noise (that can be seen as through the Langevin approximation of the system valid for large population sizes as used in appendix~\ref{append:RodigTuck} and in~\cite{bressloff:10}) interacts with the Hopf bifurcation.}
 
 To fix ideas, we focus in this section on the particular network introduced in section \ref{sec:NegFeedback}, namely the two-population rescaled Bressloff model. Denoting $c=C$, the equations read: 
 \begin{equation*}
 	\begin{cases}
 		\nu_1'&=-\alpha_1 \nu_1+f(s_1)+\frac{1}{2}\,f''(s_1)\,(w_{11}^2\,c_{11}+w_{12}^2\,c_{22}+2\,w_{12}\,w_{11}\,c_{12})\\
 		\nu_2'&=-\alpha_2 \nu_2+f(s_2)+\frac{1}{2}\,f''(s_2)\,(w_{22}^2\,c_{22}+w_{21}^2\,c_{11}+2\,w_{22}\,w_{21}\,c_{12})\\
 		c_{11}'&=\frac{1}{N_1} [\alpha_1 \nu_1 + f(s_1)]-2\alpha_1 c_{11}+2\,f'(s_{1})\,(w_{11}\,c_{11}+w_{12}\,c_{12})\\
 		c_{22}'&=\frac{1}{N_2} [\alpha_2 \nu_2 + f(s_2)]-2\alpha_2 c_{22}+2\,f'(s_{2})\,(w_{21}\,c_{12}+w_{22}\,c_{22})\\
 		c_{12}'&=-(\alpha_1+\alpha_2) c_{12}+f'(s_{1})\,(w_{11}\,c_{12}+w_{12}\,c_{22}) \\
 		& \qquad + f'(s_2)\,(w_{21}\,c_{11}+w_{22}\,c_{12})\\
 	\end{cases}
 \end{equation*}
 where we denoted:
 \[
 \begin{cases}
 	s_1 &= w_{11}\,\nu_1+w_{12}\,\nu_2+i_1\\
 	s_2 &= w_{21}\,\nu_2+w_{22}\,\nu_2+i_2
 \end{cases}
 \]
 
 For the set of parameters studied in this section and a number of neurons $N=50$, Bressloff finite-size system undergoes four Hopf bifurcations labelled H1 through H4 and six saddle node bifurcations LP1 through LP6. It features two distinct families of limit cycles, one that undergoes two folds of limit cycles LPC1 and LPC2, and one that undergoes a Neimark-Sacker (Torus) bifurcation, as displayed in Figure~\ref{fig:BressBifs02}(b), and are compared with the standard WC system that displays one Hopf bifurcation point H and two saddle node bifurcations LPa and LPb. 
 
 \begin{figure}[!h]
 	\centering
 		\subfigure[Wilson and Cowan system]
 		{\includegraphics[width=.45\textwidth]{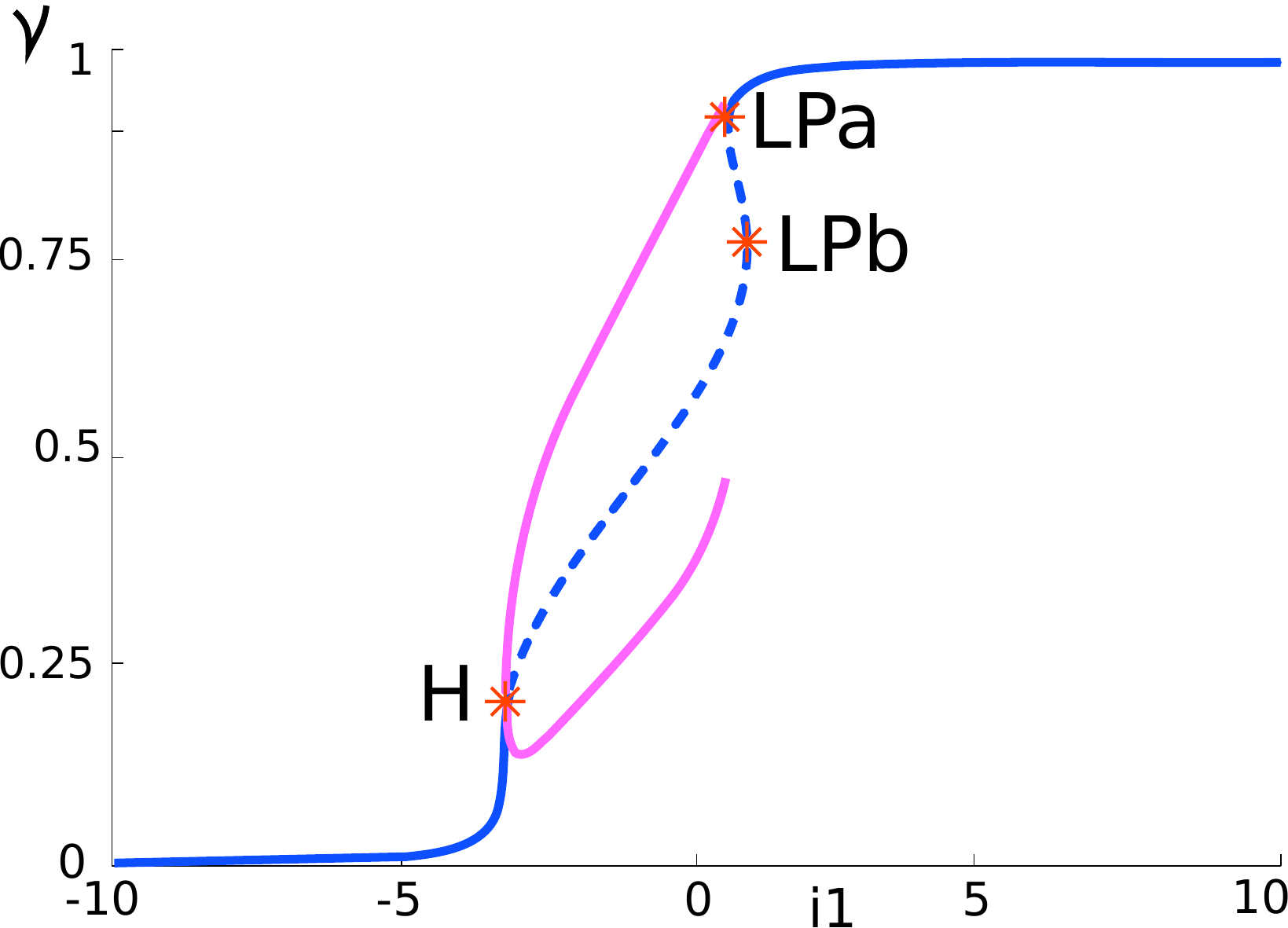}}\qquad
 		\subfigure[Bressloff model ]
 		{\includegraphics[width=.45\textwidth]{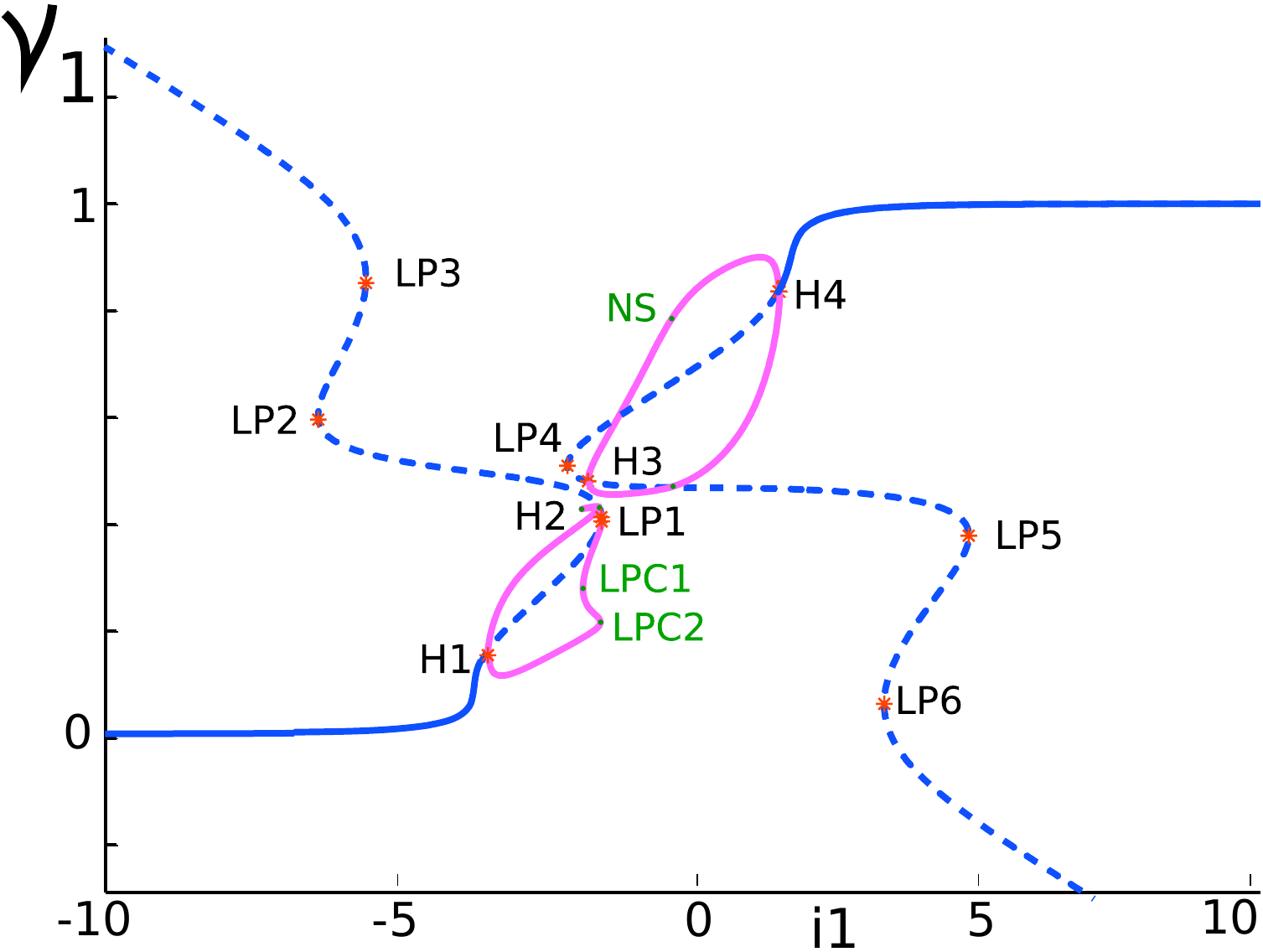}}
 	\caption{Bifurcation diagram of Model I: (a) Wilson and Cowan system and (b) Bressloff system. Blue: equilibria, pink: cycles. Bifurcations of equilibria are denoted with a red star, LP represents a saddle-node bifurcation (Limit Point), H a Hopf bifurcation. The four Hopf bifurcations share two families of limit cycles. One of these undergoes two fold of limit cycles LPC, and the other branch of limit cycle a Neimark Sacker (Torus) bifurcation. }
 	\label{fig:BressBifs02}
 \end{figure}
 
 The analysis of this bifurcation illustrates the fact that the small finite-size system studied have clear qualitative differences with the WC system. However, because of the properties of the vector field, and in particular its smooth dependency with respect to the size of the network $N$, the bifurcation diagram of Fig.~\ref{fig:BressBifs02} smoothly connects to the degenerate bifurcation diagram of the infinite-size system presented in section \ref{sec:NegFeedback}. 
 
Let us now investigate how the bifurcation diagram of the infinite system unfolds into the bifurcation diagram of the finite-size system. To this end, we continue of the bifurcation diagram of figure Fig.~\ref{fig:BressBifs02}(b) as the number of neurons increases, as shown in the codimension two bifurcation diagram of figure Fig.~\ref{fig:Codim2BC2Pops}. The structure of the new bifurcation diagram appear way more intricate than WC's, and in particular, each bifurcation point of WC system, corresponding to degenerate dynamics of the infinite-size system, unfolds non-trivially into different generic bifurcations in the finite-size moment equations. Let us describe the diagram plotted in  figure Fig.~\ref{fig:Codim2BC2Pops}. We observe that the bifurcation diagram of the finite-size system, for any $N<232$ neurons will be similar to the one plotted in figure Fig.~\ref{fig:BressBifs02}. As $N$ increases, the bifurcation diagram is smoothly transformed except at two particular codimension two bifurcation point: a cusp bifurcation at $n\approx 0.0867$, from which emerge two additional saddle-node bifurcations denoted LP7 and LP8. The value of the correlation variables of these points progressively converge to zero as the network size increases. These two saddle-node bifurcations correspond in the infinite-size limit to the two saddle nodes of WC system LPa and LPb, and are not present in the finite-size model for a small number of neurons. The second codimension two bifurcation appearing in the diagram correspond to the disappearance of the Hopf bifurcation point H4 through a Bogdanov-Takens bifurcation with the saddle-node bifurcation manifold corresponding to LP8. Except from these two codimension two bifurcations, the continuation of all other bifurcations is smooth for any $n>0$. But in the limiting case $n=0$, the system is highly singular and different bifurcation manifold collide, as we now discuss in more details.
 \begin{figure}[!h]
 	\centering
 		\includegraphics[width=\textwidth]{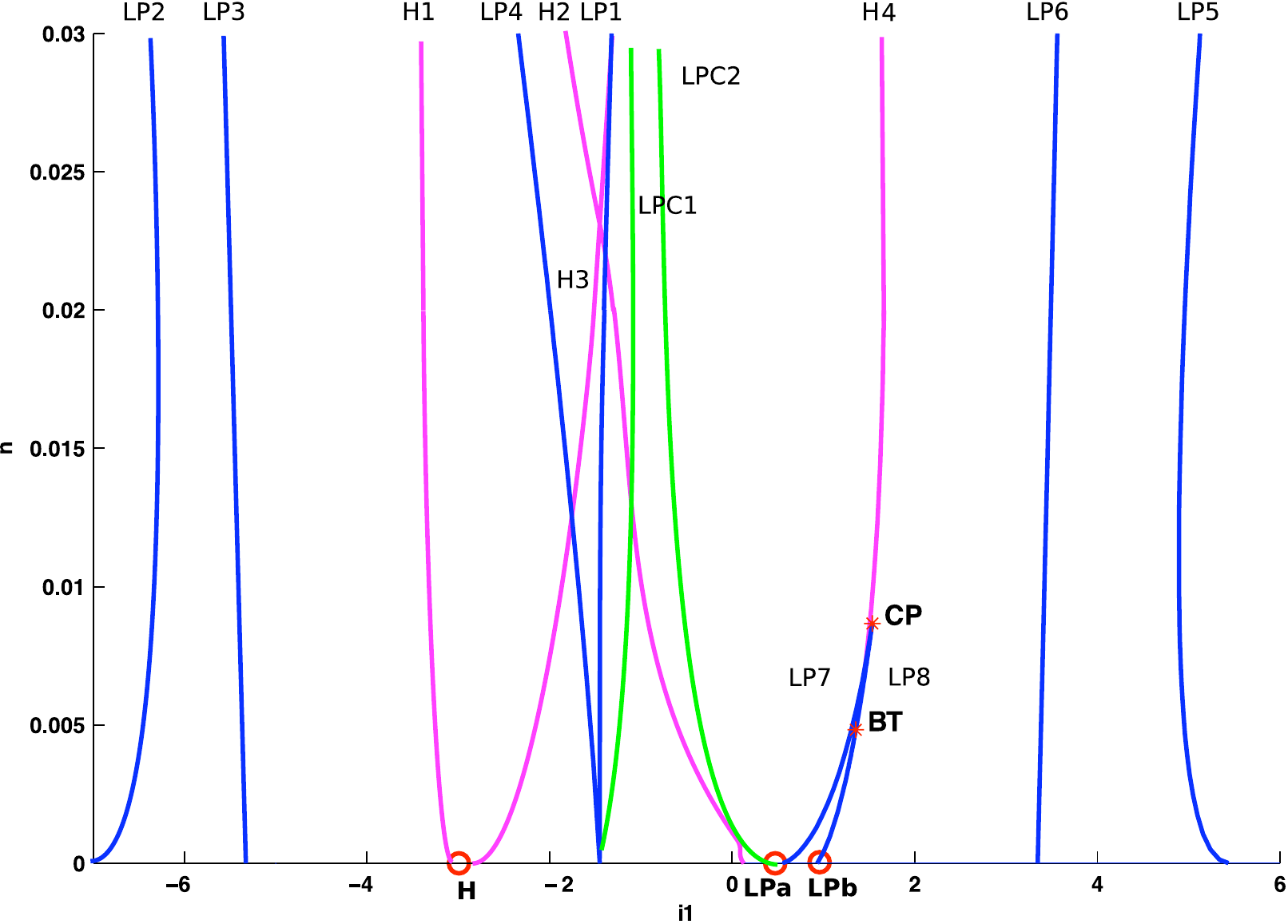}
 	\caption{Codimension two bifurcations  with respect to the input to population 1 and the population size in Bressloff's model. Pink: Hopf bifurcations, blue: saddle node bifurcations, green: folds of limit cycles. The continuation of the Hopf bifurcations H1 and H4 merge at WC's H point when $N\to \infty$. We observe that the fold of limit cycles merge with other bifurcations in the limit of infinitely many neurons (see text). In the limit $N\to \infty$ (corresponding to $n=0$) we labelled the points H, LPa and LPb arising in WC system. Points with zero correlations are circled in red, all other points have non-zero correlations.}
 	\label{fig:Codim2BC2Pops}
 \end{figure}
Let us first follow the continuation the Hopf bifurcation labelled H1 in diagram of Fig.~\ref{fig:BressBifs02}(b) as the network size increases. We observe that this Hopf bifurcation collides, in the limit $N\to \infty$, with the continuation of the Hopf bifurcation H3 of the second branch of fixed point. Their collapse corresponds to the degenerate Hopf bifurcation point of the infinite model, denoted H in Fig.~\ref{fig:BressBifs02}(a). The cycle originating from H1 and the cycle originating from H3 will therefore emerge from this fixed same point, accounting for the result presented in section~\ref{section:Infinite} and in particular figure Fig.~\ref{fig:CyclesMeanField2pops} where we evidenced the presence of two different cycles emerging from the same bifurcation point H. Moreover, we have seen that the corresponding branch of limit cycles in the infinite-size system bifurcation diagram Fig.~\ref{fig:CyclesMeanField2pops} loses stability. This loss of stability corresponds to the continuation of the Neimark-Sacker bifurcation NS of figure~\ref{fig:BressBifs02} that can be continued in the mean-field limit. However, an important distinction between the finite and infinite case is that This cycle undergoes an additional period-doubling bifurcation when the size of the network increases, as observed in the bifurcation diagram of the infinite-size system, and as shown in figure \ref{fig:PoincareChaos}. In that figure, we continued these cycles for one fixed value of the parameter $I_1=-0.5$, as the number of neuron increases. The family of limit cycles continued is stable in a bounded interval of values for the parameter $N$. For a number of neurons close to $N=142$ (precisely $n_2=0.01418$), the branch of limit cycles undergoes a fold of limit cycle, and therefore the branch of limit cycles does not exists for smaller networks. For networks larger than $N=142$ neurons, the system presents a branch of stable and a branch of unstable limit cycles. Stable cycles persist up to networks as large as $N=1700$ neurons (precisely $n_2=0.001174$). At this size, where the branch stable limit cycles changes stability through a subcritical Neimark Saker (Torus) bifurcation with no strong resonance. When increasing further the size of the network, the system shows the presence of a chaotic attractor which clearly cannot exist in the two-dimensional WC system, yielding an important qualitative distinction between large scale networks and the relative mean-field limit (see \ref{fig:PoincareChaos}). This chaotic behavior exists until the infinite-size is reached.
 \begin{figure}[!h]
 	\begin{center}
 		\subfigure[Bifurcations of cycle]{\includegraphics[width=.45\textwidth]{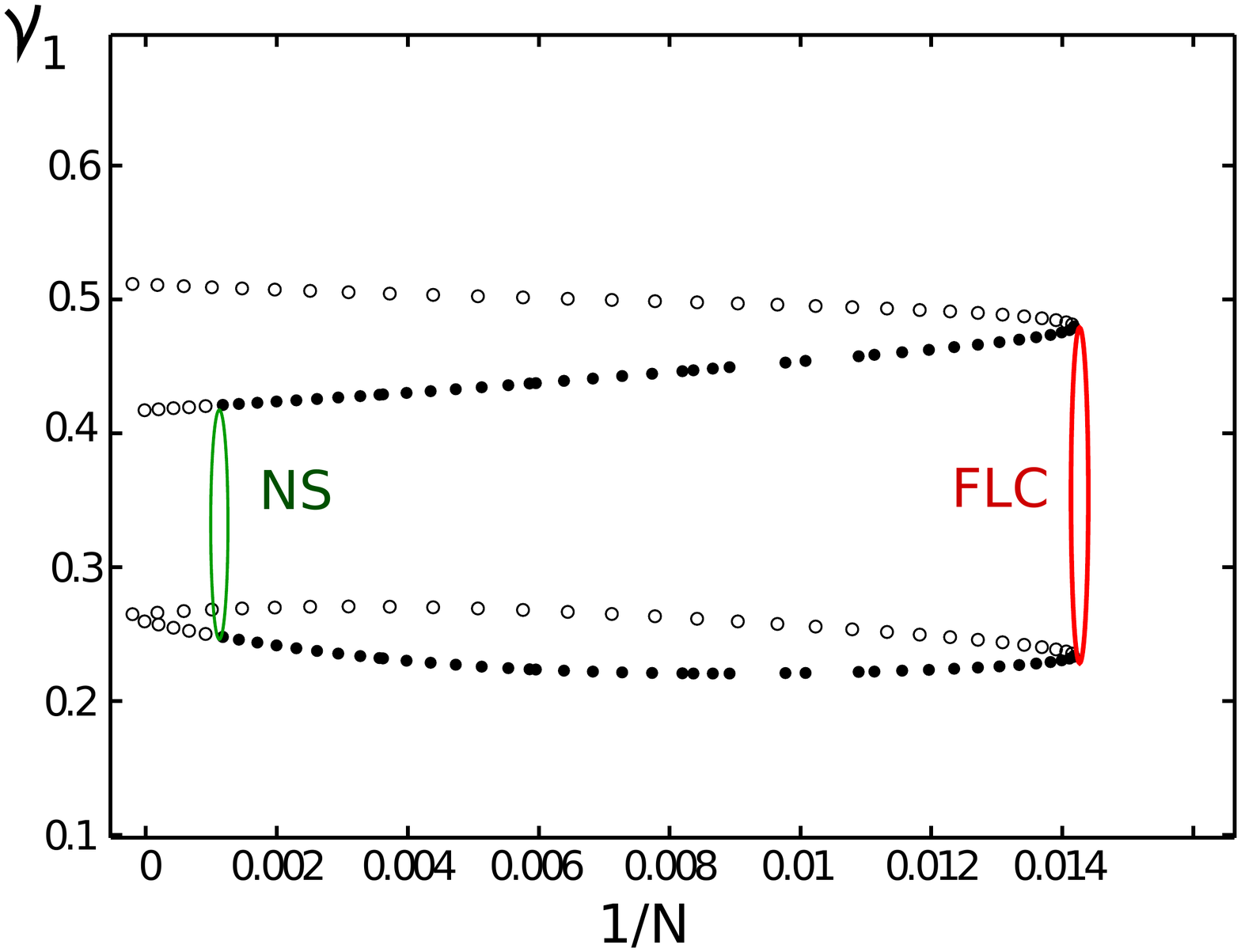}}
 		\subfigure[Chaos]{\includegraphics[width=.45\textwidth]{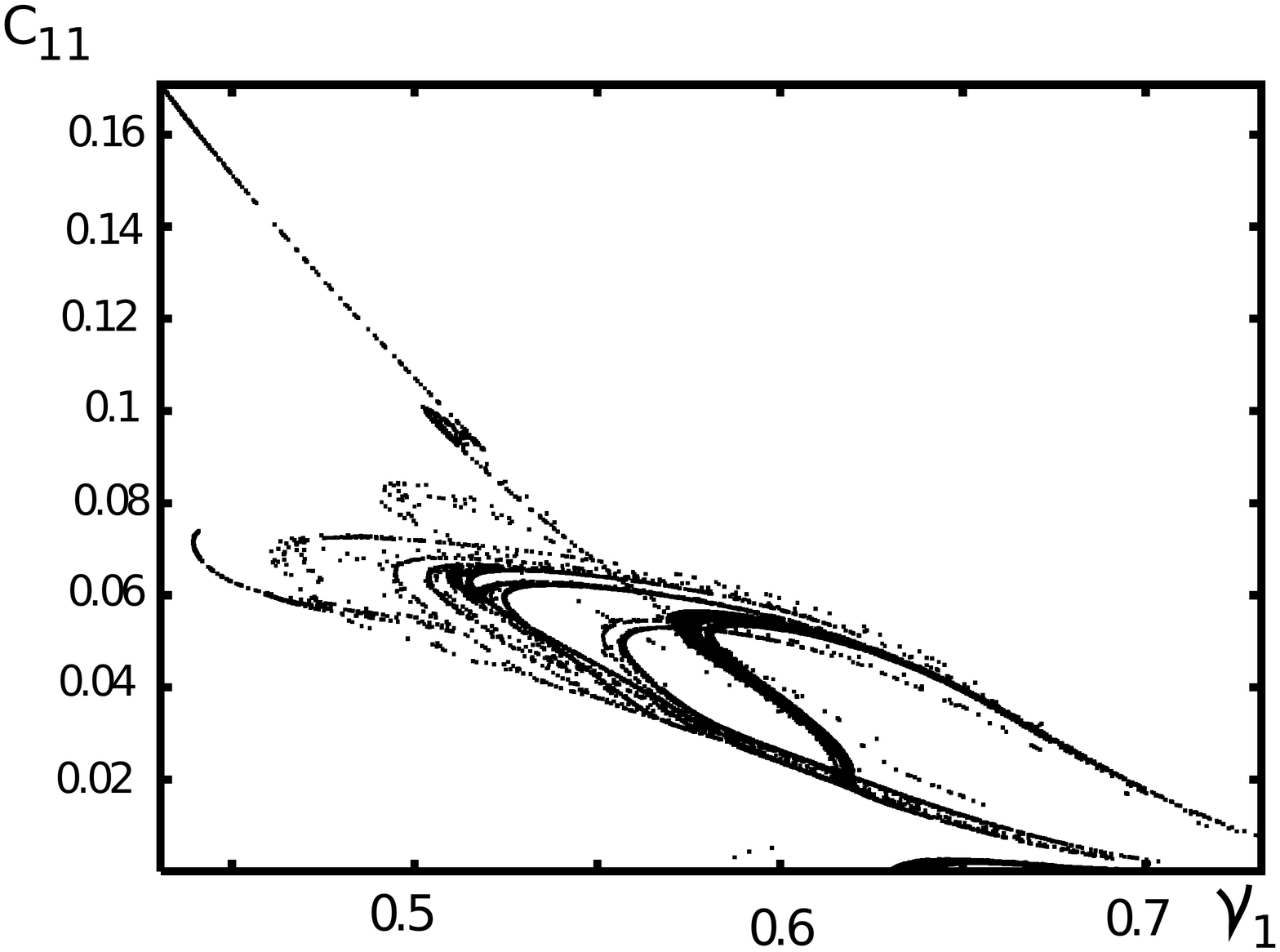}}
 	\end{center}
 	\caption{Chaotic behavior for large networks: (a) bifurcation diagram as a function of the network size. NS: Neimark-Sacker (Torus) bifurcation and FLC: fold of limit cycle. (b) Poincar\'e maps representing $a_1$ and $c_{11}$ with Poincar\'e section $a_2=0.7$. }
 	\label{fig:PoincareChaos}
 \end{figure}

Let us now follow the branch of limit cycles arising from H3 as the network size varies. We observe in our codimension three bifurcation diagram that the branch of Hopf bifurcations H3 exists until the infinite-size is reached. In that limit, we observe that this branch of Hopf bifurcation asymptotically meets the fold of limit cycles LPC2 and the curve of saddle-node LP7 and the saddle homoclinic curve arising from the Bogdanov-Takens bifurcations. This point corresponds therefore to a very degenerate behavior of the system, and corresponds, in WC bifurcation diagram, to the saddle-node bifurcation LPa that also corresponds to the homoclinic bifurcation point. Regarding the fourth Hopf bifurcation H4, we already observed that it was continued until it disappeared, in finite size, through a Bogdanov-Takens bifurcation.

Let us now consider the continuation of the saddle-node bifurcations. The bifurcation diagram of WC presents two saddle-node bifurcations LPa and LPb that correspond to the two saddle-node arising from the cusp bifurcation CP of Fig.\ref{fig:Codim2BC2Pops} that are labelled LP7 and LP8. In the mean-field limit, the saddle-node bifurcations corresponding to the continuation of LP1 and LP4 merge in a cusp bifurcation and disappear. At this same point, the fold of limit cycles LPC1 collapses and disappears in a homoclinic bifurcation. The saddle-node bifurcations LP2 and LP6 disappear in the mean-field limit by being continued tangentially to the line $n=0$. The continuations of the saddle-node bifurcations LP3 and LP6 crosses the line $n=0$ with no singularity and therefore exists in the infinite size system, corresponding to a non-zero correlation saddle-node bifurcation and that do not correspond to any plausible behavior (all plausible behaviors correspond to zero correlation points). Eventually, the saddle-node bifurcation LP7 connects with the fold of limit cycle, generating the homoclinic bifurcation of the infinite-size system.

 We therefore observe that the bifurcation diagram of the finite-size Bressloff model non-trivially unfolds  the infinite-size system, and that at any finite scale the bifurcation diagram and the solutions it present differ significantly from the infinite-size system and from WC system. In particular, it shows additional cycles, as observed in the infinite-size system previously, and a chaotic effect that is not present in the two-dimensional WC system. This analysis also illustrates a mesoscopic-scale phenomenon in the presence of a periodic orbit that is stable only when the size of the network is comprised roughly between one hunded and one thousand neurons, a typical size of a cortical columns often modelled by Wilson and Cowan system. These results remain valid for BCC model, as shown in appendix \ref{appen:2PopsBCFinite}. The question that now arises is whether  these behaviors indeed reflect behaviors of the Markovian model.

 \subsection{Finite-size Markovian model}
 In the previous section, we investigated the distinctions between the solutions of the finite-size BCC model and compared these with those of the infinite-size and of WC systems. We now return back to the original Markovian system that allowed derivation of Bressloff and BCC equations, and compare simulations of this network with the finite-size Bressloff model, composed of two differential equations arising from the moment truncation of this random variable. We will perform this analysis in the two models with two populations investigated in the paper, Model I and Model II. \review{This analysis is particularly interesting in regions of the parameter space where the system presents a fixed-point behavior, and where the moment expansion is a relevant representation of the system. In the periodic orbit regimes, we discuss a suitable methodology, beyond the scope of the present paper, that is suitable to analyze the system.}
 
In order to simulate the Markov chain, we make use of Doob-Gillespie's algorithm \cite{doob:45,gillespie:76,gillespie:77}. From the master equation governing, e.g. BCC model:
\begin{multline*}
	\der{P(n,t)}{t}=\sum_{i=1}^M \Big[\alpha_i\,(n_i+1) P(n_{i+},t) - \alpha_i n_i P(n,t) \\ 
	+ F_i(n_{i-})P(n_{i-},t) - F_i(n)\,P(n,t) \Big]
\end{multline*}
we derive the transition rates of the state variable $n$. We use these rates to simulate  sample paths of the state variable $n$ and rescale these to compare to the solutions obtained for BCC or Bressloff models and the infinite-size models. Given the configuration $n(t)$ of the network at time $t$, we have {in the BCC case (Bressloff case is straightforwardly treated using the same method)}:
 \begin{itemize}
 	\item The probabilistic intensity for the transition $n_i(t+1)=n_i(t)-1$ is equal to  $q_i=\alpha_i \, n_i$,
 	\item The probabilistic intensity for the transition $n_i(t+1)=n_i(t)+1$ is equal\footnote{{Note that the activation functions are different in BCC and Bressloff models.}} to $f_i=f(\sum_{j=1}^M w_{ij}\,n_j+I_i)$.
 \end{itemize}  
 Given an initial condition $n(0)$, a simulation consists of computing all possible transitions probabilities $q_i$ and $f_i$. We then draw the time for the next transition as an exponential random variable with intensity $Q=\sum_i (q_i+f_i)$. This time provides the next event occurring in the chain. By the properties of exponential laws, the transition is $n_i(t+1)=n_i(t)-1$ (resp.  $n_i(t+1)=n_i(t)+1$) for some $i\in\{1,\ldots,M\}$ with probability $q_i/Q$ (resp. $f_i/Q$). This simulation algorithm for the Markov process $n(t)$ is exact in law, and does not involve any approximation such as time-discretization, and therefore allows efficient simulation of the sample path and of its statistics. These statistics are computed using a Monte-Carlo algorithm consisting of simulating many independent sample paths of the process and deriving from these the trajectory of the mean firing rate $\nu_i(t)$ and the correlation in the firing activity $C_{ij}(t)$ through the empirical mean and correlations obtained. {In the simulations we focus on BCC model. The same numerical experiments can be performed in the Bressloff case. }

\subsubsection{One-Population Model}
We start by considering the one-population model studied in section \ref{sec:OnePop}. In that case we have shown that both Wilson and Cowan and the infinite-size systems had the same behavior, and that this behavior was regularly unfolded in the finite-size systems. In these cases, two saddle-node bifurcations defined a zone of bistability that depended on the size of the network as shown in Figure~\ref{fig:1popcodim2BC}(d). We simulate the network using Doob-Gillespie algorithm and identify the boundaries of the zones where bistability occurs by continuation of the fixed points obtained as the parameters increase or decrease. The continuation is initialized by running the algorithm with parameters associated to a single attractive fixed point of WC system. Then the parameter $I$ is slowly varied and simulations are run with initial condition being the last values of the network state of the previous simulation. We observe the hysteresis phenomenon and bistability, and the bistability zones closely matches WC system, as shown in Figure~\ref{fig:MarkovOnePop}.
\begin{figure}[!h]
	\centering
	\includegraphics[width=.4\textwidth]{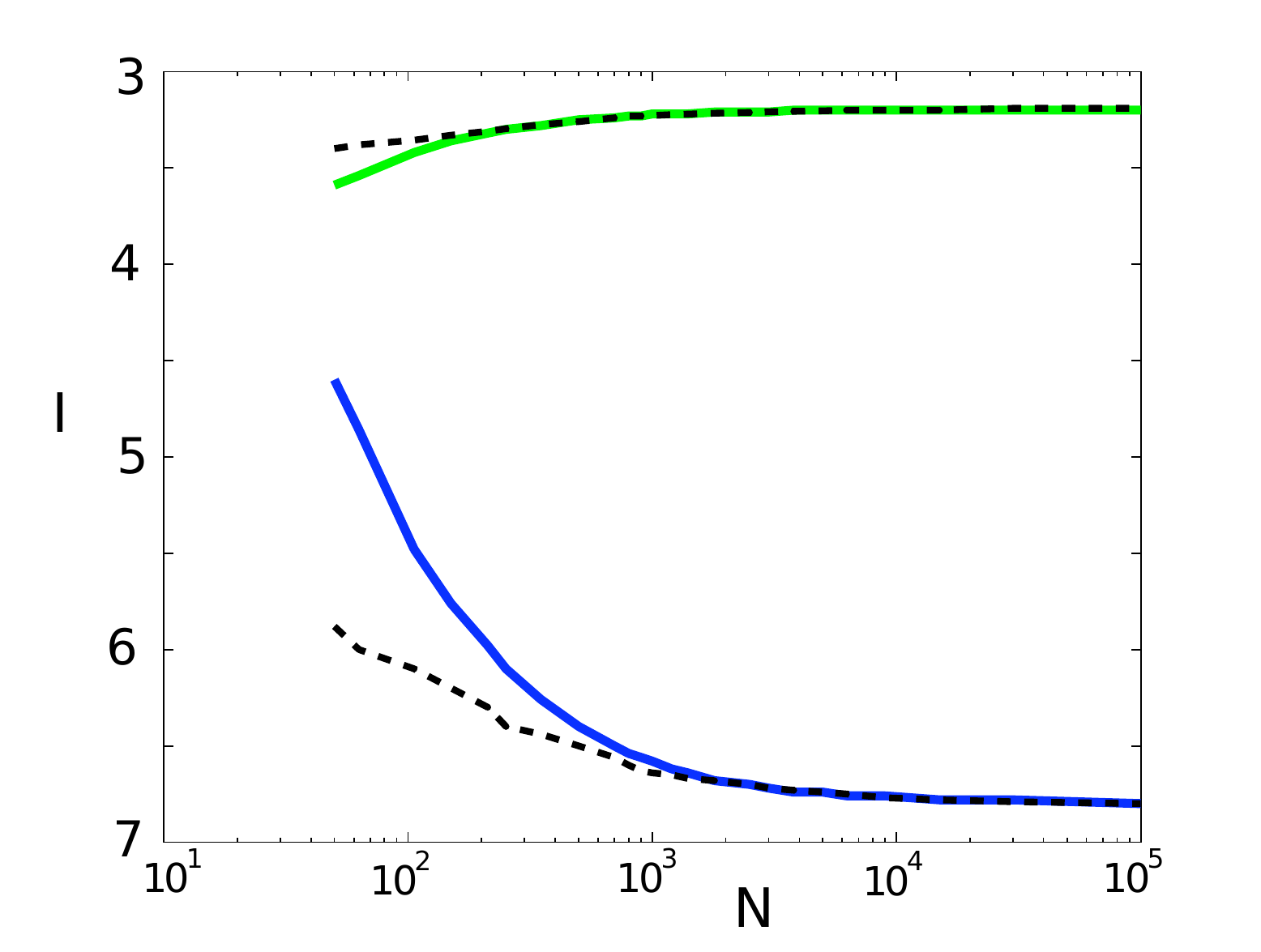}
	\caption{Bistability zones in BCC finite-size model (black dashed curve) and computed from the Markovian {BCC} model (colored plain curve).}
	\label{fig:MarkovOnePop}
\end{figure}
We simulate the network for increasingly large networks ranging from $N=50$ neurons to $N=100\,000$ neurons with $10\,000$ simulated sample path from which we extract the empirical mean and standard deviations. From the curves obtained in that fashion, we automatically identify zones where there is bistability. These boundaries are plotted together with the values of the saddle-node bifurcation at these points arising from the codimension 2 bifurcation diagram of {BCC} finite-size model, and we observe that the curves closely match as soon as the network is larger than $1\,000$ neurons. Therefore in this model, BCC system closely reproduces the behavior of the finite-size Markov model for sufficiently large networks. A similar result is obtained in BCC case (not shown).
 
 \subsubsection{Markovian Model I \review{: Fixed points regimes}}
We now turn to study the two-dimensional network Model I. In that case we showed that depending on the input parameter $I_1$, Wilson and Cowan system can either present a stable cycle or a stable fixed point. We first chose a case where WC system converges towards a fixed point, and to this end will study the system for $I_1=-5$. In that case, both WC system, the infinite-size system, BCC and Bressloff systems converge towards a fixed point. \review{The periodic regimes are addressed in section \ref{sec:CyclesMarkov}.} 
 
\medskip 
In the case where WC system presents a unique hyperbolic attractive fixed point, we showed that the infinite-size presents an exponentially stable fixed point with the same values of mean firing rate and null correlations, which has a counterpart in the unfolding of BCC and Bressloff model. Simulations of the Markovian model illustrate the fact that all the sample paths of the process converge quite fast towards this fixed point with null correlations displaying stochastic variations, and each sample path presents a precise match with the solution of WC solution at this point. The stochastic variations around Wilson and Cowan trajectory arising from the randomness of the Markov chain are averaged, and produce a fixed mean firing-rate corresponding to the fixed point of the infinite-size system, i.e. having null correlations and the same value of mean firing rates. Moreover, we observe that the correlations vanish, which implies that the state obtained is an asynchronous state\review{, as expected from Bressloff's expansion}. 

Figure~\ref{fig:MarkovFP} represents simulations of Model I with $I_1=-5$. 
 \begin{figure}[!h]
 	\centering
 		\subfigure[Different Sample Path]{\includegraphics[width=.45\textwidth]{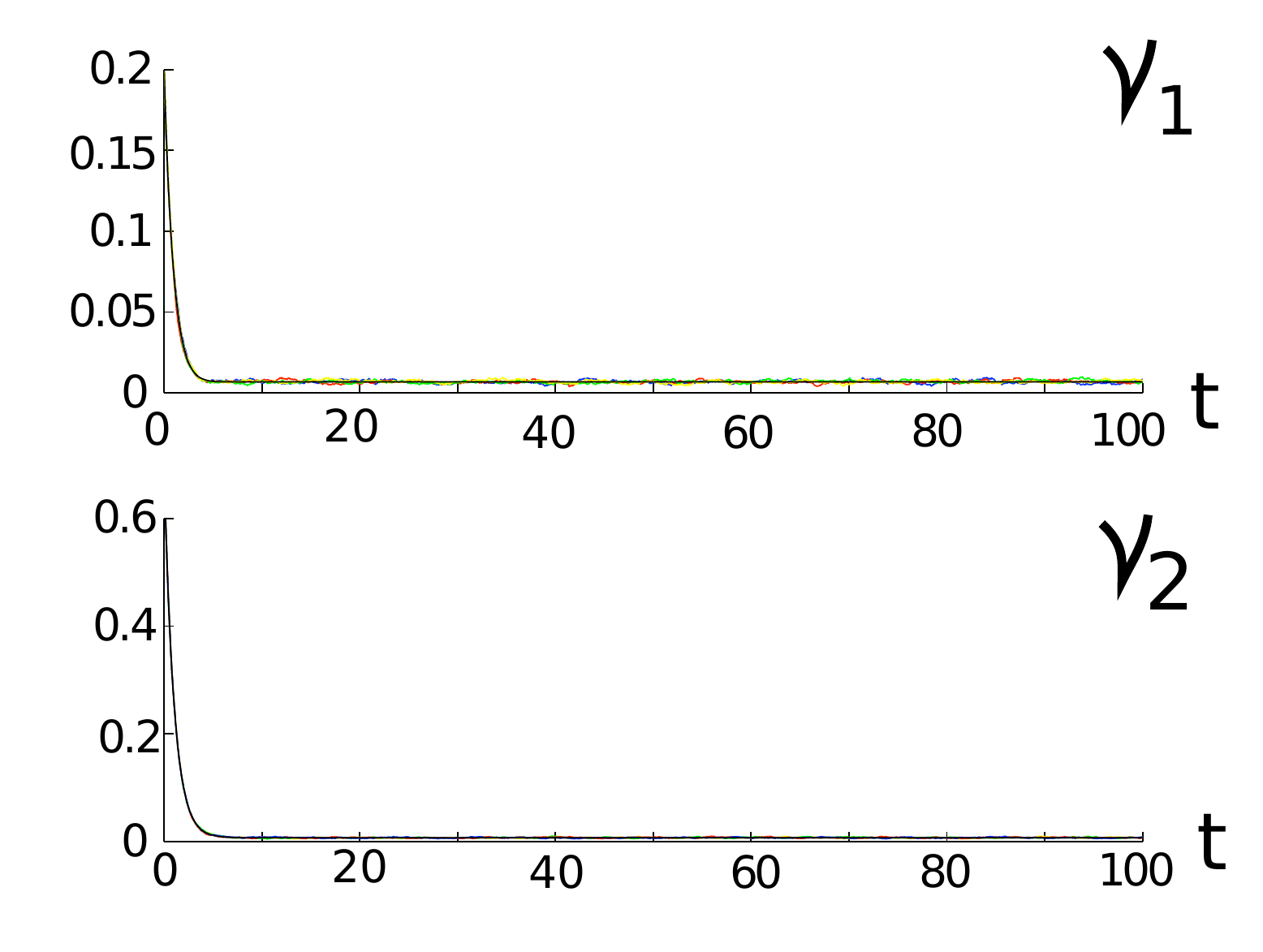}}\qquad
 		\subfigure[Stationary regime (zoomed)]{\includegraphics[width=.45\textwidth]{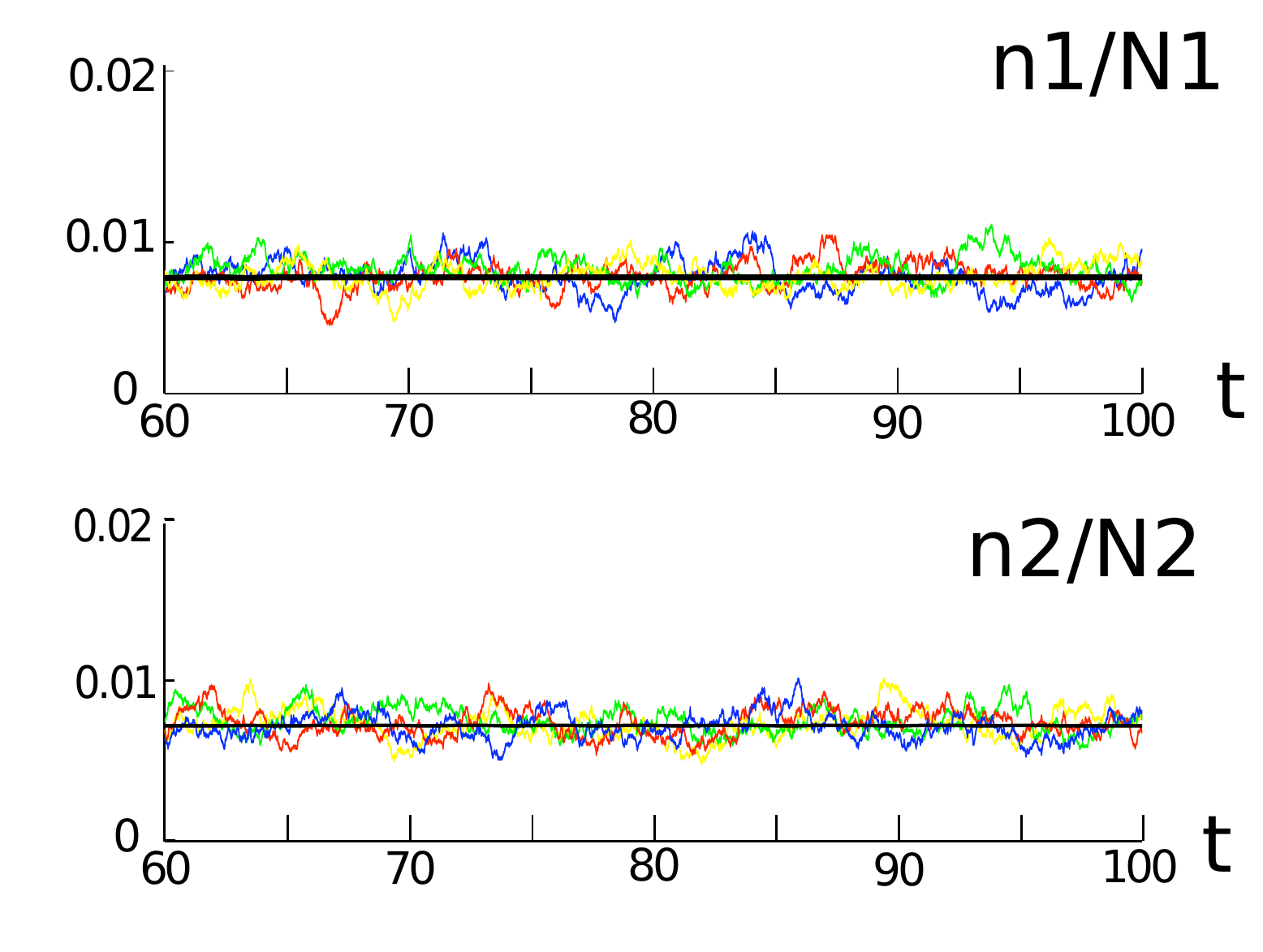}}\\
 		\subfigure[Statistics]{\includegraphics[width=.45\textwidth]{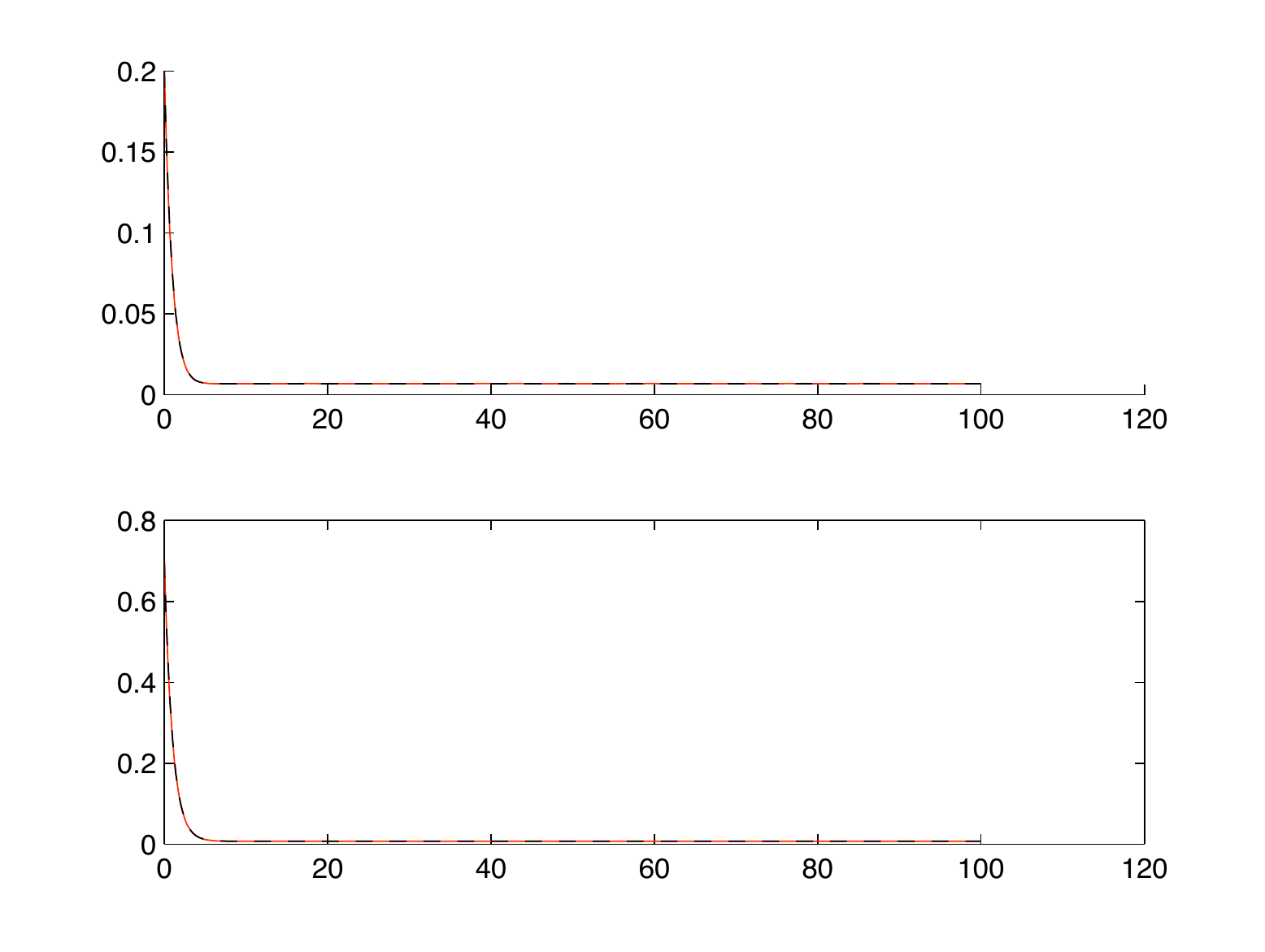}}
 		\subfigure[Statistics]{\includegraphics[width=.45\textwidth]{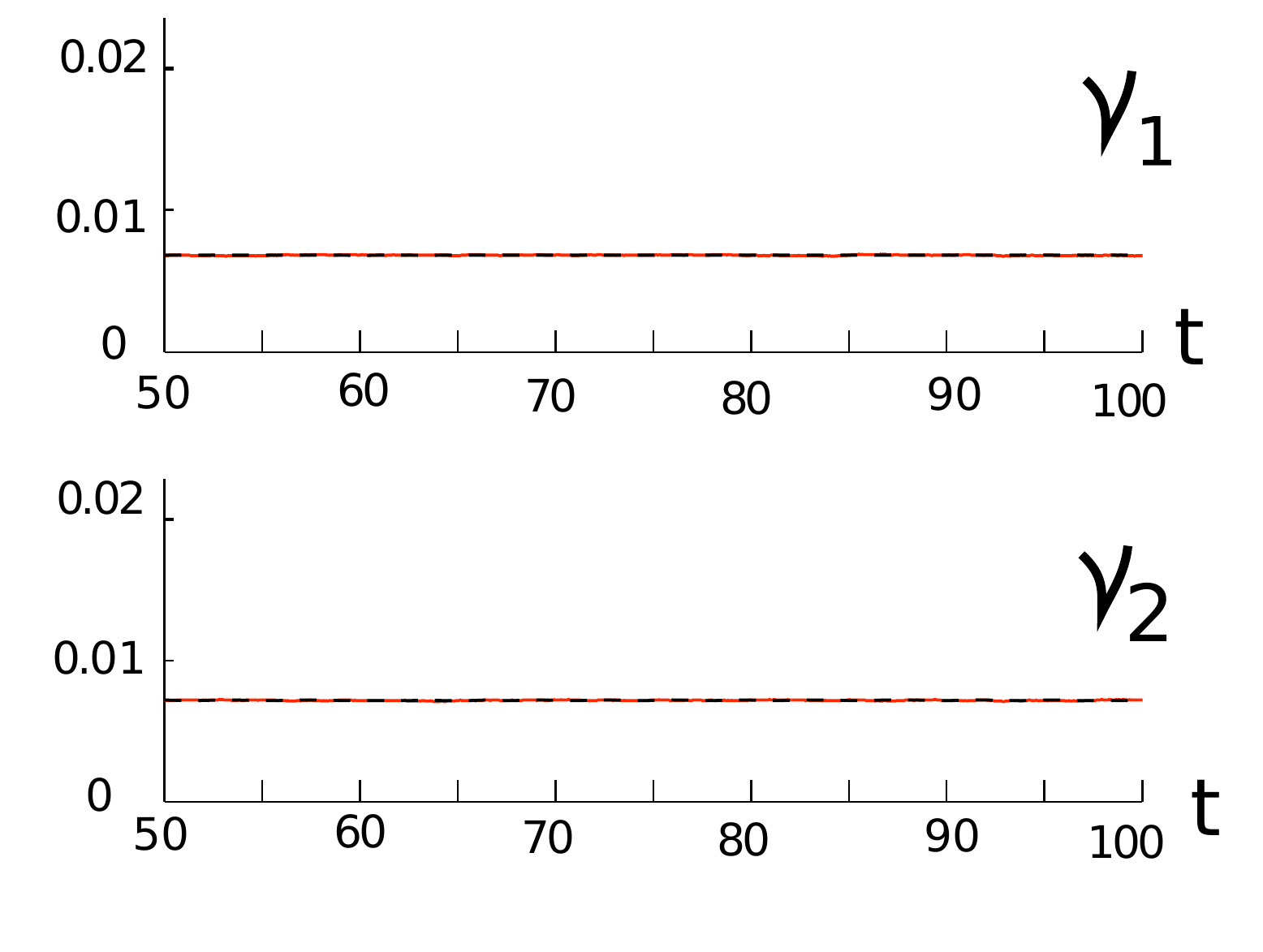}}
 	\caption{Network Model I, for parameters associated with a stable fixed point ($I_1=-5$) closely matches WC trajectory with random variations around the fixed point that average at the value of the fixed point, and correlations vanish (asynchronous state). (a),(b): The different colors correspond to different sample paths. {The simulations were performed in the BCC case}. }
 	\label{fig:MarkovFP}
 \end{figure}
 In panels (a) and (b) are plotted four sample paths of the process for $10\,000$ neurons simulated over on a time interval equal to $100$, and panels (c) and (d) the mean-firing rates averaged over $5\,000$ realization of the process, and thats shows a very close agreement with the solution of WC system. The maximal absolute value of the correlations (not plotted) are of the order of $10^{-5}$ and are neglected as an effect of the finiteness of the sample path simulated for computing the mean and covariance. The same observations hold for all the simulations performed for parameters corresponding to an hyperbolic fixed point of the infinite system. Note that the difference between mean-firing rates of the solution of WC system on one hand and of BCC system on the other hand is of the order of $10^{-5}$ at this network scale, and therefore both models fairly approximate the stochastic Markovian model.

\subsubsection{Markovian Model II}
We perform the same analysis in the case of the network Model II studied in section \ref{sec:second2Pop} with symmetric connectivity with null diagonal and inhibitory interaction in BCC case. In that case, we showed that WC system presented bistability, and the infinite-size system presented a more complex behavior with multistability and different limit cycles. In the simulations of the Markovian system, similarly to the case of Model I, we observe that when WC system presents a single hyperbolic fixed point, the Markovian system closely matches with the solution of WC system during the transient phase, and the stationary state randomly varies around the fixed point of WC system. The mean of the firing rates over many realization of the process converges exactly towards the solution of WC system, and the correlations tend to zero, implying the system is in an asynchronous state and that its behavior is efficiently summarized by Wilson and Cowan system. We note that in that case, the infinite-size system and the finite size systems both presented additional fixed points with non-zero correlations, that do not actually exist as a regime of the Markov chain.
 
 For parameters associated with bistability in WC, the situation is similar, and we observe in our simulations that the WC system captures precisely the behavior of the Markovian system and all additional solutions derived from the moment hierarchy truncation do not appear. Simulations of the Markovian system show that depending on the initial condition, the system will converge towards one or the other fixed point. Albeit the stochasticity of the Markov chain yielding the possibility to switch between the two fixed points, this phenomenon does not occurs often, and except for input parameters very close to the saddle-node bifurcations, no switching between fixed points is observed. In that case, we note that the additional fixed point with non-zero correlations and the additional cycles evidenced in the infinite-size system and that are present in the finite-size moment equations do not appear in the simulations of the underlying Markovian model. We present in Figure~ \ref{fig:network2} the values of taken by the mean firing rate and correlations for $1000$ realizations of the Markov chain over the $20$ last units of time for a simulation of $20\,000$ neurons as a function of the input paramter $I_1$. We treat the cases $I_2=0$ and $I_2=5$, the blue crosses are the continuation of the fixed points with $I_1$ increasing (from $-10$ to $10$ for $I_2=0$ and from $-5$ to $15$ for $I_2=5$) and red circles for $I_1$ decreasing. The continuation is initialized by running the algorithm with parameters associated to a single attractive fixed point of WC system. Then the parameter $I_1$ is slowly varied and simulations are run with initial condition being the last values of the network state of the previous simulation. We observe the hysteresis phenomenon and bistability, and the bifurcation diagram obtained perfectly matches WC's, and no counterpart of any of the identified correlations effects are observed. 
 
 \begin{figure}[!h]
 	\centering
 		\includegraphics[width=.45\textwidth]{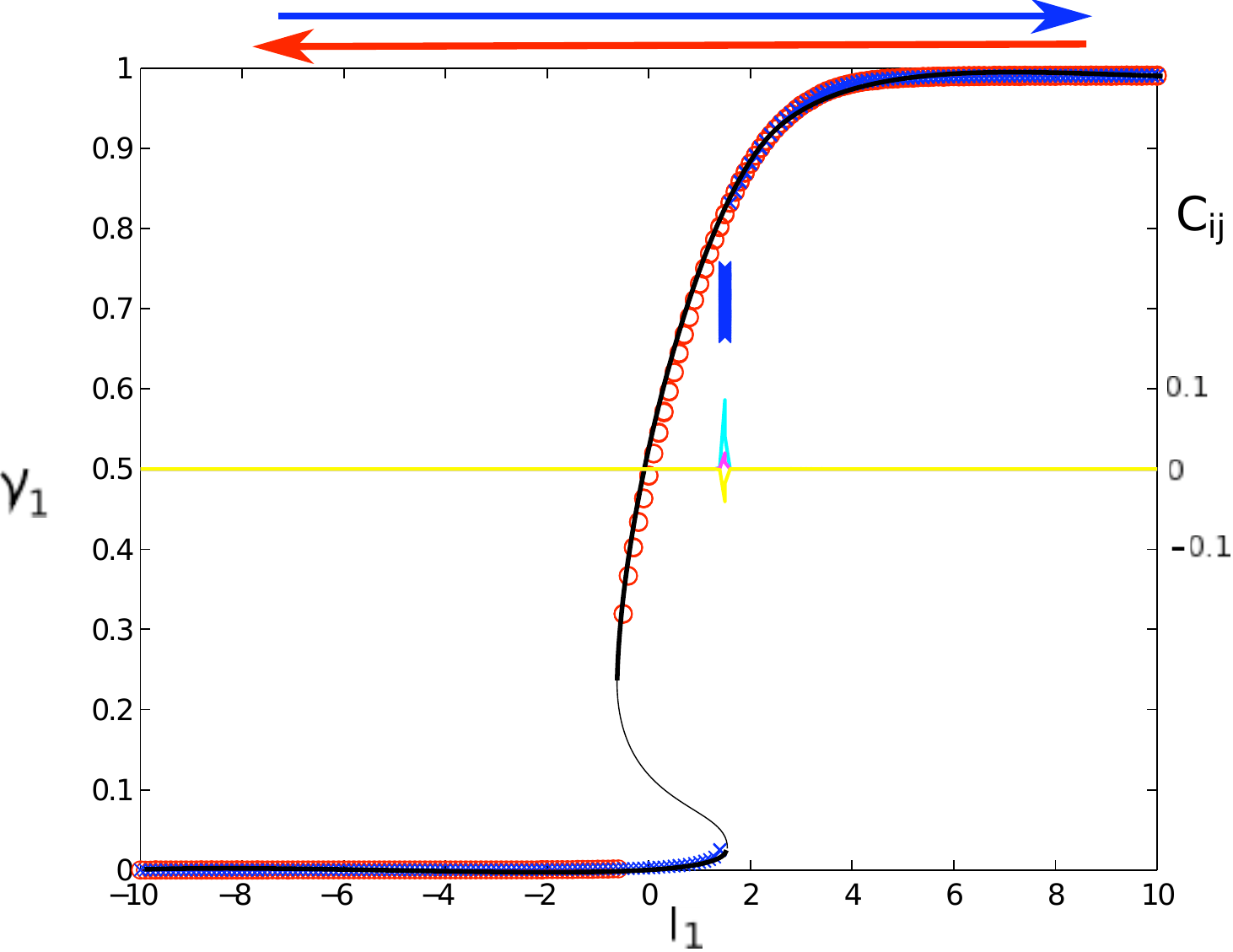}\qquad
 		\includegraphics[width=.45\textwidth]{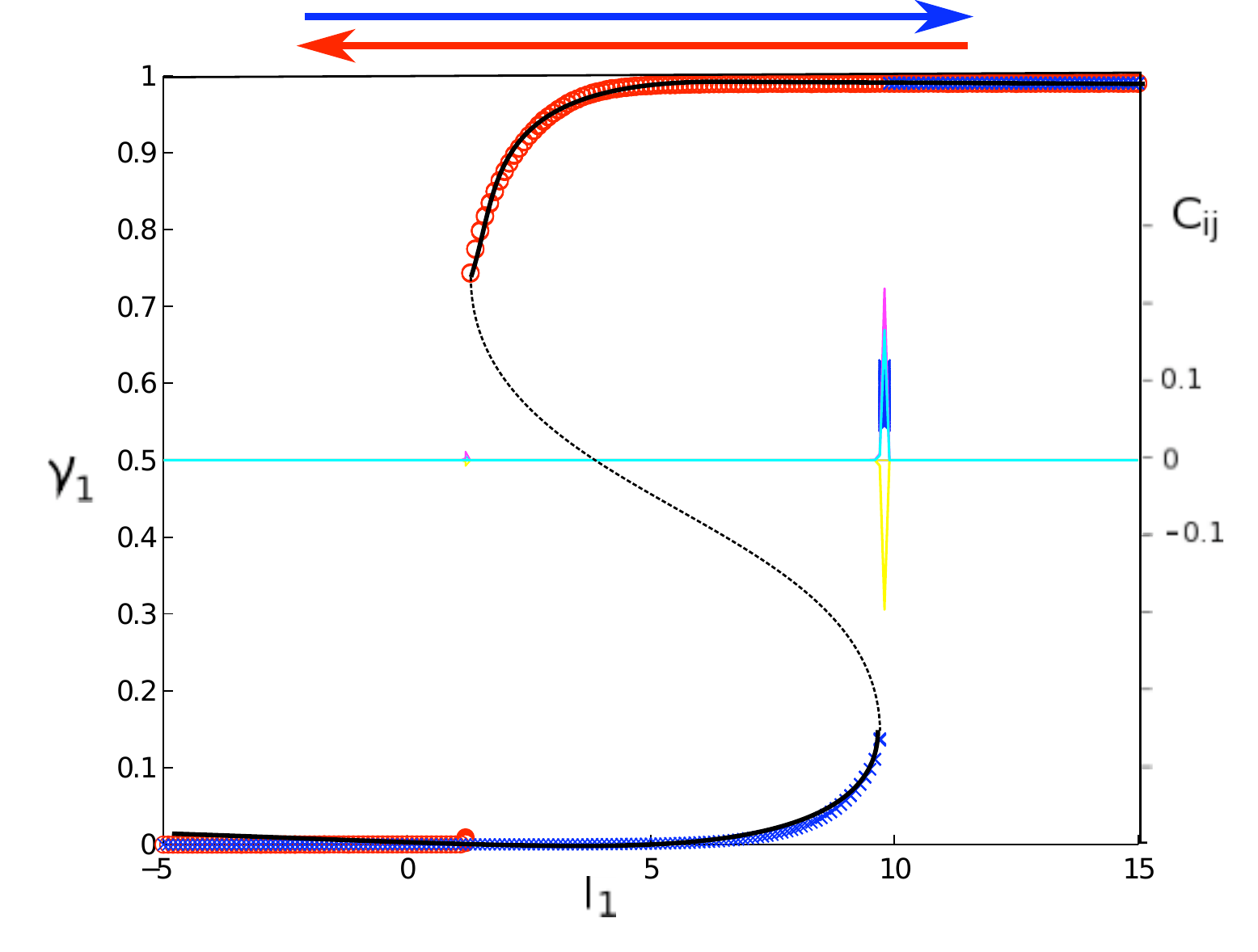}
 	\caption{Stationary values of the Markov chain in Model II with $I_2=0$ as a function of $I_1$ (see text). Continuation of equilibria by increasing $I_1$ (blue crosses) and by decreasing $I_1$ (red circles), superimposed with WC bifurcation diagram (black curve). Correlations are plotted with a different axis origin in cyan ($C_{11}$), magenta ($C_{22}$) and yellow ($C_{12}$), keep very close to zero except at the precise value of the bifurcations, where the jump between the two fixed points produces a spurious standard deviation. BCC case.}
 	\label{fig:network2}
 \end{figure}
 
\review{It appears from this analysis that the finite-size moment equations present additional solutions, associated with large correlations, that do not exist in the initial Markovian model. We emphasize the fact that though the infinite-size model is not equivalent to the moment equations derived by both authors, the solutions evidenced have a counterpart in the finite-size moment equations corresponding to both systems. These equations present solutions with large correlations (that actually diverge as the number of neurons tend to infinity), which are not solutions of the initial Markov chain, which naturally behaves as WC system for large population sizes, as expected from the moment equations by both authors. } 

\review{A complementary approach to the study of bistability in the Markov framework allowing to quantify the probabilistic intensity of noise-induced transitions can be driven through the use of Wentzel–Kramers–Brillouin (WKB) approximation (see e.g. \cite{dykman-mori-etal:94,bressloff:10} and references therein) performed directly on the Markovian system. This method was used by Bressloff in~\cite{bressloff:10} in the same context for a one-population model, and more generally allows studying noise-induced transitions between multiple metastable states where the mean-field theory based on the system-size or loop expansions break down. It allows reducing the problem to a Hamilton-Jacobi equation from which one can derive the escape rate, namely the probability intensity that the solution escapes from the attraction bassin of a given equilibrium of the deterministic equation (the WC system) and enters the attraction bassin of the other equilibrium. }

\subsubsection{\review{Periodic regimes}}\label{sec:CyclesMarkov}

\review{The models proposed in~\cite{buice-cowan:07,buice-cowan:10,bressloff:09} emphasize the role of the mean-firing rate and of the zero timelag correlation functions. As illustrated in the previous section, these variables are suitable to model the system at fixed points. However, it remains unclear how these variables characterize the behavior of the Markovian system in periodic or chaotic regimes. }

\review{The analytical study of such regimes can be carried out in a satisfactory fashion either in the theory of master equations and in the theory of stochastic dynamical systems, using the Langevin equation. Such studies have been addressed by various authors within the context of chemical master equations~\cite{mckane-nagy-etal:07,bolland-galla-etal:08,gaspard:02} and recently Bressloff applied the same tools to neural models, using Langevin equations and Floquet theory~\cite{bressloff:10}. Indeed, it appears that in regimes where the Wilson and Cowan system present a cycle, the Markovian model is characterized by random fluctuations in the neighborhood of this attractor, and the system loses phase information over time. These dynamics can be suitably studied in terms of a diffusive phase variable. In details, under the Langevin approximation of the dynamics, a possible way to study the behavior of the solutions of the Markovian system would be to use the stochastic phase reduction (see \cite{teramae-nakao-etal:09}).}

\review{In this section we do not make use of these more customary methods but rather investigate the relationship between behaviors of the finite-size moment equations and behaviors of the Markovian system. For our choice of parameters in Model I, let us consider the parameter region $I_1 \in [I_{1,H},I_{1,LPa}]$ where $I_{1,H}$ (resp. $I_{1,LPa}$) correspond to the value of $I_1$ at the point H (resp. LPa). In this region, the solutions of WC system are periodic orbits. The question we address is whether the behavior of the finite-size (moment equations and Markovian) systems is comparable to the behavior of WC solutions or to the behavior of the finite-size moment equations, which differs at any finite scale from WC's behavior as shown in section~\ref{sec:TwoPopsFinite}. Indeed, we observed in figure Fig.~\ref{fig:Codim2BC2Pops} that the Hopf bifurcation of the deterministic WC system, which is a degenerate point of the moment equations, expands in the codimension two diagram into two branches of generic Hopf bifurcations, that imply the presence of cycles in the finite-size system. These Hopf bifurcations either unfold a singularity corresponding to the actual eigenvalues $\pm \lambda$ of WC Hopf bifurcation or to twice this value by application of the lemma \ref{thm:WilsonBuiceCowan}, and can be either stable or unstable. Therefore the boundaries of the region where the system is in a periodic regime are modified as soon as the system size is finite. The same phenomenon appears close to LPa, in which case the diagram unfolds into a Hopf bifurcation, a saddle node bifurcation, a fold of limit cycles and a saddle homoclinic bifurcation. We will observe that in both cases, the behavior of the Markovian system actually reflects the properties of the finite-size moment equations, and that quasicycles will appear in the neighborhood of WC's Hopf bifurcation.}

\review{Let us first have a closer look to the behavior of the moment equations in the neighborhood of $I_1=I_{1,H}$ in the finite-size system. WC's Hopf bifurcation gives rise to two Hopf bifurcations whose codimension two curves are tangent to the axis $n=0$. The curve H1 is related to a singular eigenvalue converging as $N$ goes to infinity to the value of WC Hopf bifurcation, and gives rise to stable limit cycles with period close to the expected cycle in WC system. The other branch, H3, corresponds to faster cycles with period around half WC cycles' period. Let us now fix a certain size of the network $N$. In the moment equations, and let $I_1$ be a free bifurcation parameter. As $I_1$ increases, cycles will appear corresponding to the Hopf bifurcation H1 at a value $I_{H1}(N)$. We observe that this value is increasing as the network size increases. Therefore, in regions where WC system has a stable focus, the finite-size moment equations show oscillations of frequency close to WC cycles' frequency. For instance for $n=0.02$ corresponding to the bifurcation diagram of figure Fig.~\ref{fig:BressBifs02}, we have $I_{H1}(50)=-3.37$, defining a region of quasicycles $I_1\in[-3.37,-3.245]$. This size of the zone where quasicycle appear decreases very fast (as we can see from the shape of the graph of H1 in the codimension two plot Fig.~\ref{fig:Codim2BC2Pops} and for $N=1000$ we have $I_{H1}(50)=-3.248$. It is important to note that at a size $N=50$, oscillations are very difficult to see because of the small system size. Moreover, in the particular case treated, the graph of the H1 bifurcation is almost a vertical line: there is almost no difference between the value of H1 across different system sizes. The same phenomenon is observed in the simulations: quasicycles close to the Hopf bifurcation point barely exist, and the value of $I_1$ needs to be very close from $I_{1,H}$ to start observing these quasicycles.}

\review{WC's family of limit cycle lose stability through a saddle-node homoclinic bifurcation. We observe that unfolding this bifurcation, we find out that the branch of limit cycles in the finite-size system undergo a fold of limit cycles corresponding to a decreasing value of $I_1$, having the effect of reducing the amplitude of the interval of $I_1$ values related to a periodic orbit. The curve of fold of limit cycles is no longer vertical in that case and a wide non-trivial region distinguishes the finite-size and WC's periodic orbit parameter regions. We observe again in that case that the LPC2 line actually corresponds to a good estimate of the disappearance of cycles. Indeed, the codimension 2 bifurcation diagram shows that for $I_1=0$ oscillations should disappear at a system size close to 550 neurons. In figure Fig.~\ref{fig:spectrum} we observe that indeed, for $N=500$ neurons and $I_1=0$, the system does not show any oscillations, and for $N=700$ neurons the system starts showing an oscillatory behavior. Moreover, in that case we observe that the amplitude of the cycles is not small, consistent with the cycles we observe in the finite-size system. }

\review{The presence of a chaotic orbit as seen in Fig.~\ref{fig:PoincareChaos} and the fact that indeed in such regions the behavior of the system is aperiodic suggests that in these regions close to a limit cycle of Wilson and Cowan system, the underlying Markovian system can present a low-dimensional chaotic macroscopic behavior arising through the interaction of a large but finite number of neurons, a phenomenon that could be related to the work of El Boustani and Destexhe in~\cite{elboustani-destexhe:10b}, where intrinsically stochastic neurons are shown to yield low-dimensional chaotic macroscopic behaviors.} \review{Results are presented in figure Fig.~\ref{fig:spectrum}} 

\begin{figure}[htbp]
	\centering
		\subfigure[$i_1=-3.25$, $N=1000$: Power spectrum]{\includegraphics[width=.45\textwidth]{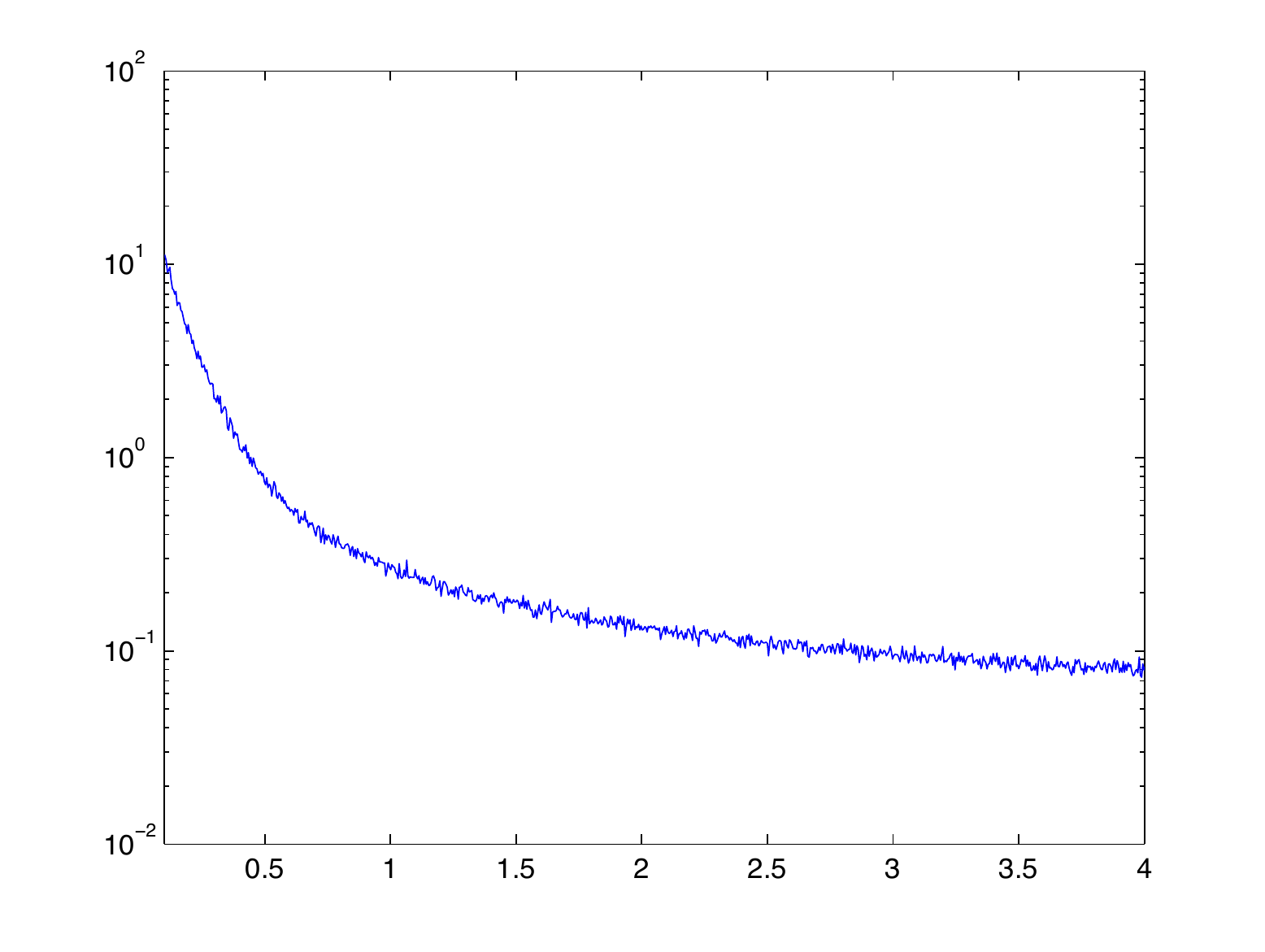}}\qquad
		\subfigure[$i_1=-3.246$, $N=1000$: Power spectrum]{\includegraphics[width=.45\textwidth]{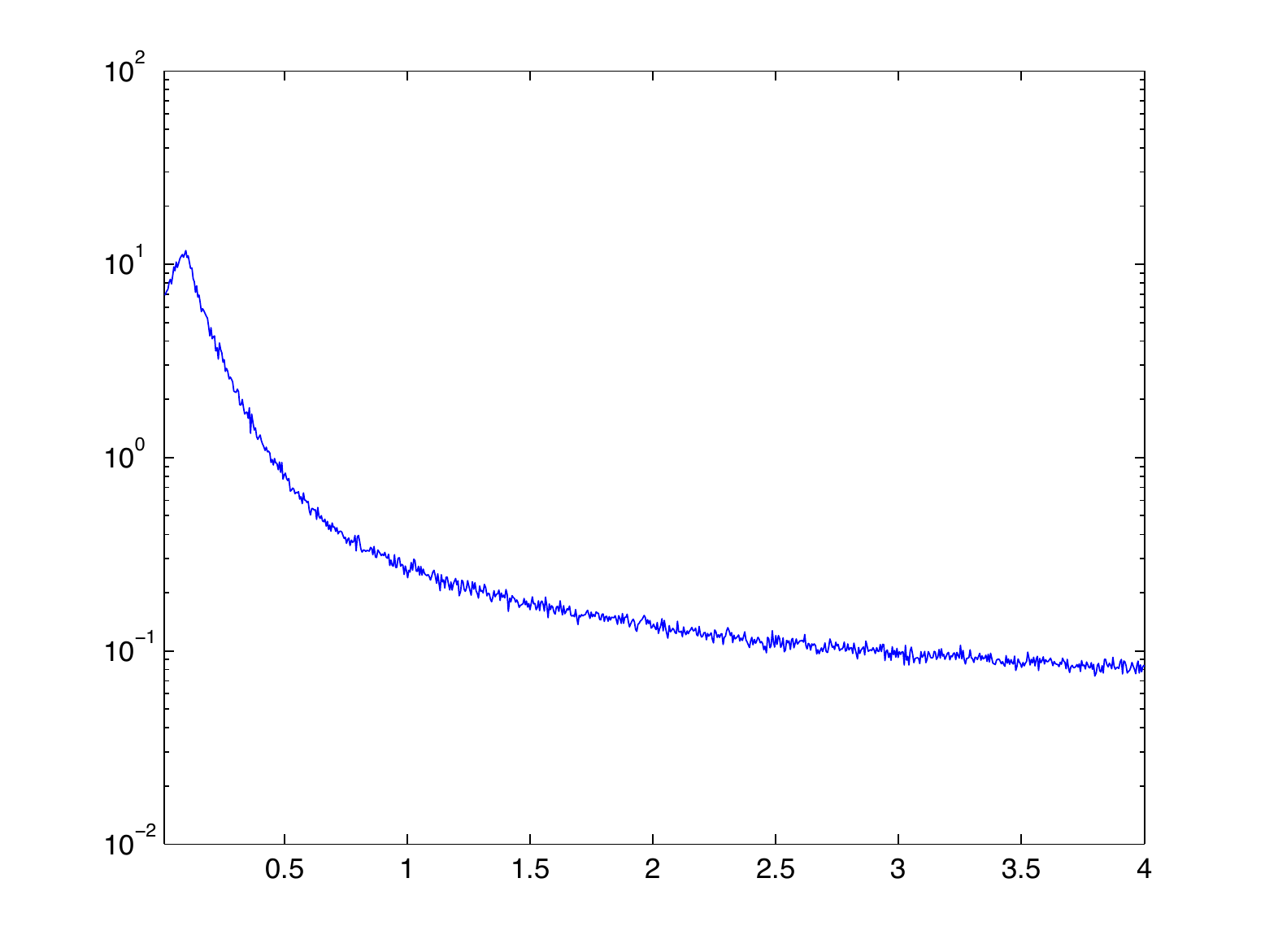}}\\
		\subfigure[$i_1=0$, $N=500$: Power spectrum]{\includegraphics[width=.45\textwidth]{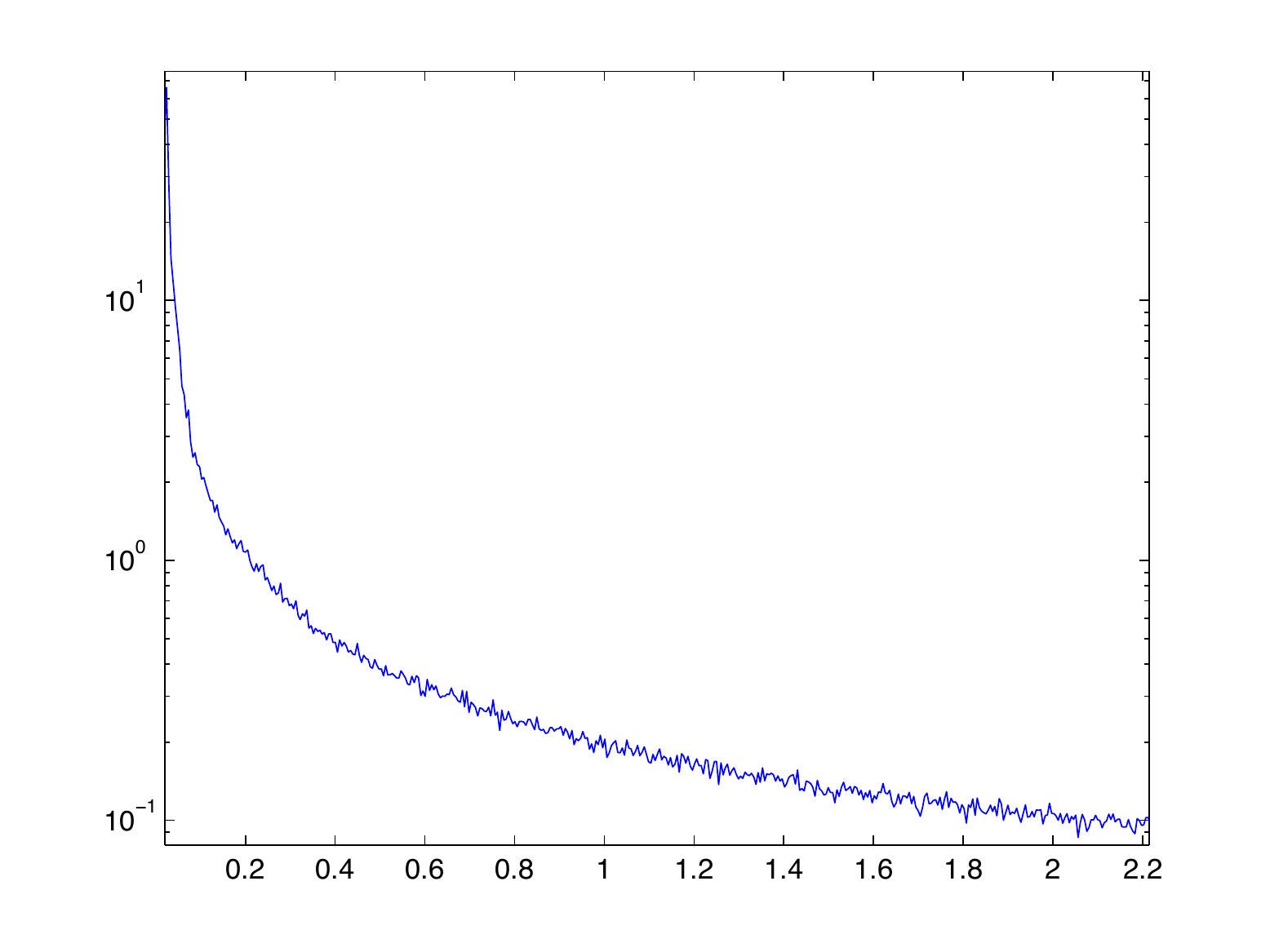}}\qquad
		\subfigure[$i_1=0$, $N=1000$: Power spectrum]{\includegraphics[width=.45\textwidth]{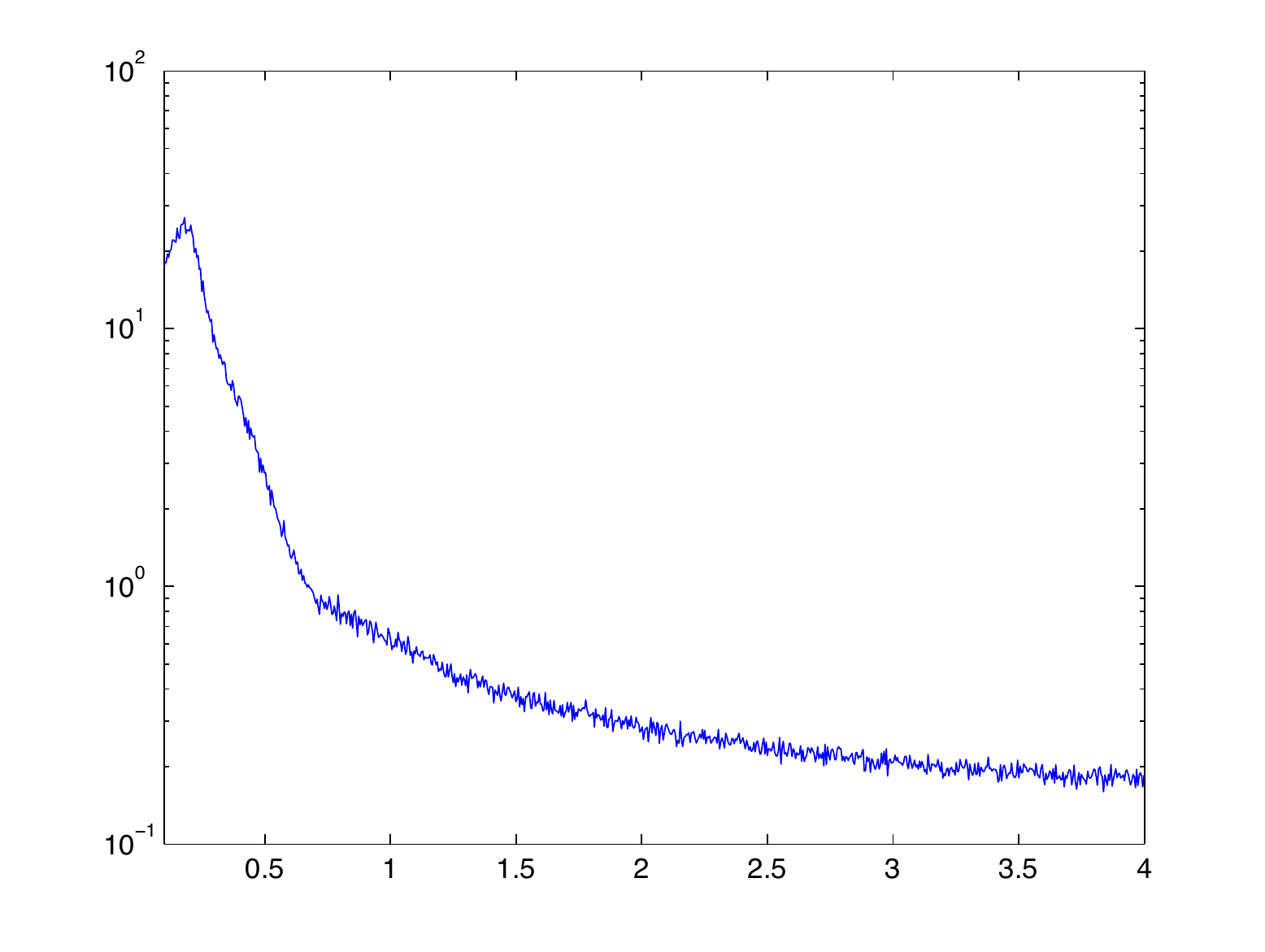}}\\
		\subfigure[$i_1=-3.246$, $N=1000$: Trajectory]{\includegraphics[width=.45\textwidth]{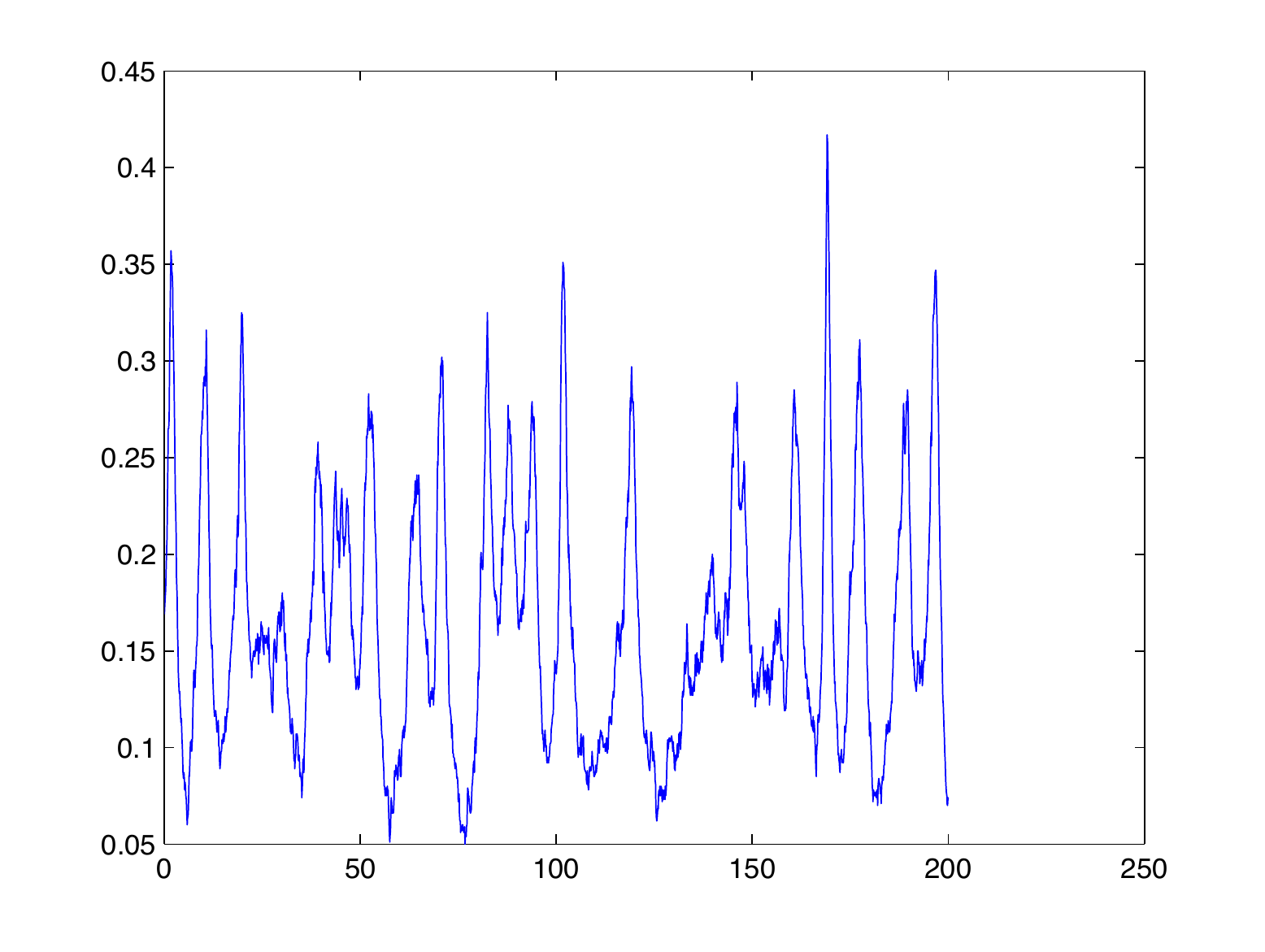}}\qquad
		\subfigure[$i_1=-3.244$: WC orbit]{\includegraphics[width=.45\textwidth]{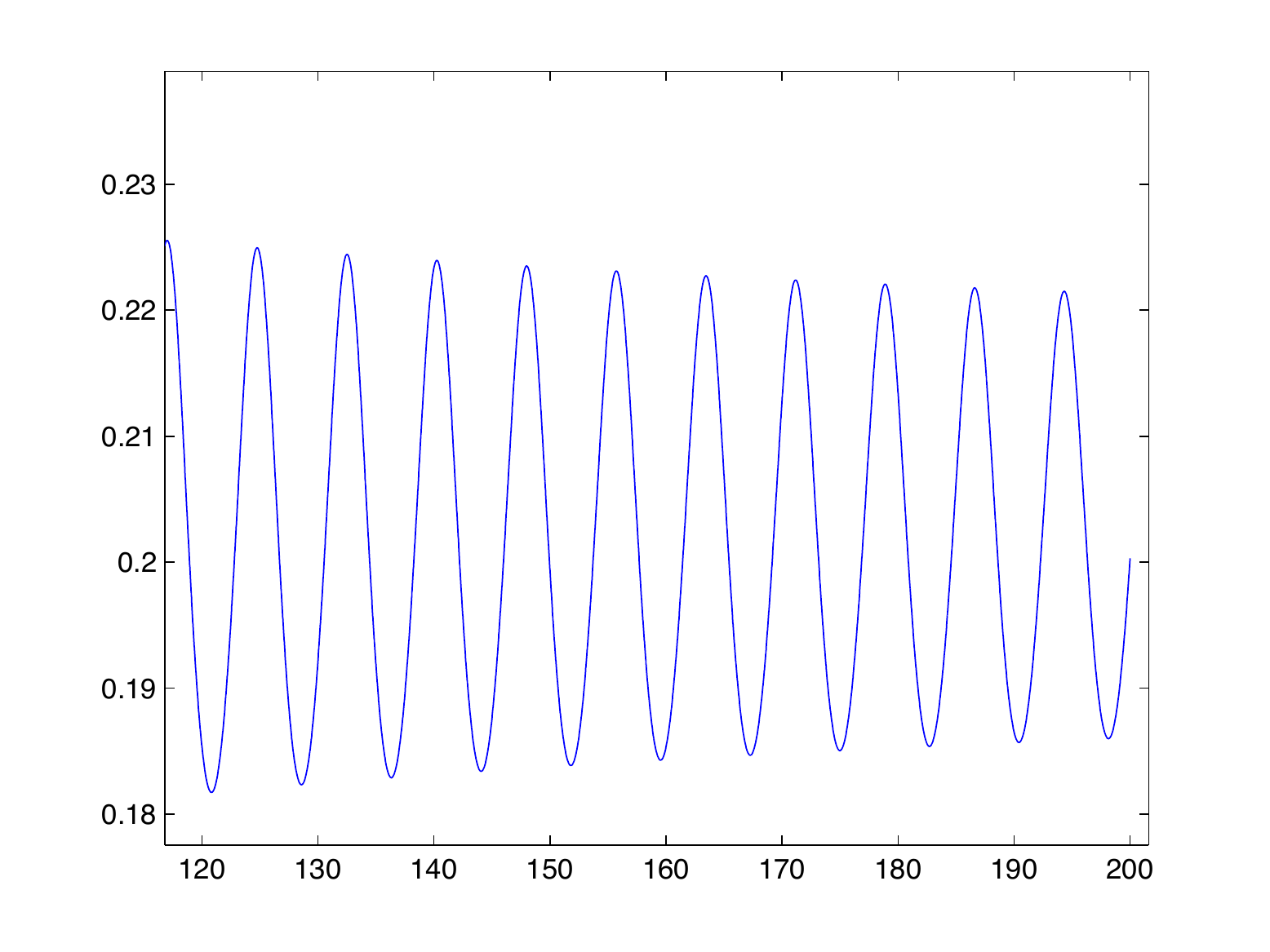}}
		\caption{Simulations of the Markovian system and periodic obits. Simulations are performed using Doob-Gillespie algorithm in order to simulate a large number of sample paths of the Markovian system, and the power spectrum presented are averaged across the realizations. For $I_2=-5$, orbits of the WC system exist for $I_1\in[-3.245,0.54]$, and around this zone, depending on the system size, cycles might exist or not (see text).}
	\label{fig:spectrum}
\end{figure}

\review{This particularly good correspondence between cycles of the moment equations and quasiperiodic orbits of the Markov chain is a very promising result. However, it is important to note that the additional cycles corresponding to the branches H2 and H4 do not appear in the Markov simulations. This might be linked with the stability of such orbits and the system size at which stable periodic orbit exist. This study therefore indicates that further exploration is needed to assess more clearly a definitive relationship between quasicycles and actual orbits of the moment equations, and there is a good chance that these two phenomena are related. This precise study is not within the scope of the present paper and we are currently actively working on this topic. Of high interest also is the fact that in some regions corresponding to the vicinity of a Hopf bifurcation, the moment equations present a low-dimensional chaotic behavior. Such phenomena, of interest in neuroscience, have, to the best of our knowledge, not been shown, and we believe that if these phenomena indeed appear in the Markovian system, this is of crucial interest in neuroscience and in the understanding of large-size stochastic systems.}

 \section{Discussion}

In this paper, we studied the different solutions of newly developed models for describing the mean-field limit of Markovian neural networks modeling large-scale cortical areas: BCC and Bressloff models\cite{buice-cowan:10,bressloff:09}. These equations {were originally} derived from a Markovian implementation of the firing of each neurons in the population, whose transition probability satisfied a differential equation, the master equation of the process\cite{elboustani-destexhe:09}. Both models studied drew on this framework {but, since they addressed different neuronal regimes, dealt with different variables and different scaling}, and after truncation of the moment hierarchy yielded sets two coupled equations, one governing the mean firing activity and the other governing the same-time correlations of the firing activity in the network (or the ordered cumulant that can be deduced from the correlation and the mean firing rate). The particular aim of both models was to recover, in the limit where the number of neurons is infinite, the standard WC equations. 
 
We {transformed} these models {in order to study the proportion of active neurons in each populations, in} networks composed of different populations of neurons with different dynamics and composed of different number of neurons. Both models include finite-size corrections, and the size of the network appears simply in the equations as a parameter of the model. The limit when the number of neurons tend to infinity appears in that view as a regular perturbation of the finite-size system {in our particular choice of variables, and that} both the BCC- and Bressloff-{derived} models converged towards the same equations as the number of neurons tends to infinity, the infinite-size system. 
 
We analyzed the solutions of the infinite-size system, and established the fact that all solutions of WC system were associated with solutions of the infinite-size system with null correlations. However, we showed that the stability of the solutions obtained in that fashion did not necessarily have the same stability properties as in WC system: fixed points of the Wilson and Cowan system correspond to fixed points of the infinite-size system with zero correlations having the same stability, but all the cycles of the WC system are either unstable or neutrally stable in the infinite-size system. This destabilization of cycles gives rise to new behaviors (fixed points or periodic orbits) with non-zero correlations, the correlation-induced behaviors.
 
 When studying the finite-size equations and the way it relates to the properties of the infinite-size system, we presented evidence that the mean-field limit appears to be a singular limit, in the sense that for any finite-size of the network, the system presents behaviors the infinite-size system does not feature. Moreover, singular points appear in the limit of infinitely many neurons, and for infinitely many neurons, branches of bifurcations collapse and disappear. Besides these qualitative distinctions between finite and infinite-size systems, we numerically found mesoscopic-scale effects, namely behaviors that are only present at a certain population scale, and disappear for smaller or larger populations. The scale at which these behaviors took place coincided with the scale of typical cortical columns. Such behaviors are neither captured by the study of small-networks nor by the mean-field limit, stressing the importance of precise descriptions of these intermediate scales in tractable approaches. {We want to emphasize that the system-size expansion and the truncations are not valid at the bifurcation points. However, the bifurcation analysis we performed allowed us to uncover the number and stability of fixed points and limit cycles corresponding to different parameters. The actual transition of the stochastic process is a very complex phenomenon which would be worth studying but that present very important technical intricacies. We also note that the underlying Markovian models are stochastic models. If the moment equations present multiple attractors, the stochastic system will randomly switch between the different attraction basins. A large-deviation study in the case where there is a small parameter (Bressloff case and BCC case in the fully connected system treated in the manuscript) can allow one to identify the attractor switching and the mean exit time of each attraction basin in order to more precisely study the behavior of the underlying stochastic system\review{, that could possibly be performed using Freidlin and Wentzell theory in the large $N$ case~\cite{freidlin-wentzell:98}}.} \review{As already mentioned, Paul Bressloff opened the way to such studies by using WKB approximations allowing to compute the transition rates between metastable attractors in a one-population model~\cite{bressloff:10}. Alternative approaches could make use of the stochastic dynamical systems theory~\cite{arnold:95,horsthemke-lefever:06}. These approaches would allow directly studying the bifurcations of the initial stochastic system in the Langevin approximation. }
 
 Correlation-induced and finite-size behaviors were then compared to the initial Markovian model. We observed that when WC system presented an attractive fixed-point behavior, each sample path of the Markov chain presented stochastic variations around these fixed points, the correlations tended towards zero and the mean coincided with WC fixed point. In that case, the network regime is therefore an asynchronous state, and is well approximated by Wilson and Cowan system. \review{The validity of the moment equations close to bifurcations was shown in the derivations to be mainly valid far away from bifurcations. However, we showed that the codimension two bifurcations of the moment equations was closely matched with similar phenomena (e.g. bistability) observed in simulations of the Markovian system}. \review{However, we also observed that both BCC and Bressloff models can present additional behaviors in these regimes that are not present in the simulations of the Markov chain. The moment expansion breaks down when considering periodic orbits of the WC system, corresponding to quasi-periodic stochastic solutions of the stochastic system, and different techniques such as the stochastic phase reduction, Langevin approximation and Floquet theory are very promising. However, we provided numerical evidences showing that actual cycles in the moment equations could be related to the phenomenon of quasicycles in the Markovian system. These numerical observations strongly suggest further exploration along these lines.  }

The precise study of all these points will further allow understanding the behavior of large scale systems, and will particularly be interesting in the study of rhythms in the brain, and possible synchronization or desynchronization as a function of the network size and of the connectivity. A particularly interesting point of the models studied in this paper is that they describe a stochastic process by a set of differential equations, where the size of the network appears as a perturbation parameter of a set of ordinary differential equations. This approach can be promising in order to study bifurcations of a stochastic processes and qualitatively analyzing its behaviors and the dependency of its solutions as a function of the parameters. We also noted that the behavior of each sample path of the process was possibly very different from the behavior of its first moments, which raises the question of the best variable to describe such stochastic processes. 
 
 The mathematical method developed to study the stability of the correlation equations, involving Kronecker products and sum and the application of Floquet theory is of potential importance in the study of networks with correlations, since the form of the correlation equation is quite standard in this kind of problems (see e.g. \cite{rodriguez-tuckwell:98} in the case of small noise expansion). This approach can also be applied to the study of the stability of solutions in networks involving Hebbian learning, since the connectivity weights evolution depends on the firing correlations and the stability of the weights involves very similar equations as the ones we studied here. 
 
 The need for  extensions of the present study  to more realistic Markovian models involving more relevant biologically realistic transitions is a direct consequence of the present work. Indeed, we have seen that the transitions chosen in the Markovian model are only one-step jumps. In biological terms, this assumption means that when $n_i(t)$ neurons are active in population at time $i$, at least $n_i(t)-1$ neurons are active at time $t+dt$, i.e. keep firing, which is not biologically realistic since after firing, a neuron immediately returns to quiescence. Therefore extending the framework to a more complex state space of the Markov chain would be of great interest for the development of these models towards more realistic biological grounds, and would probably allow a better assessment of the properties of cortical areas.   
 
 Finally, in order to go beyond the description of the moments of the Markov chain, we believe that one shall turn towards alternative ways to derive mean-field equations for such systems. A particularly interesting way to study such systems we are currently exploring builds upon recent methods developed for the study of Markov chains, mainly in the field of queueing theory, such as fluid limits and large deviations approaches. The fluid limit technique is based on a space-time proper rescaling depending on the process under consideration, that yields deterministic or stochastic limits of the activity of the network when the number of neurons tends to infinity. Large deviation techniques allow one to derive the finite-size corrections (see e.g. \cite{wainrib:2010} in a case of stochastic ion channels) and is a natural extension of our work.

 \noindent{\bf Acknowledgments:}
This work was supported by NSF DMS 0817131.
 
 \appendix
\section{An alternative derivation of the moment equations}\label{append:RodigTuck}
\review{In this appendix we show that the moment equation corresponding to the Markov chain governed by equations \eqref{eq:Master} can be derived from a Rodriguez-Tuckwell moment expansion on the Langevin approximation of the Markov chain. Following Kurtz approach in~\cite{kurtz:76}, we know that the dynamics of the our rescaled variables $p_i=n_i/N_i$, which is a Markov chain in the initial setting, approaches as $N$ goes to infinity a continuous diffusion process $(X_t)_{t\geq 0}$ satisfying the equation:
\begin{equation}\label{eq:Langevin}
	dX_t^i=\Big(-\alpha_i\,X^i_t + f_i(s_i(t,X_t))\Big)\;dt + \frac 1 N \sqrt{\alpha_i\,X^i_t + f_i(s_i(t,X_t))}\;dW^i_t
\end{equation}
where $S_i(t,X_t)=\sum_{j=1}^N w_{ij}X^j_t + I_i(t)$.}

\review{Following the works of Rodriguez and Tuckwell~\cite{rodriguez-tuckwell:96,rodriguez-tuckwell:98}, we derive from this equation the dynamical system governing the approximate moments of $X$. Denoting by $m_j$ the mean value of $X_j$ and by $C_{ij}$ the correlation between $X_i$ and $X_j$, a direct application of Rodriguez and Tuckwell formula applied to our particular form of dynamical system yields:
\begin{equation*}
	\begin{cases}
		\der{m_j}{t}&=-\alpha_j m_j +f_j(s_j) + \frac 1 2 f_j''(s_j)\sum_{l=1}^N \sum_{p=1}^N w_{jp}w_{jl} C_{lp}\\
		\der{C_{ij}}{t}&=-(\alpha_i+\alpha_j)C_{ij} + f_i'(s_i)\sum_{l=1}^N w_{il} C_{lj} + f_j'(s_j)\sum_{l=1}^N w_{jl} C_{li} \\
		& \quad + \delta_{ij} \Big[\alpha_i m_i +f_i(s_i) + \frac 1 {N_i N_j} f_i''(s_i)\sum_{k,l}w_{ik}w_{il} C_{kl}\Big]
	\end{cases}
\end{equation*}
Where $s_i=\sum_{j=1}^N w_{ij}m^j_t + I_i(t)$.  These equations therefore appear as a perturbation of the BCC equations with an additional nonlinear term in the dynamics or order two. Truncation at order $1$ in the small parameter $1/N$ yields exactly BCC equations. Moreover, it is interesting to note that these equations again correspond in the limit $N\to \infty$, to the infinite model described in section \ref{section:Infinite}. }

\section{Kronecker Algebra: Some useful properties}\label{append:Kronecker}
In this appendix, we review and prove some useful properties of Kronecker products of matrixes. We recall the definition of the the function $\Vect$ transforming a $M \times N$ matrix into a $M\,N$-dimensional column vector, as defined in  \cite{neudecker:69}:
 \[\Vect: \begin{cases}
 \R^{M\times N} &\mapsto \R^{M\,N}\\
 X & \mapsto [X_{11}, \ldots, X_{M1}, X_{12}(t), \ldots, X_{M2}(t), \ldots X_{1N}(t), \ldots, X_{MN}(t)]^T
 \end{cases}\]
 Let us now denote by $\otimes$ the Kronecker product defined for $A\in \R^{m\times n}$ and $B \in \R^{r\times s}$ as the $(m\,r) \times (n\,s)$ matrix:
 \[A\otimes B := 
 \left(
 \begin{array}{cccc}
 	a_{11} B & a_{12} B & \cdots & a_{1n}B\\
 	a_{21} B & a_{22} B & \cdots & a_{2n}B\\
 	\vdots & \vdots & \ddots & \vdots\\
 	a_{m1} B & a_{m2} B & \cdots & a_{mn}B\\
 \end{array}
 \right)\]
 For standard definitions and identities in the field of Kronecker products, the reader is referred to \cite{brewer:78}. We recalled in the main text the following identities for $A,B,D,G,X \in \R^{M\times M}$, $I_M$ be the $M\times M$ identity matrix and $A\cdot B$ or $A\,B$ denote the standard matrix product:
 \begin{equation}\label{eq:Kronecker2}
 	\begin{cases}
 		\Vect (AXB) = (B^T \otimes A ) \Vect(X)\\
 		A\oplus A =A\otimes I_M + I_M \otimes A\\
 		(A \otimes B)\cdot (D \otimes G) = (A\cdot D) \times (B\cdot G)
 	\end{cases}.
 \end{equation}
 The relationship $\oplus$ is called Kronecker sum. 

\begin{proposition}\label{prop:KronEigen}
	Let $A$ and $B$ in $\R^{M\times M}$, and assume that $A$ has the eigenvalues $\{\lambda_i;\;i=1,\ldots,M\}$ and $B$ the eigenvalues $\{\mu_i;\;i=1,\ldots,M\}$. Then we have:
	\begin{itemize}
		\item $A\otimes B$ has the eigenvalues $\{\lambda_i\mu_j;\;(i,j)\in \{1,\ldots,M\}^2\}$.
		\item $A\oplus B$ has the eigenvalues $\{\lambda_i+\mu_j;\;(i,j)\in \{1,\ldots,M\}^2\}$.
	\end{itemize}
\end{proposition}

\begin{proof}
Let $\lambda_i$ (resp.$\mu_j$) be an eigenvalue of $A$ (resp. $B$) with eigenvector $u$ (resp. $v$), and define the matrix $z=u\cdot v^T$ (i.e. $z_{ij} = u_i \, v_j$). We have:
\begin{align*}
	\Big(A\oplus B \Big) \, \Vect (z) &= \Vect\Big(A \cdot u\cdot v^T + u\cdot v^T \cdot B^T\Big)\\
	&= \Vect\Big(\lambda_i u \cdot v^T + u \cdot (B\cdot v)^T\Big) \\
	&= \lambda_i \Vect(z) + \Vect\Big(u (\mu_j v^T)\Big)\\
	&= (\lambda_i + \mu_j) \Vect(z).
\end{align*}
which entails that $\Vect(z)$ is an eigenvector of $A(\nu)\oplus A(\nu)$ associated with the eigenvalue $\lambda_i + \lambda_j$. Similarly, we have for $z=v\cdot u^T$:
\begin{align*}
	\Big(A\otimes B \Big) \, \Vect (z) &= \Vect\Big(B \cdot v\cdot u^T A^T\Big)\\
	&= \Vect\Big((\mu_j v) \cdot (\lambda_i \,u^T) \Big) \\
	&= \lambda_i\mu_h \Vect(z)
\end{align*}

The dimension of $A\oplus B$ and $A\otimes B$ is $M^2$ and there are exactly $M^2$ linearly independent matrices $z$ possible built in the proposed fashion, therefore we identified all possible eigenvalues.
\end{proof}

\begin{theorem}\label{theo:PsiKron}
	Let $\Phi(t)$ be the solution of the matrix differential equation:
	\begin{equation*}
		\begin{cases}
			\der{\Phi(t)}{t} &= A(t)\Phi(t)\\
			\Phi(0) &= I_M
		\end{cases}
	\end{equation*}
	and $\Psi(t)$ the solution of:
	\begin{equation*}
		\begin{cases}
			\der{\Psi(t)}{t} &= (A(t)\oplus A(t))\Psi(t)\\
			\Psi(0) &= I_{M^2}			
		\end{cases}
	\end{equation*}
	We have $\Psi(t)=\Phi(t)\otimes \Phi(t)$.
\end{theorem}

\begin{proof}
Indeed, we have, using the basic properties of the Kronecker product recalled in equation \eqref{eq:Kronecker}
\begin{align*}
	\der{\Phi(t)\otimes \Phi(t)}{t} &= \der{\Phi(t)}{t} \otimes \Phi(t) + \Phi(t) \otimes \der{\Phi(t)}{t}\\
	&= \big(A(\nu(t))\cdot \Phi(t)\big)\otimes \Phi(t) + \Phi(t) \otimes \big(A(\nu(t))\cdot \Phi(t)\big) \\
	&= \big(A(\nu(t))\cdot \Phi(t)\big)\otimes \big(I_M\cdot\Phi(t)\big) + \big(I_M\cdot\Phi(t)\big) \otimes \big(A(\nu(t))\cdot \Phi(t)\big)\\ 
	&= (A(\nu(t))\otimes I_M)\cdot (\Phi(t)\otimes \Phi(t)) + (I_M\otimes A(\nu(t))) \cdot (\Phi(t)\otimes \Phi(t))\\
	&= \big(A(\nu(t))\otimes I_M + I_M\otimes A(\nu(t))\big) \cdot \big(\Phi(t)\otimes\Phi(t)\big)\\
	&= \big(A(\nu(t))\oplus A(\nu(t))\big) \cdot \big(\Phi(t)\otimes\Phi(t)\big)
\end{align*}
Therefore $\Phi(t)\otimes \Phi(t)$ satisfies the same differential equation as $\Psi(t)$ and moreover, $\Phi(0)\otimes \Phi(0)=I_M\otimes I_M=I_{M^2}$, and therefore by existence and uniqueness of the resolvent, $\Psi(t)=\Phi(t)\otimes \Phi(t)$.
\end{proof}

\section{Genericity of the Hopf bifurcation found}\label{append:Calculations}
 In this appendix we derive the expression of the first Lyapunov exponent of the bifurcation, which proves that the existence of the Hopf bifurcation exhibited in section \ref{ssect:OnePopHopf}. In that section, we derived the expression of the Jacobian matrix at the considered fixed point:
 \[J=\left (\begin{array}{cc}
 	-\alpha + w\, f'(s) & \frac 1 2 \,f''(s)\,w^2\\
 	2f'(s)\frac w N & 2(-\alpha+w\, f'(s))
 \end{array}\right)\]
 At the bifurcation point, we have $-\alpha^*+w\,f'(s)=0$, and therefore at this point the Jacobian matrix reads:
 \[J_0=\left (\begin{array}{cc}
 	0 & \frac 1 2 \,f''(I)\,w^2\\
 	2f'(I)\frac w N & 0
 \end{array}\right)\]
 The eigenvalues of this matrix under the assumptions of section \ref{ssect:OnePopHopf} are $\pm i \omega_0$ where $\omega_0=\sqrt{-f'(I)f''(I)w^3/N}$. We define $q$ the right eigenvector of $J_0$ associated with $i \omega_0$:
 \[q=\Big(-\frac{i}{\sqrt{2\,f'(I)\,w/N}} \qquad 1\Big)^T\]
 and $p$ the right eigenvector of $J_0^T$ associated with the eigenvalue $i\omega_0$:
 \[p=\Big(\frac{i}{\sqrt{-\frac 1 2 f''(I)\,w^2}} \qquad 1\Big)^T\]
 For the sake of simplicity, we also name the components of the vector field of the system:
 \[\begin{cases}
 	f_1(\binom{\nu}{C}, \alpha^*) &= w\,f'(I)\,nu + f(w\nu+I)+\frac 1 2 f''(w\nu+I)\,w^2\,C\\
 	f_2(\binom{\nu}{C}, \alpha^*) &= -2\,w\,f'(I)\,C + 2\,f'(w\nu+I)\,w\,\big(C+\frac{\nu}{N}\big)\\ 
 \end{cases}.\]
 
 Following \cite{kuznetsov:98}, we define $B(\binom{x_1}{y_1}, \binom{x_2}{y_2})$ and $C(\binom{x_1}{y_1}, \binom{x_2}{y_2},\binom{x_3}{y_3})$ the second and third derivatives of the vector field, which are bi- and tri-linear forms. We have the following expressions for these multilinear functions:
 
 \[\begin{cases}
 	B_1(\binom{x_1}{y_1}, \binom{x_2}{y_2}) &= w^2\,f''(I)\,x_1\,x_2 + \frac 1 2\,f''(I)\,w^2\,(x_1\,y_2 +y_1\,x_2)\\
 	B_2(\binom{x_1}{y_1}, \binom{x_2}{y_2}) &= 2\,f''(I)\,w^2\,(x_1\,y_2+y_1\,x_2) + \frac 4 N\,f''(I)\,w^2\,x_1 \, x_2\\
 	C_1(\binom{x_1}{y_1}, \binom{x_2}{y_2}, \binom{x_3}{y_3}) &= f^{(3)}(I)\,w^3\,x_1\,x_2\,x_3 + \frac 1 2 \, f^{(4)}(I)\,w^4\,(x_1\,x_2\,y_3+x_1\,y_2\,x_3+y_1\,x_2\,x_3)\\
 	C_2(\binom{x_1}{y_1}, \binom{x_2}{y_2}, \binom{x_3}{y_3}) &= \frac 6 N\,f^{(3)}(I)\, w^3 \, x_1\,x_2\,x_3 + 2\,f^{(3)}(I)\,w^3 \,(x_1\,x_2\,y_3+x_1\,y_2\,x_3+y_1\,x_2\,x_3)
 \end{cases}\]
 We are now in position to compute the first Lyapunov exponent $l_1(0)$ using the formula:
 \begin{align*}
 	l_1(0)&=\frac 1 {2\omega_0} \textrm{Re}\Big ( \langle p, C(q,q,\overline{q})\rangle - 2 \langle p, B(q, J_0^{-1} B(q, \overline{q}))\rangle + \langle p, B(\overline{q}, (2\,i\,\omega_0 Id - J_0)^{-1}B(q,q))\rangle \Big)\\
 	&= \frac 1 {2\omega_0} (\mathcal{A}-2\mathcal{B}+\mathcal{C})
 \end{align*}
 where $\langle x, y \rangle$ denotes the complex inner product $\overline{x}^T\cdot y$ and the sum of three terms denoted $\mathcal{A}$, $\mathcal{B}$ and $\mathcal{C}$ are the real parts of the terms involved in the expression of the Lyapunov exponent. After straightforward but tedious calculations (that can be conveniently performed using a formal calculation tool such as Maple), we obtain:
 \[
 \begin{cases}
 	\mathcal{A} &= \frac{w^2 \, N}{f'(I)^2} f^{(3)}(I)\,f'(I)\\
 	\mathcal{B} &= \frac{w^2 \, N}{f'(I)^2} \left( -\frac{f''(I)^2}{\omega_0} - \frac 1 2 \frac{f^{(3)}(If'(I))}{\omega_0} + 2 f''(I)^2\right)\\
 	\mathcal{C} &= -\frac{w^2\, N }{f'(I)^2}\;\frac{2 f''(I)^2}{3}
 \end{cases}
 \]
 which yields the expression for the Lyapunov exponent:
 \begin{align*}
 	l_1(0) &= \frac{1}{2\,\omega_0} \left [ \mathcal{A} -2\,\mathcal{B}+\mathcal{C}\right ]\\
 	& = \frac{w^2\,N}{2\,\omega_0\,f'(I)^2} \left [ f^{(3)}(I)f'(I) \left(1+\frac {1}{\omega_0}\right) + f''(I)^2\left(\frac{2}{\omega_0} - \frac{14}{3}\right)\right]
 \end{align*}
 
 \section{Finite-size effects in BCC two-populations Model I}\label{appen:2PopsBCFinite}
 In this appendix, we study the two-populations BCC system corresponding to Model I and show that the finite-size effects are closely related to what is observed in Bressloff model as studied in section \ref{sec:TwoPopsFinite}. BCC finite-size equations read:
 \begin{equation}\label{eq:BuiceCowan2Pops}
 	\begin{cases}
 		a_1'&=-\alpha a_1+f(s_1)+\frac{1}{2}\,f''(s_1)\,(w_{11}^2\,c_{11}+w_{12}^2\,c_{22}+2\,w_{12}\,w_{11}\,c_{12})\\
 		a_2'&=-\alpha a_2+f(s_2)+\frac{1}{2}\,f''(s_2)\,(w_{22}^2\,c_{22}+w_{21}^2\,c_{11}+2\,w_{22}\,w_{21}\,c_{12})\\
 		c_{11}'&=-2\alpha c_{11}+2\,f'(s_{1})\,(w_{11}\,c_{11}+w_{12}\,c_{12})+2\,f'(s_1)\,a_1\,w_{11}\,n_e\\
 		c_{22}'&=-2\alpha c_{22}+2\,f'(s_{2})\,(w_{21}\,c_{12}+w_{22}\,c_{22})+2\,n_i\,f'(s_2)\,w_{22}\,a_2\\
 		c_{12}'&=-2\alpha c_{12}+f'(s_{1})\,(w_{11}\,c_{12}+w_{12}\,c_{22}+a_2\,w_{12}\,n_i) \ldots \\
 		& \qquad + f'(s_2)\,(w_{21}\,c_{11}+w_{22}\,c_{12}+w_{21}\,a_{1}\,n_e)\\
 	\end{cases}
 \end{equation}
 where we denoted:
 \[
 \begin{cases}
 	s_1 &= w_{11}\,a_1+w_{12}\,a_2+i_1\\
 	s_2 &= w_{21}\,a_2+w_{22}\,a_2+i_2
 \end{cases}
 \]
 Similarly to Bressloff case, BCC model features two families of limit cycles. One of these branches corresponds exactly to the branch of limit cycles of WC system starting from a Hopf bifurcation and disappearing through a homoclinic bifurcations. Two additional Hopf bifurcations appear related to the family of periodic orbit corresponding to the correlation-induced cycle evidence in the analysis of the infinite-size system. This branch of limit cycles exist whatever $n$, and loses stability through a Neimark-Sacker bifurcation as the number of neurons increases. This bifurcation generates chaos for large enough networks, which clearly does not exists in WC system. 
 \begin{figure}[!h]
 	\centering
 		\subfigure[Wilson and Cowan ]
 		{\includegraphics[width=.45\textwidth]{WC2PopsArranged.pdf}}
 		\subfigure[BCC ]
 		{\includegraphics[width=.45\textwidth]{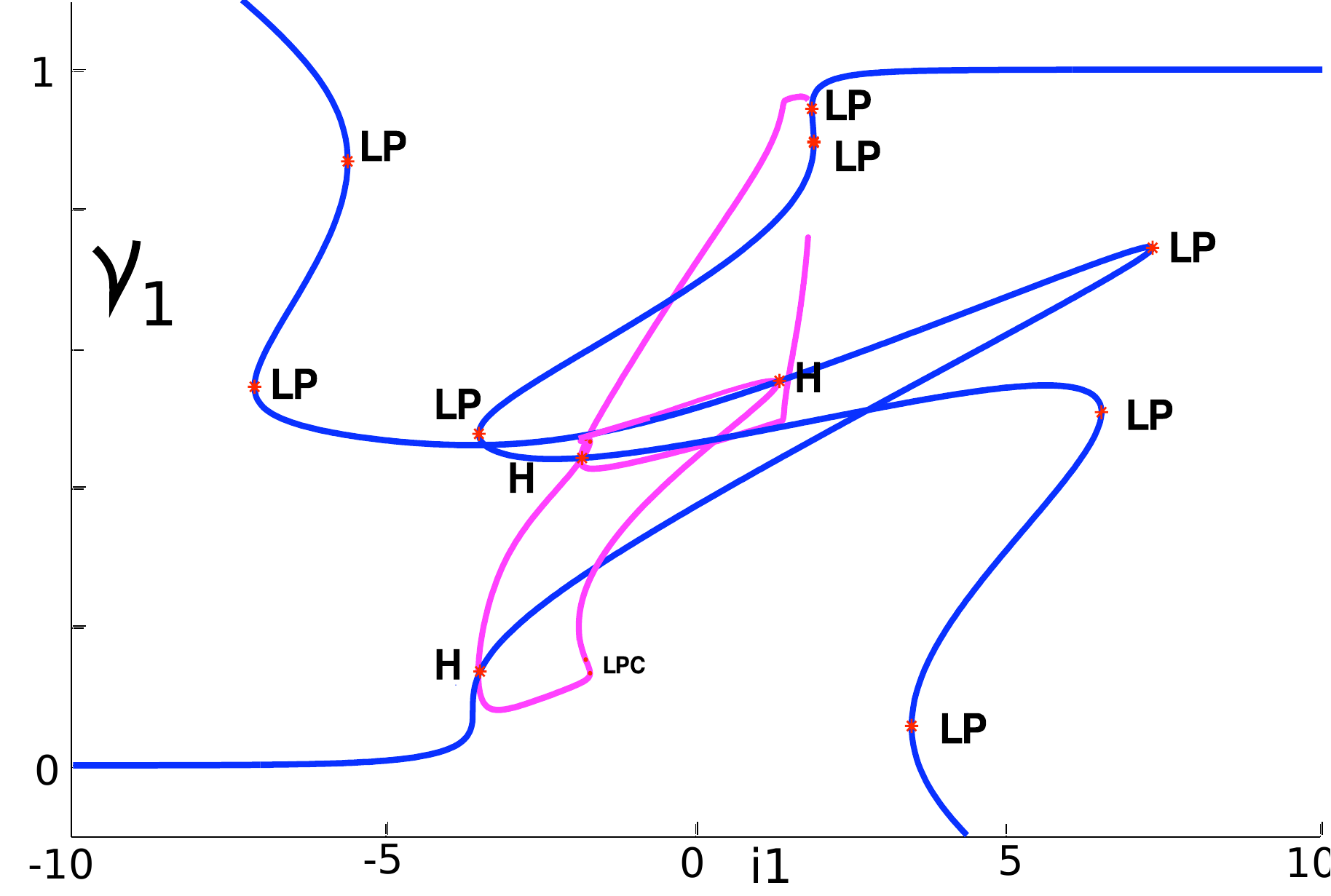}}
 	\caption{Bifurcation diagram for the BCC system. Blue lines represent the equilibria, pink lines the extremal values of the cycles in the system. Bifurcations of equilibria are denoted with a red star, LP represents a saddle-node bifurcation (Limit Point), H a Hopf bifurcation. The four Hopf bifurcations share two families of limit cycles. The branch corresponding to the smaller values of $i_1$ undergoes two fold of limit cycles, and the other branch of limit cycle a Neimark Sacker (Torus) bifurcation. }
 	\label{fig:BCBifs0.2}
 \end{figure}
 
 This family of limit cycles presents stability for networks containing between 100 and 17.000 neurons, corresponding to a finite size effect that appears only in the region of interest of the cortical columns. As the number of neurons increase, the system keep the same number and type of bifurcations, and the infinite model appears to be a singular case where different bifurcations meet (see Fig.~\ref{fig:BCCod2}), and very large networks lose the property of presenting a stable cycle, and present a chaotic behavior.
 
 \begin{figure}[!h]
 	\centering
 		\subfigure[Codimension 2 bifurcations]{\includegraphics[width=.80\textwidth]{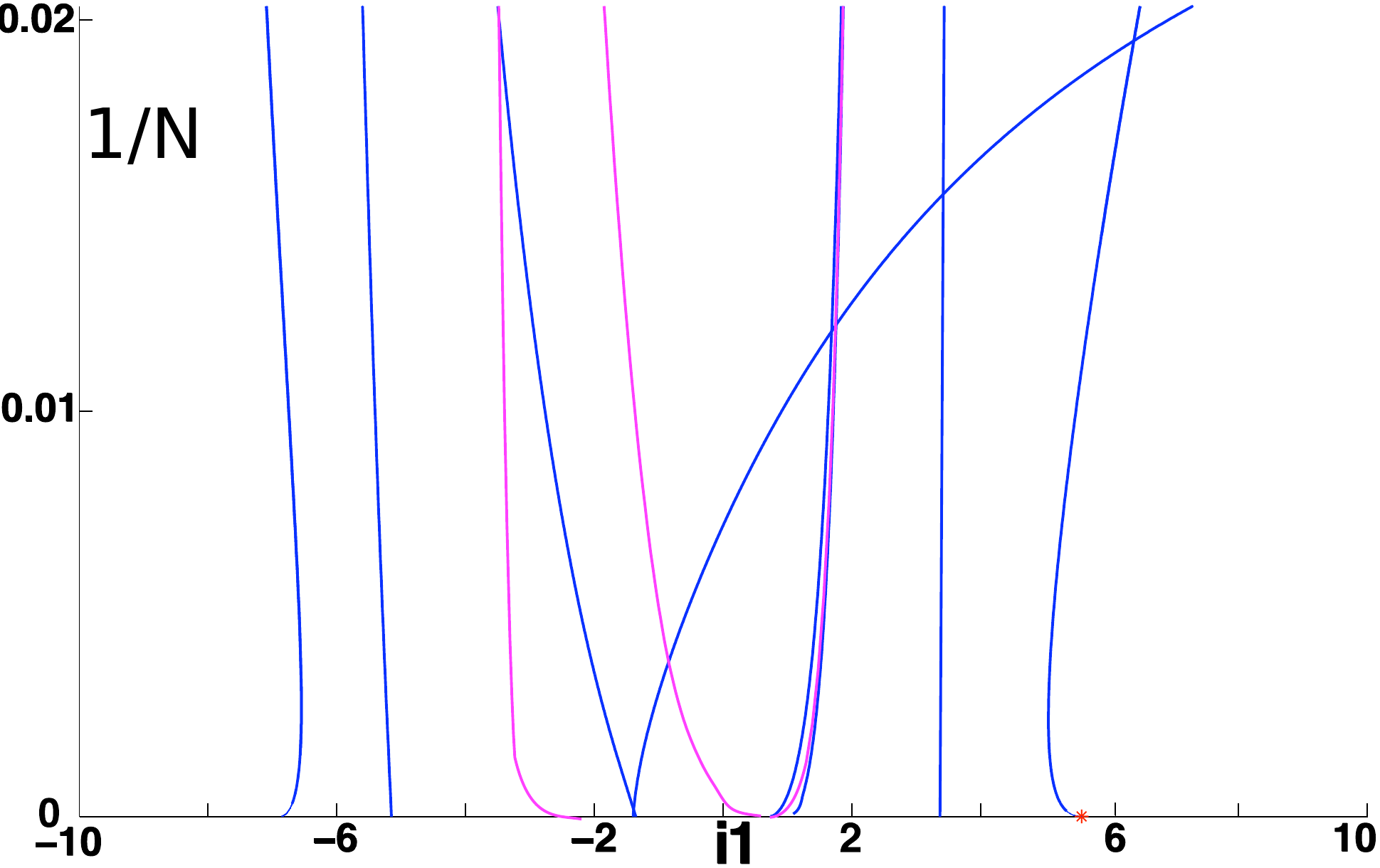}}\\
 		\subfigure[The stable cycle yielding chaos]{\includegraphics[width=.45\textwidth]{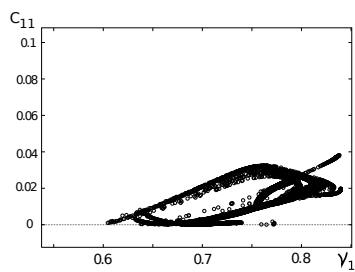}}
 	\caption{Codimension two bifurcations and chaotic behavior as the number of neuron increases.}
 	\label{fig:BCCod2}
 \end{figure}
 
\bibliographystyle{siam}

\begin{thebibliography}{10}

\bibitem{abbott-van-vreeswijk:93}
{\sc L.F Abbott and C.A. Van~Vreeswijk}, {\em Asynchronous states in networks
  of pulse-coupled neuron}, Phys. Rev, 48 (1993), pp.~1483--1490.

\bibitem{amari:72}
{\sc S.~Amari}, {\em Characteristics of random nets of analog neuron-like
  elements}, Syst. Man Cybernet. SMC-2,  (1972).

\bibitem{amari:77}
{\sc S.-I. Amari}, {\em Dynamics of pattern formation in lateral-inhibition
  type neural fields}, Biological Cybernetics, 27 (1977), pp.~77--87.

\bibitem{amit-brunel:97}
{\sc D.J. Amit and N.~Brunel}, {\em Model of global spontaneous activity and
  local structured delay activity during delay periods in the cerebral cortex},
  Cerebral Cortex, 7 (1997), pp.~237--252.

\bibitem{arnold} Arnold, V.I., {\em Ordinary Differential Equations}, MIT Press, 1981, chapt.5.

\bibitem{arnold:95}
{\sc L.~Arnold}, {\em {Random dynamical systems}}, Springer, 1995.


\bibitem{beggs-plenz:04}
{\sc John~M Beggs and Dietmar Plenz}, {\em Neuronal avalanches are diverse and
  precise activity patterns that are stable for many hours in cortical slice
  cultures.}, J Neurosci, 24 (2004), pp.~5216--5229.

\bibitem{benayoun-cowan:10}
{\sc Marc Benayoun, Jack~D. Cowan, Wim van Drongelen, and Edward Wallace}, {\em
  Avalanches in a stochastic model of spiking neurons}, PLoS Comput Biol, 6
  (2010/07/08), pp.~e1000846--.

\bibitem{bolland-galla-etal:08}
{\sc R. P. Boland, T. Galla and A. J. McKane}, {\em How limit cycles and quasi-cycles are related in systems with intrinsic noise}, J. Stat. Mech, 9 (2008), P09001.

\bibitem{bressloff:09}
{\sc Paul Bressfloff}, {\em Stochastic neural field theory and the system-size
  expansion}, SIAM J. Appl. Math., 70 (2010),
  pp.~1488--1521.

\bibitem{bressloff:10}
{\sc Paul Bressfloff}, {\em Metastable states and quasicycles in a stochastic Wilson--Cowan model}, Phys. Rev. E, 82 (2010),
  051903.

\bibitem{brewer:78}
{\sc John Brewer}, {\em {Kronecker Products and Matrix Calculus in System
  Theory}}, IEEE Trans. Circuits Syst,, CAS-25 (1978).

\bibitem{brunel:00}
{\sc N.~Brunel}, {\em Dynamics of sparsely connected networks of excitatory and
  inhibitory spiking neurons}, Journal of Computational Neuroscience, 8 (2000),
  pp.~183--208.

\bibitem{brunel-hakim:99}
{\sc N.~Brunel and V.~Hakim}, {\em Fast global oscillations in networks of
  integrate-and-fire neurons with low firing rates}, Neural Computation, 11
  (1999), pp.~1621--1671.

\bibitem{brunel-latham:03}
{\sc N~Brunel and P~Latham}, {\em Firing rate of noisy quadratic
  integrate-and-fire neurons}, Neural Computation, 15 (2003), pp.~2281--2306.

\bibitem{buice-cowan:07}
{\sc M.A. Buice and J.D. Cowan}, {\em Field theoretic approach to fluctuation effects for neural networks}, Physical Review E, 75:051919 (2007).

\bibitem{buice-cowan:10}
{\sc M.A. Buice, J.D. Cowan, and C.C. Chow}, {\em Systematic fluctuation
  expansion for neural network activity equations}, Neural Computation,  22 (2010), pp.~377-426.

\bibitem{cai-tao-etal:04}
{\sc D~Cai, L~Tao, M~Shelley, and DW~McLaughlin}, {\em An effective kinetic
  representation of fluctuation-driven neuronal networks with application to
  simple and complex cells in visual cortex}, Proceedings of the National
  Academy of Sciences, 101 (2004), pp.~7757--7762.

\bibitem{coombes-owen:05}
{\sc S.~Coombes and M.~R. Owen}, {\em Bumps, breathers, and waves in a neural
  network with spike frequency adaptation}, Phys. Rev. Lett., 94 (2005).


\bibitem{doob:45}
{\sc JL~Doob}, {\em Markoff chains--denumerable case}, Transactions of the
  American Mathematical Society, 58 (1945), pp.~455--473.

\bibitem{dykman-mori-etal:94}
{\sc M.I.~Dykman, E.~Mori, J.~Ross and P.M.~Hunt}, {\em Large fluctuations and optimal paths in chemical kinetics}, Journal of Chemical Physics, 100 (1994), pp.~5735--5750.

\bibitem{elboustani-destexhe:09}
{\sc S~El~Boustani and A~Destexhe}, {\em A master equation formalism for
  macroscopic modeling of asynchronous irregular activity states}, Neural
  computation, 21 (2009), pp.~46--100.

\bibitem{elboustani-destexhe:10b}
{\sc S~El~Boustani and A~Destexhe}, {\em Brain dynamics at multiple scales: can one reconcile the apparent low-dimensional chaos of macroscopic variables with the seemingly stochastic behavior of single neurons?}, International Journal of Bifurcation and Chaos, 20 (2010), pp.~1--16.

\bibitem{ermentrout:98}
{\sc Bard Ermentrout}, {\em Neural networks as spatio-temporal pattern-forming
  systems}, Reports on Progress in Physics, 61 (1998), pp.~353--430.

\bibitem{ermentrout:02}
{\sc B.~Ermentrout}, {\em {Simulating, Analyzing, and Animating Dynamical
  Systems: A Guide to XPPAUT for Researchers and Students}}, Society for
  Industrial Mathematics, 2002.

\bibitem{ermentrout-cowan:79}
{\sc G~B Ermentrout and J~D Cowan}, {\em A mathematical theory of visual
  hallucination patterns.}, Biol Cybern, 34 (1979), pp.~137--150.

\bibitem{faugeras-touboul-etal:09}
{\sc O.~Faugeras, J.~Touboul, and B.~Cessac}, {\em A constructive mean field
  analysis of multi population neural networks with random synaptic weights and
  stochastic inputs}, Frontiers in Neuroscience, 3 (2009).

\bibitem{freidlin-wentzell:98}
{\sc M.I.~Freidlin and A.D.~Wentzell}, {\em {Random perturbations of dynamical systems}}, Springer Verlag, 1998.

\bibitem{gaspard:02}
{\sc P.~Gaspard}, {\em Correlation time of mesoscopic chemical clocks}, 
  J. Chem. Phys., 117 (2002), 8905-8916.

\bibitem{gillespie:76}
{\sc DT~Gillespie}, {\em A general method for numerically simulating the
  stochastic time evolution of coupled chemical reactions}, Journal of
  computational physics, 22 (1976), pp.~403--434.

\bibitem{gillespie:77}
\leavevmode\vrule height 2pt depth -1.6pt width 23pt, {\em Exact stochastic
  simulation of coupled chemical reactions}, The journal of physical chemistry,
  81 (1977), pp.~2340--2361.

\bibitem{guckenheimer-holmes:83}
{\sc J.~Guckenheimer and P.~J. Holmes}, {\em Nonlinear Oscillations, Dynamical
  Systems and Bifurcations of Vector Fields}, vol.~42 of Applied mathematical
  sciences, Springer, 1983.

\bibitem{horsthemke-lefever:06}
{\sc W.~Horsthemke and R.~Lefever}, {\em {Noise-induced transitions}}, Springer, 2006.

\bibitem{kurtz:76}
{\sc T. G. Kurtz}, {\em Limit theorems and diffusion approximations for density dependent
Markov chains}, { Math.Prog. Stud.} 5:67 (1976).

\bibitem{kuznetsov:98}
{\sc Yuri~A. Kuznetsov}, {\em Elements of Applied Bifurcation Theory}, Applied
  Mathematical Sciences, Springer, 2nd~ed., 1998.

\bibitem{laing-troy-etal:02}
{\sc C.L. Laing, W.C. Troy, B.~Gutkin, and G.B. Ermentrout}, {\em Multiple
  bumps in a neuronal model of working memory}, SIAM J. Appl. Math., 63 (2002),
  pp.~62--97.

\bibitem{levina-etal:09}
{\sc Anna Levina, J.~Michael Herrmann, and Theo Geisel}, {\em Phase transitions
  towards criticality in a neural system with adaptive interactions}, Phys.
  Rev. Lett., 102 (2009).

\bibitem{ly-tranchina:07}
{\sc Cheng Ly and Daniel Tranchina}, {\em Critical analysis of dimension
  reduction by a moment closure method in a population density approach to
  neural network modeling.}, Neural Computation, 19 (2007), pp.~2032--2092.

\bibitem{mckane-nagy-etal:07}
{\sc A. J. McKane, J. D. Nagy, T. J. Newman and M. O. Stefanini}, {\em {Amplified biochemical oscillations in cellular systems
}}, J.  Stat. Phys. 71, (2007), 165.


\bibitem{mattia-del-giudice:02}
{\sc M.~Mattia and P.~Del~Giudice}, {\em {Population dynamics of interacting
  spiking neurons}}, Physical Review E, 66 (2002), p.~51917.

\bibitem{neudecker:69}
{\sc H~Neudecker}, {\em Some theorems on matrix differentiation with special
  reference to kronecker matrix products}, Journal of the American Statistical
  Association, 64 (1969), pp.~953--963.

\bibitem{ohira-cowan:93}
{\sc Toru Ohira and Jack D. Cowan}, {\em Master-equation approach to stochastic neurodynamics}, Phys. Rev. E, 48:3 (1993), pp.~2259--2266.


\bibitem{plesser:99}
{\sc H.~E. Plesser}, {\em Aspects of signal processing in noisy neurons}, PhD
  thesis, Georg-August-Universit{\"a}t, 1999.

\bibitem{rodriguez-tuckwell:96}
{\sc R~Rodriguez and H~C Tuckwell}, {\em A Dynamical System for the Approximate Moments of Nonlinear Stochastic Models of Spiking Neurons and Networks.},
  Mathematical and Computer Modeling, 31 (1996), pp.~175--180.

\bibitem{rodriguez-tuckwell:98}
{\sc R~Rodriguez and H~C Tuckwell}, {\em Noisy spiking neurons and networks:
  useful approximations for firing probabilities and global behavior.},
  Biosystems, 48 (1998), pp.~187--194.

\bibitem{rolls-deco:10}
{\sc ET~Rolls and G~Deco}, {\em The noisy brain: stochastic dynamics as a
  principle of brain function}, Oxford university press, 2010.

\bibitem{softky-koch:93}
{\sc William~R. Softky and Christof Koch}, {\em The highly irregular firing of
  cortical cells is inconsistent with temporal integration of random epsps},
  Journal of Neuroscience, 13 (1993), pp.~334--350.

\bibitem{teramae-nakao-etal:09}
{\sc J.~Teramae, H.~Nakao and G.B.~Ermentrout}, {\em Stochastic phase reduction for a general class of noisy limit cycle oscillators}, Physical review letters, 102 (2009), 194102.

\bibitem{touboul-destexhe:10}
{\sc J.~Touboul and A.~Destexhe}, {\em Can power-law scaling and neuronal
  avalanches arise from stochastic dynamics?}, PLoS ONE, 5 (2010), p.~e8982.

\bibitem{touboul-faugeras:07b}
{\sc Jonathan Touboul and Olivier Faugeras}, {\em The spikes trains probability
  distributions: a stochastic calculus approach}, Journal of Physiology, Paris,
  101/1-3 (2007), pp.~78--98.

\bibitem{touboul-faugeras:08}
\leavevmode\vrule height 2pt depth -1.6pt width 23pt, {\em First hitting time
  of double integral processes to curved boundaries}, Advances in Applied
  Probability, 40 (2008), pp.~501--528.

\bibitem{wainrib:2010}
{\sc Gilles Wainrib}, {\em Randomness in Neurons: a multiscale probabilistic
  analysis}, PhD thesis, Ecole Polytechnique, 2010.

\bibitem{wilson-cowan:72}
{\sc H.R. Wilson and J.D. Cowan}, {\em Excitatory and inhibitory interactions
  in localized populations of model neurons}, Biophys. J., 12 (1972),
  pp.~1--24.

\bibitem{wilson-cowan:73}
{\sc H.R. Wilson and J.D. Cowan}, {\em A mathematical theory of the functional
  dynamics of cortical and thalamic nervous tissue}, Biological Cybernetics, 13
  (1973), pp.~55--80.

\end{thebibliography}

 \end{document}